\documentclass[11pt,a4paper]{article}
\usepackage{aligned-overset}
\usepackage{amsthm}
\usepackage{amsmath}
\usepackage{mathtools}
\usepackage{authblk}
\usepackage[makeroom]{cancel}
\usepackage[tablesonly,nolists]{endfloat}
\usepackage[hmargin={28mm,28mm},vmargin={30mm,35mm}]{geometry}
\usepackage[colorlinks,allcolors={blue}]{hyperref}
\usepackage{newtxtext,newtxmath}
\usepackage{nicefrac}
\usepackage{pgfplots}
\usepackage{tikz}
\usepackage{booktabs}
\usetikzlibrary{spy}
\usepackage{paralist}
\usepackage{rotating}
\usepackage{stmaryrd}
\usepackage{subcaption}
\usepackage[compact,small]{titlesec}
\usepackage{multirow} 

\usepackage[backend=bibtex,style=numeric-comp,maxbibnames=99,giveninits=true]{biblatex}
\bibliography{nstime_resemi_pr_hho}

\newcommand{\email}[1]{\href{mailto:#1}{#1}}

\newcommand{\bb}{\boldsymbol{b}}
\newcommand{\bc}{\boldsymbol{c}}
\newcommand{\be}{\boldsymbol{e}}
\newcommand{\bef}{\boldsymbol{f}}
\newcommand{\bg}{\boldsymbol{g}}
\newcommand{\bq}{\boldsymbol{q}}

\newcommand{\bn}{\boldsymbol{n}}
\newcommand{\bu}{\boldsymbol{u}}
\newcommand{\bv}{\boldsymbol{v}}
\newcommand{\bw}{\boldsymbol{w}}
\newcommand{\bx}{\boldsymbol{x}}
\newcommand{\by}{\boldsymbol{y}}
\newcommand{\bz}{\boldsymbol{z}}

\newcommand{\bG}{\boldsymbol{G}}
\newcommand{\bI}{\boldsymbol{I}}
\newcommand{\bL}{\boldsymbol{L}}
\newcommand{\bR}{\boldsymbol{R}}
\newcommand{\bU}{\boldsymbol{U}}

\newcommand{\bH}{\boldsymbol{H}}
\newcommand{\bW}{\boldsymbol{W}}

\newcommand{\bzero}{\mathbf{0}}

\newcommand{\bbeta}{{\boldsymbol \eta}}
\newcommand{\bdelta}{\boldsymbol{\delta}}
\newcommand{\bGamma}{{\boldsymbol \Gamma}}

\newcommand{\bvs}{\boldsymbol{\mathfrak{v}}}
\newcommand{\bws}{\boldsymbol{\mathfrak{w}}}

\newcommand{\qs}{{\mathfrak{\phi}}}
\newcommand{\ps}{{\mathfrak{\psi}}}
\newcommand{\ts}{\boldsymbol{\mathfrak{\theta}}}
\newcommand{\xis}{\boldsymbol{\mathfrak{\xi}}}

\newcommand{\Poly}[1]{\mathcal{P}^{#1}}
\newcommand{\Polyd}[1]{\boldsymbol{\mathcal{P}}^{#1}}

\newcommand{\Goly}[1]{\boldsymbol{\mathcal{G}}^{#1}}

\newcommand{\tR}{\text{R}}
\DeclareMathOperator{\Id}{\mathrm{Id}}

\newcommand{\RTN}[1]{\boldsymbol{\mathcal{RT}}^{#1}}

\newcommand{\Mh}[1][h]{\mathcal{M}_{#1}}
\newcommand{\Th}[1][h]{\mathcal{T}_{#1}}
\newcommand{\Fh}[1][h]{\mathcal{F}_{#1}}
\newcommand{\Fhi}[1][h]{\mathcal{F}_{#1}^{{\rm i}}}
\newcommand{\Fhb}[1][h]{\mathcal{F}_{#1}^{{\rm b}}}

\newcommand{\Mhs}[1][h]{\mathfrak{M}_{#1}}
\newcommand{\Ths}[1][h]{\mathfrak{T}_{#1}}
\newcommand{\TTs}{\mathfrak{T}_T}
\newcommand{\Fhs}[1][h]{\mathfrak{F}_{#1}}
\newcommand{\Fhsi}[1][h]{\mathfrak{F}_{#1}^{{\rm i}}}

\newcommand{\normal}{\bn}

\newcommand{\uline}[1]{\underline{#1}}
\newcommand{\Real}{\mathbb R}

\newcommand{\GRAD}{\nabla}
\newcommand{\DIV}{\nabla\cdot}

\newcommand{\LAP}{\Delta}

\newcommand{\Ldeux}[1][\Omega]{{L}^2({#1})}
\newcommand{\Ldeuxd}[1][\Omega]{{\bL}^2({#1})}

\newcommand{\Linftyd}[1][\Omega]{{\bL}^{\infty}({#1})}

\newcommand{\Hun}[1][\Omega]{{H}^1({#1})}

\newcommand{\Hund}[1][\Omega]{{\bH}^1({#1})}
\newcommand{\Hunzd}[1][\Omega]{{\bH}^1_0({#1})}
\newcommand{\Hdiv}[1][\Omega]{{\boldsymbol{H}}_{\text{div}}({#1})}

\newcommand{\hhointerph}[1]{\hat{\uline{#1}}_h}
\newcommand{\hhointerpT}[1]{\hat{\uline{#1}}_T}
\newcommand{\hhointerpTT}[1]{\hat{{#1}}_T}
\newcommand{\norm}[2]{\|#2\|_{#1}}
\newcommand{\seminorm}[2]{|#2|_{#1}}

\newcommand{\calF}{{\mathcal F}}
\newcommand{\calL}{{\mathcal L}}
\newcommand{\calT}{{\mathcal T}}

\newcommand{\frakI}{{\mathfrak I}}
\newcommand{\frakT}{{\mathfrak T}}
\newcommand{\frakN}{{\mathfrak N}}
\newcommand{\frakR}{{\mathfrak R}}

\newcommand{\RecGT}[1]{{\frakR^{#1}_{\Goly{},\TTs}}}
\newcommand{\GammaGT}[1]{{\boldsymbol{\Gamma}_{\Goly{},\TTs}^{#1}}}
\newcommand{\GammacGT}[1]{{\boldsymbol{\Gamma}_{\Goly{},\TTs}^{{\rm c},#1}}}

\newcommand{\potOp}[1]{{\varrho_{\TTs}^{#1}}}

\newcommand{\bpi}{{\boldsymbol \pi}}
\newcommand{\btau}{{\boldsymbol \tau}}

\newcommand{\bphi}{{\boldsymbol \phi}}

\DeclareMathOperator{\card}{card}
\newcommand*\xbar[1]{%
  \hbox{%
    \vbox{%
      \hrule height 0.75pt % The actual bar
      \kern0.5ex%         % Distance between bar and symbol
      \hbox{%
        \kern-0.1em%      % Shortening on the left side
        \ensuremath{#1}%
        \kern-0.1em%      % Shortening on the right side
      }%
    }%
  }%
}

\newcommand{\tF}{t_{\rm F}}

\usepackage[normalem]{ulem}
\newcounter{corr}
\definecolor{violet}{rgb}{0.580,0.,0.827}
\newcommand{\corr}[3]{\typeout{Warning : a correction remains in page
    \thepage}
				\stepcounter{corr}        
				{\color{blue}\ifmmode\text{\,\protect\sout{\ensuremath{#1}}\,}\else\protect\sout{#1}\fi}
        {\color{red}#2}
        {\color{violet} #3}}

\newtheorem{theorem}{Theorem}

\newtheorem{lemma}[theorem]{Lemma}
\newtheorem{corollary}[theorem]{Corollary}
\theoremstyle{remark}
\newtheorem{remark}[theorem]{Remark}
\theoremstyle{definition}

\newcommand{\figpath}{./figures}

\begin{document}
\title{A Reynolds-semi-robust and pressure robust Hybrid High-Order method for the time dependent incompressible Navier--Stokes equations on general meshes}

\author[1]{Daniel Castanon Quiroz\footnote{\email{daniel.castanon@iimas.unam.mx}}}
\author[2]{Daniele A. Di Pietro\footnote{\email{daniele.di-pietro@umontpellier.fr}}}

\affil[1]{Instituto de Investigaciones en Matemáticas Aplicadas y en Sistemas, Universidad Nacional Autónoma de México, Circuito Escolar s/n, Ciudad Universitaria C.P. 04510 Cd. Mx. (México)
}
\affil[2]{IMAG, Univ Montpellier, CNRS, Montpellier 34090, France}
\maketitle

\begin{abstract}
	In this work we develop and analyze a Reynolds-semi-robust and pressure-robust Hybrid High-Order (HHO) discretization of the incompressible Navier--Stokes equations.
  Reynolds-semi-robustness refers to the fact that, under suitable regularity assumptions, the right-hand side of the velocity error estimate does not depend on the inverse of the viscosity.
  This property is obtained here through a penalty term which involves a subtle projection of the convective term on a subgrid space constructed element by element.
  The estimated convergence order for the $L^\infty(L^2)$- and $L^2(\text{energy})$-norm of the velocity is $h^{k+\frac12}$, which matches the best results for continuous and discontinuous Galerkin methods and corresponds to the one expected for HHO methods in convection-dominated regimes.
  Two-dimensional numerical results on a variety of polygonal meshes complete the exposition.
  \bigskip\\
  \textbf{Key words:} Hybrid High-Order methods, time-dependent incompressible flow, general meshes, Re-semi-robust error estimates, pressure robustness
  \medskip\\
  \textbf{MSC 2010:} 65N08, 65N30, 65N12, 35Q30, 76D05

\end{abstract}

%------------------------------------------------------------------------------%

%% \tableofcontents

%------------------------------------------------------------------------------%

\section{Introduction}

We consider the time-dependent incompressible  Navier--Stokes equations:
\begin{subequations}
  \label{eq:nstokes:strong}
  \begin{alignat}{2}
    \partial_t \bu - \nu \LAP \bu + (\bu \cdot \nabla )  \bu + \GRAD p  &= \bef &\quad& \text{in } (0,\tF]\times\Omega ,
      \label{eq:nstokes:strong:momentum}\\
      \DIV \bu&=0 &\quad& \text{in } (0,\tF]\times\Omega,\\
        \label{eq:nstokes:strong:mass}
        \bu & = \bzero &\quad& \text{on } [0,\tF]\times\partial \Omega, \\
        \bu(0,\cdot) &= \bu_0(\cdot)  &\quad& \text{in } \Omega,
        %%  \bu(0,\bx) &= \bu_0(\bx)  &\quad& \text{for } \bx \in \Omega,
  \end{alignat}
\end{subequations}
where 
$\Omega \subset \mathbb{R}^d$, for $d\in\{2,3\}$, denotes an open, bounded, simply connected polyhedral domain with Lipschitz boundary $\partial \Omega$ and $\tF > 0$ denotes the final time.
In addition,
$\bu : \lbrack 0, \tF\rbrack\times\Omega\rightarrow \mathbb{R}^d$ 
is the fluid velocity field, $p :(0, \tF\rbrack\times\Omega\rightarrow \mathbb{R}$ is the (zero-average) kinematic pressure,
$\bu_0 : \Omega\rightarrow \mathbb{R}^d$ represents a given initial condition,
and $\bef:(0, \tF\rbrack\times\Omega\rightarrow \mathbb{R}^d$ represents a given body force. The fluid is assumed to be Newtonian with constant kinematic  viscosity $\nu>0$.

In this work we propose a ``Reynolds-semi robust'' and  pressure-robust Hybrid-High Order (HHO) method on general meshes to approximate the solution
of the weak form of problem \eqref{eq:nstokes:strong}.
A numerical scheme is considered ``Reynolds-semi robust'' \cite{Schroeder.ea:2018} if its velocity error estimates are independent of the Reynolds  number (or $\nu^{-1}$). 
On the other hand, pressure robustness means that the velocity error estimates are independent of the pressure, a property whose relevance
was first emphasized in \cite{Linke:14,Volker.ea:2017}. Recently, a detailed analysis of vortex-dominated flows has been carried out in \cite{Gauger.ea:2019}, where the authors conclude that pressure robustness is a crucial prerequisite for the accurate  discretisation of non-trivial Navier--Stokes flows; numerical tests where pressure robust methods significantly outperform non-pressure robust methods for transient incompressible flows at high Reynolds numbers are provided, focusing on simplicial meshes.

Reynolds-semi robust numerical schemes have recently made the object of several works.
To obtain this kind of estimates, it is standard to assume additional regularity condition on $\nabla \bu$. For instance, in the analysis of a continuous interior penalty  finite element method \cite{Burman2007ContinuousIP}, assuming
 $\bu\in L^\infty(0,\tF;\bW^{1,\infty}(\Omega))$
a velocity error estimate in the
$L^\infty(0,\tF;\bL^2(\Omega))$-norm was obtained.
Other Reynolds-semi robust numerical schemes using simplicial meshes
have been proposed \cite{Arndt.ea:2015,Beirao-da-Veiga.ea.SUPG:2023,DallmannArndt:2016,DeFrutos.ea.2018,DeFrutos.ea:2018.2,garciaarchilla2024pressure,HanHou:2021}. In particular, we refer to  \cite{Schroeder.ea:2018} for an outstanding review and further insight into the importance of Reynolds-semi robust  numerical methods. At the time of the writing of this work, the best known velocity error estimate   in the ${L^\infty(0,\tF;\bL^2(\Omega))}$-norm is of order $h^{k+\frac{1}{2}}$ (where $k$ denotes the order of the polynomial approximation and $h$ is the mesh size), for instance  see the works of \cite{Beirao-da-Veiga.ea.SUPG:2023,garciaarchilla2024pressure,HanHou:2021}, where Reynolds-semi robust numerical methods are proposed which additionally satisfy the  pressure robustness  property.

In recent years, the mathematical community has become interested in developing numerical schemes that can make use of
general polygonal and polyhedral meshes, as opposed to more standard triangular/quadrilateral (tetrahedral/hexahedral) meshes.
A representative but by far non exhaustive sample concerning incompressible flow problems includes \cite{Di-Pietro.Ern:10,Di-Pietro.Ern:12,Di-Pietro.Krell:18,Gatica.Munar.ea:18,Botti.Di-Pietro.ea:19,Beirao-da-Veiga.Lovadina.ea:18,Beirao-da-Veiga.Dassi.ea:20,Zhang.Zhao.ea:21};
see also \cite{Botti.Castanon-Quiroz.ea:21,Castanon-Quiroz.Di-Pietro.ea:21} concerning non-Newtonian fluids and \cite{antonietti2022virtual} concerning the coupling with the heat equation.
Regarding pressure robust methods on general meshes for the Stokes and Navier--Stokes equations, some work has recently been 
done using the Virtual Element method, generalized barycentric
coordinates, the staggered Discontinuous Garlekin method, HHO methods, and Discrete de Rham  methods; see, e.g., \cite{Castanon-Quiroz.Di-Pietro:23, Liu.Harper.ea:20,Wang.Mu.ea:21,Zhao.Park.ea:22,Frerichs.Merdon:20, Kim.Zhao.ea:21,Botti.Massa:2022,Di-Pietro.Droniou:21, Beirao-da-Veiga.Dassi.ea:22,Di-Pietro.Droniou.ea:24}. Pressure-robust HHO methods for the Stokes and Navier–Stokes problem have been proposed in \cite{Di-Pietro.Ern.ea:16,Castanon-Quiroz.Di-Pietro:20,Castanon-Quiroz.Di-Pietro:23}; see also \cite{Botti.Botti.ea:24} for variants with hybrid pressure. 
To the authors' knowledge, the development of Reynolds-semi robust and pressure-robust numerical schemes  on general meshes has not yet been adressed. 
The present work fills this gap.
To achieve this, the proposed method uses a  slight variation of the divergence-preserving velocity reconstruction proposed by the same authors in \cite{Castanon-Quiroz.Di-Pietro:23}.
Contrary to \cite{Castanon-Quiroz.Di-Pietro:23}, however, we use here the convective form of the nonlinear term in the equation \eqref{eq:nstokes:strong:momentum} in the spirit of \cite{HanHou:2021,Schroeder.ea:2018}, where {$\boldsymbol{H}_{\text{div}}$-conforming} spaces are used.
Reynolds-semi robustness is obtained through a
new  term which penalizes the  jumps of a potential operator within a decomposition of some discrete piecewise polinomial spaces (cf. \eqref{eq:kz.decomp} and \eqref{def:golyk} below).
With these ingredients, we prove a velocity error estimate in the ${L^\infty(0,\tF;\bL^2(\Omega))}$-norm  of order $h^{k+\frac{1}{2}}$, which equals the best known velocity error estimate on simplicial meshes and corresponds to the typical order of convergence of HHO methods in convectoin-dominated regimes \cite{Di-Pietro.Droniou.ea:15}.

The rest of the paper is organised as follows.
In Section \ref{subsec:weak.NST} we introduce some notations and present the weak form of problem \eqref{eq:nstokes:strong}.
In Section \ref{sec:setting} we present the discrete setting.
In Section \ref{sec:hho-scheme}, which contains the statement of the discrete problem, particular emphasis is put on a novel scalar potential reconstruction that is used to stabilize the convective term and achieve Reynolds semi-robustness.
Section \ref{sec:convergence.analysis.vel} contains the velocity error analysis and the main results.
Finally, in Section \ref{sec:ntest1} we present numerical experiments to verify our theoretical results.

\section{Continuous setting}\label{subsec:weak.NST}

Throughout the paper, given an open bounded set $D\subset\mathbb{R}^d$ as well as two integers $m \geq 0$ and $p \geq 1$, we use the Sobolev space $W^{m,p}(D)$ for scalar-valued functions with associated norm $\norm{W^{m,p}(D)}{{\cdot}}$ and seminorm $\seminorm{W^{m,p}(D)}{{\cdot}}$.
Spaces of vector- and tensor-valued functions are indicated with bold letters.
In the case $m = 0$, we obtain the Lebesgue space $L^p(D) \coloneq W^{0,p}(D)$ and, when
$p = 2$, the Hilbert space $H^m(D) \coloneq W^{m,2}(D)$.
Additionally, the closed subspaces $H_0^1(D)$ consisting of $H^1 (D)$-functions with vanishing trace on $\partial D$,
and
the set $L_0^2 (D)$ of $L^2(D)$-functions with zero average in $D$ are used in what follows.
%The $L^2(D)$-inner product is denoted by $(\cdot, \cdot)_D$  and, if $D = \Omega$ , we usually omit the domain completely when no confusion can arise.
Given a Banach space $X$  and a real number $t > 0$,
we denote by $L^p(0,t;X)$, $p \in [1,\infty]$, the classical Bochner space.
In the case $t = \tF$, we often use the abbreviation
$L^p(X)\coloneq L^p(0,\tF;X)$.

Letting $\bU \coloneq \Hunzd$ and $P\coloneq L_0^2(\Omega)$, we consider the following weak form of problem \eqref{eq:nstokes:strong}:
Find  $\bu: [0,\tF] \rightarrow \bU$  and $p : (0,\tF\rbrack \to P$
with $\bu(0)=\bu_0\in \bU$, such that it holds, for all $(\bv,q) \in \bU\times P$ and almost every $t \in (0,\tF)$,
\begin{equation}
  \label{eq:nstokes:weak}
  ({d_t} \bu(t), \bv)
  + \nu a(\bu(t),\bv)
  + t(\bu(t),\bu(t),\bv)
  + b(\bv,p(t))
  - b(\bu(t),q)
  = \ell(\bef(t),\bv),
\end{equation}
with $(\cdot,\cdot)$ denoting the standard $\bL^2(\Omega)$-product, while the bilinear forms $a:\bU\times\bU\to\Real$, $b:\bU\times P\to\Real$, and $\ell:\Ldeuxd\times\bU\to\Real$ are defined by
\[ %% \begin{align}
a(\bw,\bv)\coloneq  \int_\Omega \nabla \bw :  \nabla \bv,
\quad b(\bv,q)\coloneq - \int_\Omega (\DIV \bv) q,
\quad \ell(\bef,\bv)\coloneq\int_\Omega\bef\cdot\bv,
%%   \quad
%%   \label{eq:nstokes:weak:bilinear.forms}
\] %% \end{align}
and the trilinear form $t:\bU\times\bU\times\bU\to\Real$ is such that
\[ %% \begin{equation}\label{eq:nstokes:weak:trilinear.form} 
t(\bw,\bv,\bz)\coloneq  \int_\Omega ((\bu \cdot \GRAD) \bv) \cdot \bz.
\] %% \end{equation}

%------------------------------------------------------------------------------%
%------------------------------------------------------------------------------%

\section{Discrete setting}\label{sec:setting}

\subsection{Mesh}\label{sec:setting:mesh}

In what follows, for the sake of simplicity, we will systematically use the term polyhedral instead of polygonal and face instead of edge also when $d=2$.

Following \cite[Definition 1.4]{Di-Pietro.Droniou:20}, we consider a polyhedral mesh defined as a couple $\Mh\coloneq(\Th,\Fh)$,
where $\Th$ is a finite collection of polyhedral elements, while $\Fh$ is a finite collection of planar faces.
We assume that every  element $T\in\Th$ is 
star-shaped with respect to a ball \cite[Remark 2.5]{Castanon-Quiroz.Di-Pietro:23}.
For any mesh element or face $X\in\Th\cup\Fh$, we denote by $|X|$ its Hausdorff measure and by $h_X$ its diameter, so that the meshsize satisfies $h = \max_{T\in\Th}h_T$.
Boundary faces lying on $\partial\Omega$ and internal faces contained in
$\Omega$ are collected in the sets $\Fhb$ and $\Fhi$, respectively.
For each mesh element $T\in\Th$, we denote by $\Fh[T]$ the set collecting the faces that lie on the boundary $\partial T$ of $T$ and, for all $F\in\Fh[T]$, we denote by $\normal_{TF}$ the (constant) unit vector normal to $F$ and pointing out of $T$.

It is assumed that $\Mh$ belongs to a regular mesh sequence $(\Mh)_h$ in the sense of \cite[Definition 1.9]{Di-Pietro.Droniou:20}.
This assumption entails the existence of a matching simplicial submesh $\Mhs\coloneq(\Ths,\Fhs)$ of $\Mh$ with the following properties:
$\Ths$ is a finite collection of simplicial elements;
for any simplex $\tau \in \Ths$, there is a unique mesh element $T \in \Th$ such that $\tau \subset T$;
for any simplicial face $\sigma \in \Fhs$ and any mesh face $F \in \Fh$ , either $\sigma \cap F = \emptyset $ or $\sigma \subset F$.
  We additionally assume that, for any element $T\in\Th$, its submesh is constructed in such way that all simplices contained in $T$ and collected in the set $\TTs$ (see Figure \ref{fig:simplices.faces.T.a})  have at least one common vertex  ${\bx}_T$.
This assumption is directly used in Lemmas \ref{lemma:RT.lifting} and \ref{lem:rhoOp} below.
Regarding the Lemma \ref{lemma:RT.lifting}, we refer to \cite[Remarks 2.1  and A.1]{Castanon-Quiroz.Di-Pietro:23} for further insight into this assumption.

We decompose the set of simplicial faces as $\Fhs= \Fhs^{\rm i} \cup \Fhs^{\rm b}$  where 
$\Fhs^{\rm i}$ and $\Fhs^{\rm b}$ respectivaly collect interior and boundary simplicial faces.
For any  $T\in \Th$, we define  $\Fhsi[T]$ as the set of simplicial faces of $\Fhs$ that lie in the interior of $T$.
For any face $F \in \Fh$ lying on the boundary of $T \in \Th$, $\Fhs[F]$ denotes the set of simplicial faces $\sigma$ for which $\sigma \subset F$, and we let $\bn_\sigma\coloneq\bn_{TF}$, and $\bn_{\tau\sigma}\coloneq\bn_{\sigma}$ for the unique element $\tau\in \TTs$, which contains $\sigma$. 
Additional notations for mesh elements and faces are introduced at the beginning of Section \ref{sec:discrete.problem:convective.term} below and illustrated in Figure \ref{fig:simplices.faces.T.b}.
For future use, we notice that, by \cite[Lemma 1.12]{Di-Pietro.Droniou:20}, mesh regularity implies the existence of an integer $N\geq 0$ depending only on the mesh regularity parameter such that
\begin{align} \max_h \max_{T\in\Th} \card(\TTs) \leq N \qquad \text{and}\qquad \max_h \max_{T\in\Th} \card(\Fh[T]) \leq N.  \label{ineq:card.IT.F} \end{align} \begin{figure}[ht]
  \begin{minipage}[t]{.5\textwidth}
    \centering
    % include first image
    \def\svgwidth{.9\columnwidth}
    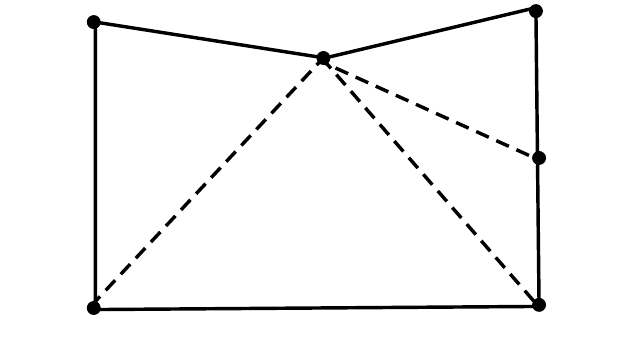
    \subcaption{The elements of $\TTs$ and $\Fh[T]$.}
    \label{fig:simplices.faces.T.a}
  \end{minipage}
  \begin{minipage}[t]{.5\textwidth}
    \centering
    % include second image
    \def\svgwidth{0.94\columnwidth}
    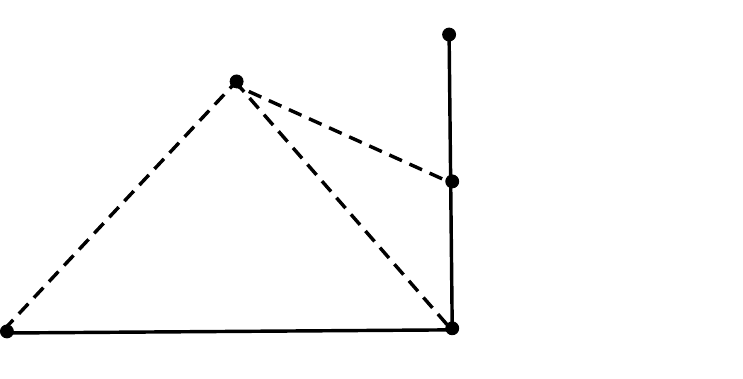
    \subcaption{A closer look to the right part: The simplicial faces $\sigma_1,\sigma_3,\sigma_5$ belong to the set of interior faces $\Fhsi[T]$ and we have $\sigma_2=F_1,\sigma_4=F_2,\sigma_6=F_3$,  and $\sigma_7=F_4$.}
\label{fig:simplices.faces.T.b}
  \end{minipage}
  \caption{An illustration of the sets $\TTs, \Fh[T]$ and $\Fhsi[T]$ for a given element $T\in \Th$ in $\Real^2$.}
  \label{fig:simplices.faces.T}
\end{figure}

In order to prevent the proliferation of generic constants we write, whenever possible, $a\lesssim b$ in place of $a\le Cb$ with $C>0$ independent of $\nu,h$ and, for local inequalities, also on the mesh element or face.
The dependencies of the hidden constant will be further specified when relevant. Moreover, we write
$a  \simeq b$, when both $a \lesssim b$  and $b \lesssim a$ hold.

\subsection{Local and broken spaces and projectors}

Let $X$ denote a mesh element or face and, for an integer $l\geq 0$, denote by
$\Poly{l}(X)$  the space spanned by the restrictions to $X$
of polynomials in the space variables of total degree $\leq l$.
The $L^2$-orthogonal projector $\pi_X^l:L^1(X)\rightarrow \Poly{l}(X)$ is such that, for all $\zeta\in L^1(X)$,
\[
\int_X (\zeta-\pi_X^l \zeta)w=0 \qquad \forall w \in \Poly{l}(X).
\]
Vector and matrix versions of the $L^2$-orthogonal projector are obtained by applying $\pi_X^l$ component-wise, and are both denoted with the bold symbol $\bpi_X^l$ in what follows.
Optimal approximation properties for the $L^2$-orthogonal projector are proved in \cite[Appendix A.2]{Di-Pietro.Droniou:17}; see also \cite[Chapter 1]{Di-Pietro.Droniou:20}, where these estimates are extended to non-star shaped elements.
Specifically, let $s\in\{0,\dots,l+1\}$ and $r \in [1,+\infty]$.
Then, it holds, with hidden constant only depending on $l$, $s$, $r$, and the mesh regularity parameter:
For all $T\in {\calT_h}$, all $\zeta\in W^{s,r}(T)$, and all $m\in\{0,\dots,s\}$,
\begin{equation}
  |\zeta - \pi_T^l \zeta |_{W^{m,r}(T)} \lesssim h_T^{s-m}|\zeta|_{W^{s,r}(T)},\label{eq:l2proj:error:cell}
\end{equation}
and, if $s \geq 1$,
\begin{equation}
  h_T^{\frac{1}{r}}%
  \| \zeta - \pi_T^l \zeta \|_{L^r(\partial T)}
  \lesssim h_T^{s}|\zeta|_{W^{s,r}(T)}.\label{eq:l2proj:error:faces}
\end{equation}

At the global level, the space of broken polynomial functions on $\calT_h$ of total degree $\leq l$ is denoted by $\Poly{l}(\calT_h)$, and $\pi_h^l$ is the corresponding $L^2$-orthogonal projector such that, for all $\zeta \in L^1(\Omega)$, $(\pi_h^l \zeta)_{|T}\coloneq\pi_T^l \zeta_{|T}$ for all $T\in\calT_h$.
Regularity requirements in error estimates will be expressed  in terms of the broken Sobolev spaces $W^{s,r}(\calT_h)$, spanned by functions in $L^r(\Omega)$ the restriction of which to every $T\in\Th$ is in $W^{s,r}(T)$.
We additionally set, as usual, $H^s(\calT_h)\coloneq W^{s,2}(\calT_h)$.

\subsection{Discrete spaces and norms}\label{sec:discspaces}

Let a polynomial degree $k\geq 0$ be fixed and set
\[ %% \begin{equation}\label{def:l.k}
  {k}^\star \coloneq
  \begin{cases}
    k   & \text{if $k\in\{0,1\}$,}
    \\
    k+1 & \text{otherwise.}
  \end{cases}
  \] %% \end{equation}
We then define a variant of the usual HHO space  as follows
\begin{multline}\label{def:HHO.space}
  \uline{\bU}_h^k\coloneq
 \big\{
 \uline{\bv}_h=((\bv_T)_{T\in \calT_h},(\bv_F)_{F\in \calF_h}) :
 \\
  \mbox{$\bv_T \in \Polyd{{k}^\star}(T)$ for all $T \in \calT_h$ and
	$\bv_F \in \Polyd{k}(F)$ for all $F \in \calF_h$}
  \big\}.  
\end{multline}
The restrictions of $\uline{\bU}_h^k$ and $\uline{\bv}_h \in \uline{\bU}_h^k$ to a generic mesh element $T\in \calT_h$ are respectively denoted by $\uline{\bU}_T^k$ and $\uline{\bv}_T=(\bv_T,(\bv_F)_{F\in \calF_T})$.
The vector of polynomials corresponding to a smooth function over $\Omega$ is obtained via the global interpolation operator $\uline{\bI}_h^k: \Hund \rightarrow \uline{\bU}_h^k$ such that, for all $\bv \in \Hund$,
\[ %% \begin{equation}\label{eq:Ih}
  \uline{\bI}_h^k\bv \coloneq  ((\bpi_T^{k^\star}\bv_{|T})_{T\in\calT_h},(\bpi_F^k\bv_{|F})_{F\in \calF_h}).
\] %% \end{equation}
Its restriction to a generic mesh element $T \in \calT_h$, collecting the components on $T$ and its faces, is denoted by $\uline{\bI}_T^k$.
We furnish $\uline{\bU}_{h}^k$ with the discrete $H^1$-like seminorm such that, for all
$\uline{\bv}_h \in \uline{\bU}_h^k$,
\begin{equation} \label{eq:norm.1h}
  \|\underline{\bv}_h\|_{1,h}\coloneq\left(
  \sum_{T \in \calT_h} \|\underline{\bv}_T\|_{1,T}^2
  \right)^{\nicefrac12},
\end{equation}
where, for all $T\in\calT_h$,
\begin{equation}\label{eq:norm.1T}
  \|\underline{\bv}_T\|_{1,T}\coloneq\left( 
  \|\nabla {\bv}_T\|_{\bL^2(T)}^2 + |\uline{\bv}_T|_{1,\partial T}^2
  \right)^{\nicefrac12}
  \mbox{ with }
  |\uline{\bv}_T|_{1,\partial T}\coloneq\left(
	\sum_{F\in \calF_T}\! h_F^{-1}\| \bv_F-\bv_T\|^2_{\Ldeuxd[F]}
  \right)^{\nicefrac12}.
\end{equation}

The discrete spaces for the velocity and the pressure, respectively accounting for the wall boundary condition and the zero-average condition, are 
\begin{equation*}%% \label{eq:ns:vUhD:Ph}
  \uline{\bU}_{h,0}^k
  \coloneq
  \left\{
  \uline{\bv}_h = ((\bv_T)_{T\in\calT_h},(\bv_F)_{F\in\calF_h})\in  \uline{\bU}_h^k : \bv_F = \boldsymbol{0} \quad \forall F\in\calF_h^{\rm b}
  \right\},\qquad
  P_h^k \coloneq \Poly{k}({\cal {T}}_h)\cap P.
\end{equation*}
For all  $\uline{\bv}_h \in \uline{\bU}_h^k$, we denote by ${\bv}_h \in \Polyd{k^\star}(\calT_h)$ the vector-valued broken polynomial function obtained patching element-based unkowns, that is $(\bv_h)_{|T}\coloneq \bv_T$ for all $T \in \calT_h$.
The following discrete Sobolev embeddings in $\uline{\bU}_{h,0}^k$ have been proved in \cite[Proposition 5.4]{Di-Pietro.Droniou:17} for the standard HHO space (the modifications required to treat the variation considered here are straightforward):
For all $r\in[1,6]$ it holds, for all  $\uline{\bv}_h \in \uline{\bU}_{h,0}^k$,
\begin{align}
  \|{\bv}_h\|_{\bL^r(\Omega)}
  \lesssim
  \|\uline{\bv}_h\|_{1,h}.
  \label{eq:disc:sobembd}
\end{align}
where the hidden constant is independent of both $h$ and $\uline{\bv}_h$, but possibly depends on $\Omega$, $k$, $r$, and the mesh regularity parameter.
It follows from \eqref{eq:disc:sobembd} that the map $\| {\cdot} \|_{1,h}$ defines a norm on $\uline{\bU}_{h,0}^k$.
Classically, the corresponding dual norm of a linear form $\calL _h:\uline{\bU}_{h,0}^k\to\Real$ is given by 
\begin{equation}\label{eq:dual.norm}
  \|\mathcal{L}_h\|_{1,h,*}
  \coloneq\sup_{\uline{\bv}_h\in\uline{\bU}_{h,0}^k,\|\uline{\bv}_h\|_{1,h}=1}\left|
  \mathcal{L}_h(\uline{\bv}_h)
  \right|.
\end{equation}

%------------------------------------------------------------------------------%

\section{A pressure-robust and Reynolds semi-robust HHO scheme}\label{sec:hho-scheme}

\subsection{Divergence-preserving local velocity reconstruction}\label{sec:discrete.setting:velocity.reconstruction}

Following  \cite{Di-Pietro.Ern:15}, for any element $T \in \calT_h$ we define the discrete divergence operator $D_T^k:\uline{\bU}_T^k \rightarrow \Poly{k}(T)$ such that, for all $\uline{\bv}_T \in \uline{\bU}_T^k$ and all $q \in \Poly{k}(T)$,
\begin{align}\label{eq:hho:div:op}
  \int_T D_T^k\uline{\bv}_T\, q
  &= - \int_T \bv_T\cdot \GRAD q + \sum_{F\in \calF_T} \int_F (\bv_F \cdot \bn_{TF})\, q.
\end{align}
The operator $D_T^k$ satisfies the following crucial commutation property (see \cite[Eq. (8.21)]{Di-Pietro.Droniou:20}):
\begin{equation}
  D_T^k\uline{\bI}_T^k \bv  = \pi_T^k(\DIV \bv) \qquad \forall \bv \in \Hund[T].
  \label{eq:DT.commuting}
\end{equation}

To achieve pressure robustness, we proceed similarly to \cite[Section 2.4]{Castanon-Quiroz.Di-Pietro:23}, constructing divergence-preserving velocity test functions which are then used for the discretization of the body force and the nonlinear term.
Let an element $T \in \calT_h$ be fixed and, for $\tau \in \TTs$, denote by 
$\RTN{k}(\tau)$ the classical Raviart--Thomas--N\'ed\'elec space of degree $k$ on $\tau$ \cite{Raviart.Thomas:77,Nedelec:80}.
We recall that a function in $\RTN{k}(\tau)$ is uniquely determined by its polynomial moments of degree up to $(k-1)$ inside $\tau$ and the polynomial moments of degree $k$ of its normal components on the simplicial faces of $\tau$, collected in the set $\Fhs[\tau]$.
We introduce the Raviart--Thomas--N\'ed\'elec space of degree $k$ on the matching simplicial submesh $\TTs$ of $T$ defined as follows:
\begin{equation*}%% \label{eq:def:rtn:gbl}
  \RTN{k}(\TTs)\coloneq\left\{
  \bws \in \Hdiv[T] : \text{%
    $\bws_{|\tau}\in \RTN{k}(\tau)$ for all $\tau \in \TTs$
  }
  \right\},
\end{equation*}
where $\Hdiv[T]\coloneq\left\{\bws \in \Ldeuxd[T]:  \DIV \bws \in \Ldeux[T]\right\}$.

Recalling from Section \ref{sec:setting:mesh} that, for a given element $T\in\Th$,  we denote by $\bx_T$ the common vertex of all simplices in $\TTs$, we introduce the following spaces generated by the Koszul operator (see \cite[Section 7.2]{Arnold:18}):
\begin{equation}\label{def:golyck}
  \begin{alignedat}{2}
    \Goly{{\rm c},k}(T)\coloneq (\bx-\bx_T) \perp \Poly{k-1}(T),
    &\quad 
    \Goly{{\rm c},k}(\TTs)\coloneq (\bx-\bx_T) \perp \Poly{k-1}(\TTs)
    & \quad \text{for } k\geq1, d=2,\\
    \Goly{{\rm c},k}(T)\coloneq (\bx-\bx_T) \times \Polyd{k-1}(T),
    &\quad 
    \Goly{{\rm c},k}(\TTs)\coloneq (\bx-\bx_T) \times \Polyd{k-1}(\TTs)
    & \quad \text{for } k\geq1, d=3,
  \end{alignedat}
\end{equation}
where $\by\perp \alpha \coloneq\alpha[-y_2, y_1]^T\in \mathbb{R}^2$ for $\by \in \mathbb{R}^2$ and  a scalar $\alpha$, and  $\times$ is the usual cross product in $\mathbb{R}^3$.
In addition, we define $\Goly{{\rm c},-1}(T)\coloneq\Goly{{\rm c},0}(T)\coloneq\Goly{{\rm c},-1}(\TTs)\coloneq\Goly{{\rm c},0}(\TTs)\coloneq\{ \bzero \}$.
Observe that we have the following direct
decompositions (see \cite[Corollary 7.4]{Arnold:18}):
\begin{equation}\label{eq:kz.decomp}
  \Polyd{k}(T)= \Goly{k}(T) \oplus \Goly{{\rm c},k}(T)
  \qquad\text{and}\qquad
  \Polyd{k}( \TTs)=\Goly{k}(\TTs) \oplus \Goly{{\rm c},k}(\TTs),
\end{equation}
where $\Polyd{k}(\TTs)$ is the broken polynomial spaces of total degree $\le k$ on $\TTs$ and
\begin{equation}\label{def:golyk}
  \Goly{k}(T)\coloneq\GRAD \Poly{k+1}(T)
  \qquad\text{and}\qquad
  \Goly{k}( \TTs)\coloneq\GRAD \Poly{k+1}(\TTs),
\end{equation}
and the direct sums in \eqref{eq:kz.decomp} are not necessarily orthogonal.
Notice that, with a little abuse of notation, in the definition \eqref{def:golyk} of $\Goly{k}( \TTs )$ we have used the symbol $\GRAD$ for the piecewise gradient on $\TTs$.
This abuse of notation will be kept throughout the rest of the paper.

We denote the $L^2$-orthogonal projectors onto the spaces $\Goly{k}(T), \Goly{k}(\TTs),\Goly{{\rm c},k}(T)$, and $\Goly{{\rm c},k}(\TTs)$ by $\bpi^{k}_{\Goly{},T}, \bpi^{k}_{\Goly{},\TTs}, \bpi^{{\rm c},k}_{\Goly{},T}$ and $\bpi^{{\rm c},k}_{\Goly{}, \TTs}$, respectively.
Then the \emph{divergence-preserving velocity reconstruction} ${\bR}_{T}^k: \uline{\bU}_T^k \rightarrow \RTN{k}(\TTs)$ is defined, for all $\uline{\bv}_T \in \uline{\bU}_T^k$, as the first component of the solution of the following mixed problem: 
Find  $({\bR}_{T}^k\uline{\bv}_T,\ps,{\ts})\in \RTN{k}(\TTs)\times \Poly{k}(\TTs){\times\Goly{{\rm c},k-1}(\TTs)}$ such that
\begin{subequations}
  \label{eq:darcyT:weak}
  \begin{alignat}{2}
    ({\bR}_{T}^k\uline{\bv}_T)_{|\sigma} \cdot \bn_{\sigma}&= (\bv_F\cdot \bn_{TF})_{|\sigma}
    &\qquad&
    \forall \sigma  \in \Fhs[F],\,  \forall F \in \Fh[T],
    \label{eq:darcyT:weak:bd}\\
    \int_{T} (\DIV {\bR}_{T}^k\uline{\bv}_T)\, \qs &= \int_{T} (D_T^k\uline{\bv}_T)\, \qs
    &\qquad&
    \forall \qs \in \Poly{k}(\TTs),
    \label{eq:darcyT:weak:a}\\
          { \int_{T}  {\bR}_{T}^k\uline{\bv}_T \cdot \xis} &= 
          {  \int_{T}\bv_T\cdot \xis}
          &\qquad&
          {
            \forall \xis \in \Goly{{\rm c},k-1}(\TTs)},
          \label{eq:darcyT:weak:ab}\\
          \int_{T}  {\bR}_{T}^k\uline{\bv}_T \cdot \bws  +  \int_{T} (\DIV \bws )\ps 
              { + \int_{T}  {\bws} \cdot \ts}
              &= \int_{T} \bv_T \cdot \bws
              &\qquad&
              \forall \bws \in\RTN{k}_0(\TTs),
              \label{eq:darcyT:weak:b}
  \end{alignat}
\end{subequations}
with $\RTN{k}_0(\TTs)$ denoting the subspace of $\RTN{k}(\TTs)$ spanned by functions whose normal component vanishes on $\partial T$.

\begin{remark}[Comparison with the $\Hdiv$-conforming reconstruction of \cite{Castanon-Quiroz.Di-Pietro:23}]\label{rem:bRTk}
  The above definition of the operator ${\bR}_{T}^k$ is a variant of  the one originally proposed in \cite[Section 2.4]{Castanon-Quiroz.Di-Pietro:23}.
  Specifically, there are two differences:
  the first one consists in changing the last component of the trial and test spaces from  $\Goly{{\rm c},k-1}(T)$ to $\Goly{{\rm c},k-1}(\TTs)$;
  the second one is the use of the HHO space variant defined in \eqref{def:HHO.space}.
\end{remark}
\begin{lemma}[Properties of ${\bR}_{T}^k$]
  \label{lemm:rtn}
  For all $T \in \Th$, it holds:
  \begin{enumerate}[(i)]
  \item \emph{Well-posedness.}
    For a given $\uline{\bv}_T \in \uline{\bU}_T^k $, there exists a unique ${\bR}_{T}^k\uline{\bv}_T \in \RTN{k}(\TTs)$ that solves \eqref{eq:darcyT:weak} and satisfies
    \begin{align}
      \norm{\Ldeuxd[T]}{\bv_T - {\bR}_{T}^k\uline{\bv}_T  } \lesssim h_T   \norm{1, T}{\uline{\bv}_T}.
      \label{ineq:rtn:bound}
    \end{align}
  \item \emph{Approximation in $\bW^{s,p}$.} Let an integer $p\in[1,\infty]$ be given.
    Then, for all $s \in \{1,\dots,k+1\}$, $m \in \{0,1\}$, and all $\bv\in \bW^{s,p}(T)$, it holds
    \begin{align} 
      \seminorm{\bW^{m,p}(\TTs)}{\bv  - {\bR}_{T}^k(\uline{\bI}_T^k\bv)  } \lesssim h_T^{s-m}   \seminorm{\bW^{s,p}(T)}{\bv}.
      \label{ineq:rtn:approx.Wsp}
    \end{align}
  \item \emph{Consistency.} For all {$k \ge 1$} and all $\uline{\bv}_T \in \uline{\bU}_T^k$, it holds 
    \begin{align} 
      \bpi^{k-1}_T({\bR}_{T}^k\uline{\bv}_T)= \bpi^{k-1}_T(\bv_T).  \label{ineq:rtn:consis} 
    \end{align} 
  \end{enumerate}
\end{lemma}

The proof makes use of the following Lemma, whose proof is given in \cite[Appendix]{Castanon-Quiroz.Di-Pietro:23}

\begin{lemma}[Raviart--Thomas--Nédélec lifting of the projection on $\Goly{{\rm c},k-1}(\mathfrak{T}_T)$]\label{lemma:RT.lifting}
  Let $T\in \Th$ and a function $\bv\in {\Ldeuxd[T]}$ be given.
  Then, for any integer $k\geq0$, there exists $\widetilde{\bR}_T^k(\bv)\in\RTN{k}_0(\TTs)$ such that
  %%   \begin{subequations}
  %%     \label{eq:rtn:golyc}
  \[
  \begin{aligned} %% \label{eq:rtn:golyc:a}
    \bpi^{{\rm c},k-1}_{\Goly{},\TTs}\widetilde{\bR}_T^k(\bv)&=\bpi^{{\rm c},k-1}_{\Goly{},\TTs}\bv,
    \\ %% \label{eq:rtn:golyc:b}
    \DIV \widetilde{\bR}_T^k(\bv)&=0,
    \\ %% \label{eq:rtn:golyc:c}      
    \widetilde{\bR}_T^k(\bv) \cdot\bn_{\sigma} &= 0
    \qquad\forall \sigma\in \Fhs[T]^{\rm i},
    \\ %% \label{eq:rtn:golyc:d}
    \norm{\Ldeuxd[T]}{\widetilde{\bR}_T^k(\bv)}&\lesssim\norm{\Ldeuxd[T]}{\bv},   
  \end{aligned}
  \]
  %%   \end{subequations}
  where we remind the reader that $\Fhs[T]^{\rm i}$ is the set of  the interior faces of the submesh of $T$.
\end{lemma}

\begin{proof}[Proof of Lemma \ref{lemm:rtn}]
  \noindent\underline{(i) \emph{Existence, uniqueness, and boundedness}.}
  The proof of existence and uniqueness of $\bR_T^k\uline{\bv}_T$ are the same as in \cite[Lemma 2.i]{Castanon-Quiroz.Di-Pietro:23} using Lemma \ref{lemma:RT.lifting} above.
  We now prove the bound \eqref{ineq:rtn:bound}. Using   essentially the same steps as in \cite{Castanon-Quiroz.Di-Pietro:23}, it is inferred that
  \begin{align}
    \label{eq:bound.solddarcy}
    \norm{\Ldeuxd[T]}{{\bR}_{T}^k\uline{\bv}_T}
    &
    \lesssim
    \norm{\Ldeuxd[T]}{{\bv}_T}
    +
    h_T
    \|\uline{\bv}_T\|_{1,T}
    +
    \sum_{F\in \Fh[T]}h_F^{\frac{1}{2}} \norm{\Ldeuxd[F]}{\bv_F}.
  \end{align}
  Using a triangle inequality 
    after inserting $\pm(\bpi_T^k \bv_T - {\bR}_T^k(\uline{\bI}_T^k\bpi_T^k\bv_T) )$ into the norm in the left-hand side,
  we have
  \begin{equation}\label{eq:ineq:rtn:bound.proof.i}
    \begin{aligned}
      \norm{\Ldeuxd[T]}{\bv_T - {\bR}_{T}^k\uline{\bv}_T  }
      &\leq
      \norm{\Ldeuxd[T]}{\bv_T - \bpi_T^{k}\bv_T }
      + \norm{\Ldeuxd[T]}{\bpi_T^{k}\bv_T -{\bR}_{T}^k(\uline{\bI}_T^k\bpi_T^k\bv_T)}\\
      &\quad
      + \norm{\Ldeuxd[T]}{{\bR}_{T}^k(\uline{\bI}_T^k\bpi_T^k\bv_T) -{\bR}_{T}^k\uline{\bv}_T }
      \eqcolon
      \mathfrak{T}_1+\mathfrak{T}_2 + \mathfrak{T}_3.
    \end{aligned}
  \end{equation}
  To bound $\mathfrak{T}_1$, we use the approximation properties \eqref{eq:l2proj:error:cell} of $\bpi_{T}^k$ with $(l,m,r,s)=(k,0,2,1)$ to get
  \begin{equation}\label{eq:est.T1}
    \mathfrak{T}_1
    \lesssim h_T\seminorm{\Hund[T]}{ {\bv}_T}
    \overset{\eqref{eq:norm.1T}}\lesssim h_T \norm{1,T}{\uline{\bv}_T}.
  \end{equation}
  Letting ${\hat{\bvs}_T} \coloneq  {\bR}_{T}^k(\uline{\bI}_T^k\bpi_T^k\bv_T)$ and using the same arguments as in \cite{Castanon-Quiroz.Di-Pietro:23}, we obtain
  \begin{equation}\label{eq:est.T2}
    \mathfrak{T}_2=0.
  \end{equation}
  By linearity, $\mathfrak{T}_3=\norm{\Ldeuxd[T]}{{\bR}_{T}^k(\uline{\bI}_T^k\bpi_T^k\bv_T -\uline{\bv}_T) }$.
  Thus, using the bound  \eqref{eq:bound.solddarcy} with $\uline{\bv}_T$ replaced by $\uline{\bI}_T^k\bpi_T^k\bv_T -\uline{\bv}_T$, the fact that  ($\uline{\bI}_T^k\bpi_T^k\bv_T - \uline{\bv}_T)_T =(\bpi_T^k\bv_T - {\bv}_T)$ and $(\uline{\bI}_T^k\bpi_T^k\bv_T - \uline{\bv}_T)_F = (\bpi_T^k\bv_T-\bv_F)$ for all $F\in\Fh[T]$, and recalling the definition \eqref{eq:norm.1T} of the $\norm{1,T}{{\cdot}}$-norm, we can write
  \[
  \begin{aligned}
    \mathfrak{T}_3
    &\lesssim 
    \norm{\Ldeuxd[T]}{\bpi_T^{k}\bv_T -\bv_T   }
    +
    h_T \norm{1,T}{\uline{\bI}_T^k\bpi_T^{k}\bv_T - \uline{\bv}_T }
    + \sum_{F\in \Fh[T]}h_F^{\frac{1}{2}}  \norm{\Ldeuxd[F]}{\bpi_T^{k}\bv_T- \bv_F}\\
    &\lesssim
    h_T\seminorm{\Hund[T]}{ {\bv}_T}
    +
    h_T\seminorm{\Hund[T]}{\bpi_T^{k}\bv_T -\bv_T   }
    + \sum_{F\in \Fh[T]}h_F^\frac{1}{2} \left(
    \norm{\Ldeuxd[F]}{\bpi_T^{k}\bv_T- \bv_T}
    + \norm{\Ldeuxd[F]}{\bpi_T^{k}\bv_T- \bv_F}
    \right)\\
    &\lesssim
    h_T\seminorm{\Hund[T]}{ {\bv}_T}
    +
    \sum_{F\in \Fh[T]}h_F^\frac{1}{2} \norm{\Ldeuxd[F]}{\bpi_T^{k}\bv_T- \bv_T}+
    \sum_{F\in \Fh[T]}h_F^\frac{1}{2} \norm{\Ldeuxd[F]}{\bv_T- \bv_F}\\
    &\lesssim
    h_T\seminorm{\Hund[T]}{ {\bv}_T}
    +
    h_T\seminorm{1,\partial T}{ \uline{\bv}_T}
    \lesssim
    h_T \left (\seminorm{\Hund[T]}{ {\bv}_T}^2
    +  \seminorm{1,\partial T}{ \uline{\bv}_T}^2
    \right)^{\frac{1}{2}}
  \end{aligned}
  \]
  where, in the second step, we have used the approximation properties \eqref{eq:l2proj:error:cell} of $\bpi_{T}^k$ with $(l,m,r,s)=(k,0,2,1)$ for the first addend,
  expanded the second addend using the definition \eqref{eq:norm.1T} of $\norm{1,T}{{\cdot}}$, 
  and used a triangle inequality in the boundary term;
  in the third step, we have used \eqref{eq:l2proj:error:cell} with $(l,m,r,s)=(k,1,2,1)$ for the second addend, added $\pm \bv_T$  inside the norm of the last term, and used a triangle inequality;
  in the fourth step, we have used for the second addend the trace approximation properties \eqref{eq:l2proj:error:faces} of $\bpi_{T}^k$ with $(l,r,s)=(k,2,1)$ along with $h_F \simeq h_T$ and the definition \eqref{eq:norm.1T}  of $\seminorm{1,\partial T}{{\cdot}}$ for the third addend;
  finally, in the last step, we have used the inequality 
  \begin{equation}\label{eq.l1l2.bound}
    \left(\sum_{i=1}^n a_i \right)^2 {\simeq} \sum_{i=1}^n a_i^2,
  \end{equation}
  valid for any integer $n\ge 1$ and non-negative real numbers $a_i$. 
    This gives
    \begin{equation}\label{eq:est.T3}
      \mathfrak{T}_3 \lesssim h_T \norm{1,T}{\uline{\bv}_T},
    \end{equation}
  Plugging \eqref{eq:est.T1}, \eqref{eq:est.T2}, and \eqref{eq:est.T3} into \eqref{eq:ineq:rtn:bound.proof.i}, and using the definition \eqref{eq:norm.1T} of the $\norm{1,T}{{\cdot}}$-norm, we obtain \eqref{ineq:rtn:bound}.%
  \medskip\\
  \noindent\underline{(ii) \emph{Approximation in $\bW^{s,p}$}}.
  For the sake of brevity, set $\boldsymbol{\Pi}_T^k \coloneq \bR_T^k \circ \uline{\bI}_T^k$.
  To prove \eqref{ineq:rtn:approx.Wsp}, we use \cite[Theorem~1.43]{Di-Pietro.Droniou:20}, which requires to prove the following relations:
  \begin{enumerate}[(a)]
  \item\label{item:polynomial.consistency} {\emph{Polynomial consistency.}} $\boldsymbol{\Pi}_T^k\bq=\bq$ for all $\bq\in\Polyd{k}(T)$.
  \item {\emph{Boundedness.}} For all integer $p \in [1,\infty]$ and all $\bv \in \bW^{1,p}(T)$, it holds
    \begin{gather}\label{eq:approx.Pi:1}
      \norm{\bL^p(T)}{\boldsymbol{\Pi}_T^k \bv}
      \lesssim \norm{\bL^p(T)}{\bv}
      + h_T \seminorm{\bW^{1,p}(T)}{\bv},
      \\ \label{eq:approx.Pi:2}
      \seminorm{\bW^{1,p}(\TTs)}{\boldsymbol{\Pi}_T^k \bv}
      \lesssim \seminorm{\bW^{1,p}(T)}{\bv}.
    \end{gather}
  \end{enumerate}
  We prove {Point~(\ref{item:polynomial.consistency})
    using the same arguments that lead to \eqref{eq:est.T2}} along with the fact that $\bpi_T^k$ is a projection onto $\Polyd{k}(T)$.
  To prove \eqref{eq:approx.Pi:1}, let $\bv \in \bW^{1,p}(T)$ and set, for the sake of brevity, $\uline{\hat{\bv}}_T\coloneq\uline{\bI}_T^k\bv$.
  We have that
  \begin{equation}\label{ineq.Pi.proof}
    \begin{aligned}
      \norm{\bL^p(T)}{\boldsymbol{\Pi}_T^k \bv}^2
      &=
      \norm{\bL^p(T)}{\bR_T^k\uline{\hat{\bv}}_T}^2
      \\
      &\lesssim h_T^{2d{\left(\frac{1}{p}-\frac{1}{2}\right)}}\norm{\bL^2(T)}{\bR_T^k\uline{\hat{\bv}}_T}^2
      \\
      &\lesssim h_T^{2d{\left(\frac{1}{p}-\frac{1}{2}\right)}}\left(
      \norm{\bL^2(T)}{\hat{\bv}_T}^2 + h_T^2 \norm{1,T}{\uline{\hat{\bv}}_T}^2
      \right)
      \\
      &\lesssim \norm{\bL^p(T)}{\hat{\bv}_T}^2 + h_T^2h_T^{2d{\left(\frac{1}{p}-\frac{1}{2}\right)}} \norm{1,T}{\uline{\hat{\bv}}_T}^2
      \\
      &\lesssim \norm{\bL^p(T)}{\hat{\bv}_T}^2 
      + h_T^2\norm{\bL^p(T)}{\GRAD{\hat{\bv}_T}}^2
      + h_T^2\sum_{F \in \mathcal{F}_T}h_T^{2d{\left(\frac{1}{p}-\frac{1}{2}\right)}} h_F^{-1} \norm{\bL^2(F)}{\hat{\bv}_F - \hat{\bv}_T}^2,
    \end{aligned}
  \end{equation}
  where:
  in the second step, we have used the 
    discrete local Lebesgue embeddings \eqref{eq:rev:disc.emb.pw}
  below with $(\alpha,\beta, X)=(p,2,T)$;
  in the third step, we have 
    inserted $\pm\hat{\bv}_T$ into the norm and used a triangle inequality followed by \eqref{ineq:rtn:bound} with $\uline{\bv}_T = \uline{\hat{\bv}}_T$;
  in the fourth step, we have used again the local discrete Lebesgue embeddings \eqref{eq:rev:disc.emb.pw} below, this time with $(\alpha,\beta, X)=(2,p,T)$, for the first term;
  the last step follows from the definition \eqref{eq:norm.1T} of $\norm{1,T}{{\cdot}}$ along with \eqref{eq:rev:disc.emb.pw} with $(\alpha,\beta, X)=(2,p,T)$ for the gradient term.
  To bound the last term in the right hand side of \eqref{ineq.Pi.proof}, we observe that, for all $F \in \mathcal{F}_T$, it holds
  \begin{align}
    h_T^{2d{\left(\frac{1}{p}-\frac{1}{2}\right)}}h_F^{-1}
    \norm{\bL^2(F)}{\hat{\bv}_F - \hat{\bv}_T}^2\
    &\lesssim
    h_T^{2d{\left(\frac{1}{p}-\frac{1}{2}\right)}}h_F^{-1}
    |F|^{2{\left(\frac{1}{2}-\frac{1}{p}\right)}}
    \norm{\bL^p(F)}{\hat{\bv}_F - \hat{\bv}_T}^2\notag\\
    &\lesssim
    h_F^{2{\left(\frac{1-p}{p}\right)}}
    \norm{\bL^p(F)}{\hat{\bv}_F - \hat{\bv}_T}^2\label{ineq.Pi.proof.2},
  \end{align}
  where in the first step we have used
  \eqref{eq:rev:disc.emb} below with $(\alpha,\beta, X)=(2,p,F)$, and in the last step the relations $|F|\simeq h_F^{(d-1)}$ and $h_T\simeq h_F$ valid for regular meshes. 
  Moreover, using \eqref{eq.l1l2.bound} with $n=2$ along with the idempotency of $\bpi_F^k$, we have
  \begin{align*}
    \norm{\bL^p(F)}{\hat{\bv}_F - \hat{\bv}_T}^2
    &\lesssim
    \norm{\bL^p(F)}{\bpi_F^k(\bv - \bpi_T^k\bv)}^2
    +
    \norm{\bL^p(F)}{ \bpi_T^k\bv - \hat{\bv}_T}^2\\
    &\lesssim
    \norm{\bL^p(F)}{\bv - \bpi_T^k\bv}^2
    +
    \norm{\bL^p(F)}{\bv - \hat{\bv}_T}^2,
  \end{align*}
  where, in the second line, we have used the $\bL^p$-boundedness of $\bpi_F^k$ {(cf. \cite[Lemma~1.44]{Di-Pietro.Droniou:20})} for the first term,
  inserted $\pm\bv$ and used a triangle inequality followed by \eqref{eq.l1l2.bound} with $n=2$
  for the second term.
  Plugging the inequality above into \eqref{ineq.Pi.proof.2}, then using the result into \eqref{ineq.Pi.proof}, taking the square root, and then using \eqref{eq.l1l2.bound}, we obtain
  \begin{align*}
    \norm{\bL^p(T)}{\boldsymbol{\Pi}_T^k \bv}
    &\lesssim \norm{\bL^p(T)}{\hat{\bv}_T} 
    + h_T\norm{\bL^p(T)}{\GRAD{\hat{\bv}_T}}
    + h_T\sum_{F \in \mathcal{F}_T} h_F^{{\frac{1-p}{p}}}
    \left(
    \norm{\bL^p(F)}{\bv - \bpi_T^k\bv}
    +\norm{\bL^p(F)}{\bv - \hat{\bv}_T}
    \right).
  \end{align*}
  Therefore, using, respectively, the $\bL^p$- and $\bW^{1,p}$-boundedness of $\bpi_T^{k^*}$ for the first and second term, a triangle inequality along with $h_T\lesssim 1$ followed by the trace approximation properties \eqref{eq:l2proj:error:faces} of the $L^2$-orthogonal projector with $(l,r,s)=(k,p,1)$ for the first term in brackets, and with $(l,r,s)=(k^\star,p,1)$ for the second term in brackets, and, finally, using the geometric bound \eqref{ineq:card.IT.F} for $\Fh[T]$, we obtain \eqref{eq:approx.Pi:1}.
  \\
  Let us now prove \eqref{eq:approx.Pi:2}.
  First of all we notice that, by the discrete inverse inequality \eqref{eq:inverse} below,  it holds
  \begin{align}\label{ineq:RT.inverse}
    \norm{\bL^{2}(T)}{\GRAD\bR_T^k\uline{\bv}_T}
    \lesssim
    h_T^{-1}\norm{\Ldeuxd[T]}{\bR_T^k\uline{\bv}_T}.
  \end{align}
  Thus, recalling that $\boldsymbol{\Pi}_T^k \coloneq \bR_T^k \circ \uline{\bI}_T^k$, we infer
  \[
  \begin{aligned}
    \seminorm{\bW^{1,p}(T)}{\boldsymbol{\Pi}_T^k \bv}
    &= \seminorm{\bW^{1,p}(T)}{\boldsymbol{\Pi}_T^k (\bv - \boldsymbol{\pi}_T^0 \bv)}
    \overset{\eqref{ineq:RT.inverse}}\lesssim h_T^{-1} \norm{\bL^p(T)}{\boldsymbol{\Pi}_T^k (\bv - \boldsymbol{\pi}_T^0 \bv)}
    \\
    \overset{\eqref{eq:approx.Pi:1}}&\lesssim
    h_T^{-1}\norm{\bL^p(T)}{\bv - \boldsymbol{\pi}_T^0 \bv}
    + \seminorm{\bW^{1,p}(T)}{{\bv} - \boldsymbol{\pi}_T^0 \bv}
    \lesssim \seminorm{\bW^{1,p}(T)}{\uline{\bv}_T},
  \end{aligned}
  \]
  where in the last step we have used the approximation properties \eqref{eq:l2proj:error:cell} of $\bpi_{T}^0$ with $(l,m,r,s)=(0,0,p,1)$ for the first term and with $(l,m,r,s)=(0,1,p,1)$ for the second term along with a triangle inequality and the fact that $h_T\lesssim 1$.
  \medskip\\
  \noindent\underline{(iii) \emph{Consistency}}.
  The proof follows the same steps as in
  \cite{Castanon-Quiroz.Di-Pietro:23}
  along with  Lemma \ref{lemma:RT.lifting} above
  and  the fact that  $\bpi^{{\rm c},k-1}_{\Goly{},T} \circ \bpi^{{\rm c},k-1}_{\Goly{},\TTs}= \bpi^{{\rm c},k-1}_{\Goly{},T}$,
  since
  $\Goly{{\rm c},k-1}(T)\subset\Goly{{\rm c},k-1}(\TTs)$.
\end{proof}

Let now $\RTN{k}(\Ths)$ denote the global ($\Hdiv$-conforming) Raviart--Thomas--N\'ed\'elec space on $\Ths$.
We define the \emph{global velocity reconstruction} $\bR_h^k: \uline{\bU}_h^k \rightarrow \RTN{k}(\Ths)$ patching the local contributions:
For all  $\uline{\bv}_h \in \uline{\bU}_h^k$,
\begin{align*}
  (\bR_h^k\uline{\bv}_h)_{|T}
  \coloneq 
  \bR_T^k\uline{\bv}_T \qquad\forall T \in {\cal{T}}_h.
\end{align*}
Notice that $\bR_h^k\uline{\bv}_h$ is well-defined, since its normal components across mesh interfaces are continuous as a consequence of  \eqref{eq:darcyT:weak:bd}  combined with the single-valuedness of interface unknowns.

To discretize the unsteady term in \eqref{eq:nstokes:weak}, for all $T \in \Th$ we introduce the bilinear form $a_{\tR,T}:\uline{\bU}_T^k\times\uline{\bU}_T^k \to \mathbb{R}$ defined as follows:
\begin{equation}\label{def:aRT}
  a_{\tR,T}(\uline{\bv}_T,\uline{\bw}_T)
  \coloneq 
  \int_T \bR_T^k\uline{\bv}_T \cdot \bR_T^k\uline{\bw}_T
  + s_{\tR,T}(\uline{\bv}_T,\uline{\bw}_T),
\end{equation}
where  the second term is the following stabilisation bilinear form:
\[
s_{\tR,T}(\uline{\bv}_T,\uline{\bw}_T)
\coloneq 
\int_T\bdelta_{\tR,T}^k\uline{\bv}_T\cdot\bdelta_{\tR,T}^k\uline{\bw}_T
+ \sum _{F\in\calF_T}h_F \int _F \bdelta_{\tR,TF}^k\uline{\bv}_T \cdot \bdelta_{\tR,TF}^k\uline{\bw}_T,
\]
with difference operators $\bdelta_{\tR,T}^k:\uline{\bU}_T^k\to\Polyd{k^\star}(T)$ and, for all $F\in\calF_T, \bdelta_{\tR,TF}^k:\uline{\bU}_T^k\to\Polyd{k}(F)$ such that, for all $\uline{\bv}_T \in \uline{\bU}^k_T$,
\begin{equation}\label{def:deltasR}
  \bdelta_{\tR,T}^k(\uline{\bv}_T)
  \coloneq 
  \bpi^{k^\star}_T(\bR_T^k\uline{\bv}_T- \bv_T)
  \qquad
  \text{and}
  \qquad
  \bdelta_{\tR,TF}^k(\uline{\bv}_T)
  \coloneq 
  \bpi^k_F(\bR_T^k\uline{\bv}_T- \bv_F) \quad \forall F \in \calF_T.
\end{equation}
We also introduce the $\bL^2$-like norm $\norm{\tR,T}{{\cdot}}:\uline{\bU}_T^k\to\mathbb{R}$ induced by  $a_{\tR,T}(\cdot,\cdot)$, i.e.,
\begin{align}\label{def:normRT}
  \norm{\tR,T}{\uline{\bv}_T}
  \coloneq
  a_{\tR,T}(\uline{\bv}_T,\uline{\bv}_T)^\frac{1}{2}.
\end{align}
Finally, we introduce  the global bilinear form $a_{\tR,h}:\uline{\bU}_h^k\times\uline{\bU}_h^k \to \mathbb{R}$, and the corresponding global norm $\norm{\tR,h}{{\cdot}}:\uline{\bU}_h^k\to\mathbb{R}$ defined, respectively, by:
For all $(\uline{\bv}_h,\uline{\bw}_h) \in \uline{\bU}_h^k \times \uline{\bU}_h^k$,
\[ %% \begin{align}\label{def:normRTh}
a_{\tR,h}(\uline{\bv}_h,\uline{\bw}_h)
\coloneq
\sum_{T\in \Th} 
a_{\tR,T}(\uline{\bv}_T,\uline{\bw}_T),
\qquad
\text{and}
\qquad
\|\uline{\bv}_h \|_{\tR,h}^2
\coloneq
\sum_{T\in \Th} 
\|\uline{\bv}_T \|_{\tR,T}^2.
\] %% \end{align}

The following Lemma summarizes some key properties relevant for the analysis.

\begin{lemma}[Properties of $\norm{\tR,T}{{\cdot}}$ ]\label{lem:ns:th}
  For all $T\in \calT_h$ and all $\uline{\bv}_T\in\uline{\bU}_{T}^k$, it holds
  \begin{equation}\label{eq:L2.hho<=RT}
    \norm{\Ldeuxd[T]}{\bv_T}
    {+ h_T\|\uline{\bv}_T \|_{1,T}}
    \lesssim \norm{\tR,T}{\uline{\bv}_T}.
  \end{equation}
\end{lemma}

\begin{remark}[Norm $\norm{\tR,h}{{\cdot}}$]
  A straightforward consequence of the above results is that the map $\norm{\tR,h}{{\cdot}}: \uline{\bU}_{h}^k\to \mathbb{R}$ is a norm in the space $\uline{\bU}_{h,0}^k$.
\end{remark}

\begin{proof}[Proof of Lemma~\ref{lem:ns:th}]
  Using the fact that $\bv_T\in\Polyd{k^\star}(T)$ along with \eqref{eq.l1l2.bound} for $n=2$, we get 
  \begin{align*}
    \|\bv_T\|_{\Ldeuxd[T]}^2 
    =
    \|\bpi_T^{k^\star}\bv_T\|_{\Ldeuxd[T]}^2
    &\lesssim \|\bpi_T^{k^\star}(\bv_T-\bR_T^k\uline{\bv}_T)\|_{\Ldeuxd[T]}^2
    + \|\bpi_T^{k^\star}\bR_T^k\uline{\bv}_T\|_{\Ldeuxd[T]}^2 \\
    &\leq \|\bpi_T^{k^\star}(\bv_T-\bR_T^k\uline{\bv}_T)\|_{\Ldeuxd[T]} ^2
    + \|\bR_T^k\uline{\bv}_T\|_{\Ldeuxd[T]}^2
    \leq
    \|\uline{\bv}_T \|_{\tR,T}^2,
  \end{align*}
  where  we have used the $\bL^2$-boundedness of $\bpi_T^{k^\star}$ in the second line and the definition \eqref{def:normRT} of $\norm{\tR,T}{{\cdot}}$ to conclude.

  To bound the second term in the left-hand side of \eqref{eq:L2.hho<=RT} we preliminarily observe that, for any face $F\in\calF_T$, it holds
  \begin{align*}
    \| {\bv}_F - {\bv}_T \| _{\bL^2(F)}
    &=
    \| {\bv}_F -\bpi_F^k\bR_T^k\uline{\bv}_T + \bpi_F^k\bR_T^k\uline{\bv}_T- {\bv}_T  \| _{\bL^2(F)}\\
    &\leq
    \| \bpi_F^k({\bv}_F - \bR_T^k\uline{\bv}_T)\|_{\bL^2(F)} + \|\bpi_F^k\bR_T^k\uline{\bv}_T- {\bv}_T  \| _{\bL^2(F)}
    \eqcolon
    \frakI_1 + \frakI_2,
  \end{align*}
  where, in the second step, we have used triangle inequalities and then
    $\bpi_F^k \bv_F = \bv_F$ (since $\bv_F\in\Polyd{k}(F)$) in the first term.
  Squaring the above inequality, using \eqref{eq.l1l2.bound} with $n=2$ in the right hand side, and multiplying by $h_F$ both sides,
  we get
  \begin{align}
    h_F\| {\bv}_F - {\bv}_T \| _{\bL^2(F)}^2
    \lesssim
    h_F\frakI_1^2 + h_F\frakI_2^2.
    \label{ineq:fjump.RTnorm}
  \end{align}
  The
  first term above is bounded as follows:
  \begin{align}
    h_F\frakI_1^2
    \overset{\eqref{def:deltasR}}=
    h_F\|\bdelta_{\tR,TF}^k(\uline{\bv}_T)\|_{\bL^2(F)}^2
    \overset{\eqref{def:normRT}}\leq
    \|\uline{\bv}_T \|_{\tR,T}^2.
    \label{ineq:fjump.RTnorm.I1}
  \end{align}
  For the second term in \eqref{ineq:fjump.RTnorm}, we begin as follows:
  \begin{align*}
    \frakI_2
    &=
    \|\bpi_F^k\bR_T^k\uline{\bv}_T + \bpi_T^{k^\star}\bR_T^k\uline{\bv}_T- \bpi_T^{k^\star}\bR_T^k\uline{\bv}_T - {\bv}_T  \| _{\bL^2(F)}\\
    &\lesssim
    \|\bR_T^k\uline{\bv}_T \| _{\bL^2(F)}  + 
    \|\bpi_T^{k^\star}\bR_T^k\uline{\bv}_T \| _{\bL^2(F)} +
    \|\bdelta_{\tR,T}^k(\uline{\bv}_T)\| _{\bL^2(F)},
  \end{align*}
  where in the last step, we have used triangle inequalities, the $\bL^2$-boundedness of  $\bpi_F^k$, and the definition \eqref{def:deltasR} of $\bdelta_{\tR,T}^k$.  
  Denoting by $\tau_\sigma$ the simplex in $\TTs$ which contains $\sigma\in \Fhs[F]$ and using a discrete trace inequality along with $h_F\leq h_T\lesssim h_{\tau_\sigma}$ (the second bound being a consequence of mesh regularity), we get
  $$
  h_F\|\bR_T^k\uline{\bv}_T \|_{\bL^2(F)}^2
  \lesssim
  \sum_{\sigma  \in \Fhs[F]}
  h_{\tau_\sigma}\|\bR_T^k\uline{\bv}_T \| _{\bL^2(\sigma)}^2
  \lesssim
  \sum_{\sigma  \in \Fhs[F]}
  \|\bR_T^k\uline{\bv}_T \| _{\bL^2(\tau_\sigma)}^2
  \le
  \|\bR_T^k\uline{\bv}_T \| _{\bL^2(T)}^2,
  $$
  where, in the last step, we have used the fact that ${\bigcup_{\sigma \in \Fhs[F]}}\tau_\sigma\subset T$.
  Using the same process, we get
  $h_F\|\bpi_T^{k^\star}\bR_T^k\uline{\bv}_T \| _{\bL^2(F)}{^2}
  \lesssim
  \| \bpi_T^{k^\star}\bR_T^k\uline{\bv}_T    \| _{\bL^2(T)}^2
  \leq
  \| \bR_T^k\uline{\bv}_T \| _{\bL^2(T)}^2
  $,
  where we have used the $\bL^2$-boundedness of  $\bpi_T^{k^\star}$ to conclude.
  Thus,
  \begin{align*}
    h_F\frakI_2^2
    &\lesssim
    \norm{\bL^2(T)}{\bR_T^k\uline{\bv}_T}^2
    + h_F\norm{\bL^2(F)}{\bdelta_{\tR,T}^k(\uline{\bv}_T)}^2
    \lesssim
    \norm{\bL^2(T)}{\bR_T^k\uline{\bv}_T}^2
    + \norm{\bL^2(T)}{\bdelta_{\tR,T}^k(\uline{\bv}_T)}^2
    {\overset{\eqref{def:normRT}}\le\norm{\tR,T}{\uline{\bv}_T}^2},
  \end{align*}
  where, in the second step, we have used a discrete trace inequality.
  Using this bound along with \eqref{ineq:fjump.RTnorm.I1} in \eqref{ineq:fjump.RTnorm}, we get
  \begin{align}
    h_F\| {\bv}_F - {\bv}_T \| _{\bL^2(F)}^2
    \lesssim
    \|\uline{\bv}_T \|_{\tR,T}^2.
    \label{ineq:fjump.RTnorm.I2}
  \end{align}
  Recalling the definition \eqref{eq:norm.1T} of  
  $\|\uline{\bv}_T \|_{1,T}$ and
  using a discrete inverse inequality,
  we have that
  $$
  h_T^2
  \|\uline{\bv}_T \|_{1,T}^2
  {\lesssim}
  \| {\bv}_T\|_{\bL^2(T)}^2 + \sum_{F\in \calF_T} h_F\| \bv_F-\bv_T\|^2_{\bL^2(F)}
  \lesssim
  \|\uline{\bv}_T \|_{\tR,T}^2
  $$
  where, in the second step, we have used additionally the inequality $h_T \lesssim h_F$  valid for regular meshes while, in the last, step  we have used \eqref{eq:L2.hho<=RT} to bound the first term above and  \eqref{ineq:fjump.RTnorm.I2} for the second.
  Taking the square root, we conclude.
\end{proof}

%------------------------------------------------------------------------------%

\subsection{Viscous term and pressure-velocity coupling}\label{sec:discrete.problem:viscous.term}

Let an element $T\in\calT_h$ be fixed.
We define the local gradient reconstruction ${\bG_T^k}:\uline{\bU}_T^k \rightarrow {\Polyd{k}(T)}$ such that, for all $\uline{\bv}_T \in \uline{\bU}_T^k$ and all $\btau \in {\Polyd{k}(T)}$,
\[
  \int_T \bG^l_T \uline{\bv}_T: \btau
  =
  - \int_T \bv_T\cdot (\DIV \btau) + \sum_{F\in\calF_T}\int_F \bv_F\cdot\btau\bn_{TF}.
\]
A global gradient reconstruction ${\bG_h^k}:\uline{\bU}_h^k \to {\Polyd{k}(\Th)}$ can be defined setting, for all $\uline{\bv}_h\in\uline{\bU}_h^k$, $(\bG_h^k\uline{\bv}_h)_{|T}\coloneq\bG_T^k\uline{\bv}_T$ for all $T\in\calT_h$.

The viscous term and the pressure-velocity coupling are essentially the same as in the standard HHO method; see, e.g., \cite{Di-Pietro.Krell:18,Botti.Di-Pietro.ea:19}.
We briefly recall them here to make the exposition self-contained.
The viscous bilinear form $a_h$: $\uline{\bU}_h^k \times \uline{\bU}_h^k\rightarrow \mathbb{R}$ is such that, for all $\uline{\bw}_h, \uline{\bv}_h, \in \uline{\bU}_h^k$,
\begin{equation*}%% \label{eq:ah}
  a_h(\uline{\bw}_h,\uline{\bv}_h)\coloneq
  \sum _{T \in \calT_h} \left[
  \int_T \bG_T^k \uline{\bw}_T  : \bG_T^k \uline{\bv}_T
  + s_T(\uline{\bw}_T,\uline{\bv}_T)
  \right],
\end{equation*}
where, for any $T\in\calT_h$, $s_T: \uline{\bU}_T^k \times \uline{\bU}_T^k\rightarrow \mathbb{R}$ 
denotes a local stabilization bilinear form designed according to the principles of \cite[Assumption 2.4]{Di-Pietro.Droniou:20},
 so that the following properties hold:
\begin{enumerate}[(i)]
\item \emph{Stability and boundedness.}
  There exists $C_a>0$ independent of $h$ (and, clearly, also of $\nu$) such that, for all $\uline{\bv}_h \in \uline{\bU}_{h}^k$,
  \begin{equation}\label{eq:ns:stab.h}
    C_a\| \uline{\bv}_h  \|_{1,h}^2
    \le a_h(\uline{\bv}_h,\uline{\bv}_h)
    \le C_a^{-1} \| \uline{\bv}_h  \|_{1,h}^2.
  \end{equation}
  
\item \emph{Consistency.} For all $\bw \in\bU\cap \bH^{k+2}(\calT_h)$ such that $\LAP \bw \in \Ldeuxd$, it holds
  \begin{equation}\label{eq:ns:ah:consistency}
    \|{\cal{E}}_{a,h}(\bw;\cdot)\|_{1,h,*}
    \lesssim
    h^{k+1}|\bw|_{\bH^{k+2}(\calT_h)},
  \end{equation}
  where the linear form
  ${\cal{E}}_{a,h}(\bw;\cdot):\uline{\bU}_{h,0}^k\rightarrow\mathbb{R}$
  representing the consistency error is such that
  \begin{equation}
    {\cal{E}}_{a,h}(\bw;\uline{\bv}_h)
    \coloneq 
    - \int_\Omega \LAP \bw \cdot \bv_h - a_h( \uline{\bI}_h^k\bw,\uline{\bv}_h)
    {\qquad\forall\uline{\bv}_h \in \uline{\bU}_{h,0}^k}.
    \label{eq:Eh.a}
  \end{equation}
\end{enumerate}
A classical example of stabilization bilinear form along with the proofs of properties \eqref{eq:ns:stab.h} and \eqref{eq:ns:ah:consistency} can be found in \cite[Section 5.1]{Di-Pietro.Droniou:20}, to which we refer for further details.

Recalling the definition  \eqref{eq:hho:div:op}  of the local divergence $D_T^k$, the global pressure-velocity coupling bilinear form $b_h:  \uline{\bU}_{h,0}^k \times \Poly{k}(\calT_h)\rightarrow \mathbb{R}$ is such that, for all $(\uline{\bv}_h,q_h) \in  \uline{\bU}_{h,0}^k \times \Poly{k}(\calT_h)$,
\begin{equation}\label{def:b_h}
b_h(\uline{\bv}_h,q_h):= - \sum _{T\in \calT_h} \int_T D_T^k \uline{\bv}_T\, q_T,
\end{equation}
where $q_T \coloneq q_{h|T}$. 
The properties of $b_h$ relevant for the analysis can be found in \cite[Lemma 8.12]{Di-Pietro.Droniou:20}.

\subsection{Body force}\label{subsec:div:rtn}

The discretization of the body force uses the  divergence-preserving velocity reconstruction introduced in Section \ref{sec:discrete.setting:velocity.reconstruction}.
Specifically, we define the bilinear form $\ell_h: \Ldeuxd \times \uline{\bU}_h^k \rightarrow  \mathbb{R}$ such that, for any $\boldsymbol{\phi} \in \Ldeuxd$ and any $\uline{\bv}_h \in \uline{\bU}_h^k$,
\begin{equation} \label{eq:lh:form}
  \ell_h(\boldsymbol{\phi},\uline{\bv}_h) \coloneq \int_\Omega \boldsymbol{\phi}\cdot \bR_h^k\underline{\bv}_h.
\end{equation}
Proceeding as in the proof of \cite[Lemma 4]{Castanon-Quiroz.Di-Pietro:23} (the only difference being the definition of ${\bR}_T^k$, as discussed in Remark~\ref{rem:bRTk}), we can prove the following  consistency property:
For  all $\boldsymbol{\phi} \in \Ldeuxd\cap \bH^k(\calT_h)$,
\[
\norm{1,h,*}{{\cal{E}}_{\ell,h}(\boldsymbol{\phi};\cdot)}
\lesssim
h^{k+1}\seminorm{\bH^{k}(\calT_h)}{\boldsymbol{\phi}},
\]
where the linear form ${\cal{E}}_{\ell,h}(\boldsymbol{\phi};\cdot):\uline{\bU}_{h}^k\rightarrow\mathbb{R}$, 
representing the consistency error {associated with $\ell_h$}, is such that        
\[
\mathcal{E}_{\ell,h}(\boldsymbol{\phi};\uline{\bv}_h)
\coloneq 
\ell_h( \boldsymbol{\phi},\uline{\bv}_h)
- \int_\Omega \boldsymbol{\phi} \cdot \bv_h
= \sum_{T\in \Th} \int_T \bphi \cdot  (\bR_T^k\uline{\bv}_T  - {\bv}_T).
\]

\subsection{Scalar potential operator for convective stabilization}

Recall the fact that, for a given element $T\in\Th$,  we denote by $\bx_T$ the common vertex of all simplices in $\TTs$, as well as the definitions \eqref{def:golyck} and \eqref{def:golyk} of $\Goly{{\rm c},k}(\TTs)$ and $\Goly{k}(\TTs)$ along with the respective $L^2$-orthogonal projectors $\bpi^{{\rm c},k}_{\Goly{},T}$ and $\bpi^{k}_{\Goly{},\TTs}$.
We introduce, for any polynomial degree $l\geq 0$, the potential operator $\potOp{l+1}: \Polyd{l}(\TTs)\to\Poly{l+1}(\TTs)$ such that, for all $\bq \in\Polyd{l}(\TTs)$,
\begin{equation}\label{def:rhoOp}
  \begin{gathered}
    \GRAD \potOp{l+1}\bq
    = (\Id - \bpi^{l}_{\Goly{},\TTs}\bpi^{{\rm c},{l}}_{\Goly{},\TTs} )^{-1}(\bpi^{l}_{\Goly{},\TTs}\bq -\bpi^{l}_{\Goly{},\TTs}\bpi^{{\rm c},l}_{\Goly{},\TTs}\bq )
    \\
    \text{%
      and $(\potOp{l+1}\bq)_{|\tau}(\bx_T)=0$ for all $\tau\in\TTs$,
    }
  \end{gathered}
\end{equation}
where $\Id$ is the identity operator.
Observe that, by definition \eqref{def:golyk}, 
we have $\GRAD \potOp{l+1}\bq\in\Goly{l}(\TTs)$.
The relevant properties of the potential operator $\potOp{l+1}$ are stated in Lemma \ref{lem:rhoOp} below.

\begin{lemma}[Properties of $\potOp{l+1}$]\label{lem:rhoOp}
  The potential operator $\potOp{l+1}$ defined in \eqref{def:rhoOp}
  is a linear operator with the following properties:
  \begin{enumerate}[(i)]
  \item For all $\bb\in\Goly{l}(\TTs)$ and all $\bc\in \Goly{{\rm c},l}(\TTs)$,
    \begin{equation}\label{lem:rhoOp.a}
      \GRAD\potOp{l+1}\bb = \bb
      \qquad\text{and}\qquad
      \potOp{l+1}\bc\equiv 0.      
    \end{equation}
  \item \emph{Boundedness.}
    For all $\bq\in\Polyd{l}(\TTs)$,
    \begin{align}
      \norm{\Ldeux[T]}{\potOp{l+1}\bq}
      + h_T \norm{\Ldeuxd[T]}{\GRAD \potOp{l+1}\bq}
      \lesssim h_T \norm{\Ldeuxd[T]}{\bq}.
      \label{lem:rhoOp.b}
    \end{align}
  \item For all $\bq\in\Polyd{l}(\TTs)$, it holds, 
    \begin{align}
      (\Id - \GRAD \potOp{l+1})\bq\in \Goly{{\rm c},l}(\TTs).
      \label{lem:rhoOp.d}
    \end{align}
  \item For all $\bq\in\Polyd{l}(T)$,
    \begin{align}
      \potOp{l+1}\bq\in \Poly{l+1}(T).
      \label{lem:rhoOp.e}
    \end{align}
  \end{enumerate}
\end{lemma}

The proof of Lemma \ref{lem:rhoOp} requires some preliminary results.
We start by introducing the \emph{recovery operator} (see \cite[Section 2.5]{Di-Pietro.Droniou:21}) $\RecGT{l}:\Goly{l}(\TTs)\times\Goly{{\rm c},l}(\TTs)\rightarrow \Polyd{l}(\TTs)$ defined by
\begin{equation}\label{def:recoveryop}
  \RecGT{l}(\bb,\bc)\coloneq
  \GammaGT{l}(\bb-\bpi^{l}_{\Goly{},\TTs}(\bc))
  +
  \GammacGT{l} (\bc-\bpi^{{\rm c},l}_{\Goly{},\TTs}(\bb)),
\end{equation}
where  $\GammaGT{l}$ and $\GammacGT{l}$ are linear functions $\Polyd{l}(\TTs)\rightarrow \Polyd{l}(\TTs)$ defined by 
\begin{equation}\label{def:golygammas}
  \text{
    $\GammaGT{l} \coloneq (\Id - \bpi^{l}_{\Goly{},\TTs}\bpi^{{\rm c},l}_{\Goly{},\TTs} )^{-1}$
    \qquad and\qquad
    $\GammacGT{l} \coloneq (\Id - \bpi^{{\rm c},l}_{\Goly{},\TTs}\bpi^{l}_{\Goly{},\TTs} )^{-1}$.
  }
\end{equation}
  By definition,
  $\GRAD \potOp{l+1}\bq
  = \GammaGT{l}(\bpi^{l}_{\Goly{},\TTs}\bq -\bpi^{l}_{\Goly{},\TTs}\bpi^{{\rm c},l}_{\Goly{},\TTs}\bq )$.

\begin{lemma}[Properties of $\RecGT{l}$, $\GammaGT{l}$, and $\GammacGT{l}$]
  \label{lem:rec.golygammas}
  The following holds:
  \begin{enumerate}[(i)]

  \item  For all $\bq\in\Polyd{l}(\TTs)$,
    \begin{equation}\label{lem:rec.golygammas.eq.a}
      \bq =\RecGT{l} (\bpi^{l}_{\Goly{},\TTs}\bq,\bpi^{{\rm c},l}_{\Goly{},\TTs}\bq).
    \end{equation}

  \item  For all $\bb\in\Goly{l}(\TTs), \bc\in\Goly{{\rm c},l}(\TTs)$, and 
    $\bq\in\Polyd{l}(\TTs)$,
    \begin{equation}\label{lem:rec.golygammas.eq.b}
      \text{%
        $\GammaGT{l}(\bb-\bpi^{l}_{\Goly{},\TTs}\bq)\in\Goly{l}(\TTs)$
        \qquad and\qquad
        $\GammacGT{l}(\bc-\bpi^{{\rm c},l}_{\Goly{},\TTs}\bq)\in\Goly{{\rm c},l}(\TTs)$.
      }
    \end{equation}

  \item  For all $\bb\in\Goly{l}(\TTs)$, and $ \bc\in\Goly{{\rm c},l}(\TTs)$,
    \begin{equation}\label{lem:rec.golygammas.eq.bb}
      \text{%
        $\GammaGT{l}(\bb-\bpi^{l}_{\Goly{},\TTs}\bpi^{{\rm c},l}_{\Goly{},\TTs}(\bb))=\bb$
        \qquad and\qquad
        $\GammacGT{l}(\bc-\bpi^{{\rm c},l}_{\Goly{},\TTs}\bpi^{l}_{\Goly{},\TTs}(\bc))=\bc$.
      }
    \end{equation}

  \item For all $\bq\in\Polyd{l}(\TTs)$,
    \begin{equation}\label{lem:rec.golygammas.eq.c}
      \text{%
        $\norm{\Ldeuxd[T]}{\GammaGT{l}\bq}\lesssim \norm{\Ldeuxd[T]}{\bq}$
        \qquad and\qquad
        $\norm{\Ldeuxd[T]}{\GammacGT{l}\bq}\lesssim \norm{\Ldeuxd[T]}{\bq}$.
      }
    \end{equation}
  \end{enumerate}
\end{lemma}

\begin{proof}
  \noindent\underline{(i-ii)} The identity
  \eqref{lem:rec.golygammas.eq.a} is a straightforward consequence of \cite[Lemma 1]{Di-Pietro.Droniou:21}.
  We now proceed to  prove the first relation in \eqref{lem:rec.golygammas.eq.b}. The second one follows using entirely similar arguments.
  By \cite[Lemma 1]{Di-Pietro.Droniou:21}, it holds $\norm{\calL(\Polyd{l}(\TTs),\Polyd{l}(\TTs))}{\bpi^{l}_{\Goly{},\TTs}\bpi^{{\rm c},l}_{\Goly{},\TTs}}<1$ for the norm induced by the $\bL^2(T)$-norm, so  we have the  following well-defined expansion:
  \begin{align*}
    \GammaGT{l}(\bb-\bpi^{l}_{\Goly{},\TTs}\bq)
    &=
    \sum_{n=0}^\infty \left({\bpi^{l}_{\Goly{},\TTs}\bpi^{{\rm c},l}_{\Goly{},\TTs}}\right)^n(\bb-\bpi^{l}_{\Goly{},\TTs}\bq)\\
    &=
    (\bb-\bpi^{l}_{\Goly{},\TTs}\bq)
    +
    \left\{\sum_{n=1}^\infty \left({\bpi^{l}_{\Goly{},\TTs}\bpi^{{\rm c},l}_{\Goly{},\TTs}}\right)^n(\bb-\bpi^{l}_{\Goly{},\TTs}\bq)\right\}
    \eqcolon \frakT_1 + \frakT_2.
  \end{align*}
  Since $\bpi^{l}_{\Goly{},\TTs}$ is a projector into the space $\Goly{l}(\TTs)$,  we have that $\frakT_1\in\Goly{l}(\TTs)$. 
  Observe that  the infinite sum contained in $\frakT_2$ is well-defined and, to compute it, let $\widetilde{\bq}\coloneq\bb-\bpi^{l}_{\Goly{},\TTs}\bq\in \Goly{l}(\TTs)$ and write
  \begin{align*}
    \frakT_2
    &=
    \sum_{n=1}^\infty \left({\bpi^{l}_{\Goly{},\TTs}\bpi^{{\rm c},l}_{\Goly{},\TTs}}\right)^n\widetilde{\bq}
    =
    \sum_{n=1}^\infty {\bpi^{l}_{\Goly{},\TTs}\bpi^{{\rm c},l}_{\Goly{},\TTs}}\left({\bpi^{l}_{\Goly{},\TTs}\bpi^{{\rm c},l}_{\Goly{},\TTs}}\right)^{n-1}\widetilde{\bq}\\
    &=
    \sum_{n=1}^\infty \bpi^{l}_{\Goly{},\TTs}\bpi^{l}_{\Goly{},\TTs}\bpi^{{\rm c},l}_{\Goly{},\TTs}\left({\bpi^{l}_{\Goly{},\TTs}\bpi^{{\rm c},l}_{\Goly{},\TTs}}\right)^{n-1}\widetilde{\bq}
    =
    \sum_{n=1}^\infty \bpi^{l}_{\Goly{},\TTs}\left(\left({\bpi^{l}_{\Goly{},\TTs}\bpi^{{\rm c},l}_{\Goly{},\TTs}}\right)^{n}\widetilde{\bq}\right),
  \end{align*}
  where in the third step we have used the idempotency property $\bpi^{l}_{\Goly{},\TTs}=\bpi^{l}_{\Goly{},\TTs} \circ \bpi^{l}_{\Goly{},\TTs}$.
Thus, $\frakT_2$ is equal to a convergent infinite sum
 where each term of this sum is in $\Goly{l}(\TTs)$, but this subspace is closed (since it is finite dimensional), thus
the limit of the partial sum is in $\Goly{l}(\TTs)$, i.e., $\frakT_2 \in \Goly{l}(\TTs)$.
  Therefore, $\frakT_1 + \frakT_2 \in \Goly{l}(\TTs)$ and we conclude.%
  \medskip\\
  \noindent\underline{(iii)} We now proceed to prove the first relation in \eqref{lem:rec.golygammas.eq.bb}. The second one is proved similarly.
  For any $\bb\in\Goly{l}(\TTs)$, expanding $\GammaGT{l}$ as in the previous point, we have that
  \begin{align*}
    \GammaGT{l}(\bb& -\bpi^{l}_{\Goly{},\TTs}\bpi^{{\rm c},l}_{\Goly{},\TTs}\bb)
    =
    \sum_{n=0}^\infty \left(\bpi^{l}_{\Goly{},\TTs}\bpi^{{\rm c},l}_{\Goly{},\TTs}\right)^n
    (\bb-\bpi^{l}_{\Goly{},\TTs}\bpi^{{\rm c},l}_{\Goly{},\TTs}\bb)\\
    &=
    \sum_{n=0}^\infty \left({\bpi^{l}_{\Goly{},\TTs}\bpi^{{\rm c},l}_{\Goly{},\TTs}}\right)^n
    \bb
    -
    \sum_{n=0}^\infty \left({\bpi^{l}_{\Goly{},\TTs}\bpi^{{\rm c},l}_{\Goly{},\TTs}}\right)^n
    \bpi^{l}_{\Goly{},\TTs}\bpi^{{\rm c},l}_{\Goly{},\TTs}\bb\\
    &=
    \bb
    + \cancel{%
      \left[
      \sum_{n=1}^\infty \left(
      \bpi^{l}_{\Goly{},\TTs}\bpi^{{\rm c},l}_{\Goly{},\TTs}
      \right)^n
      \bb
      - \sum_{n=0}^\infty \left({\bpi^{l}_{\Goly{},\TTs}\bpi^{{\rm c},l}_{\Goly{},\TTs}}\right)^{n+1}
      \bb
      \right]
    } =\bb.
  \end{align*}
  \\
  \noindent\underline{(iv)} We prove the first relation in \eqref{lem:rec.golygammas.eq.c}.
  The second one follows from a similar reasoning.
  Observe first that, for all $\bq\in\Polyd{l}(\TTs)$, we can express $\bpi^{l}_{\Goly{},\TTs}\bq$   as 
  $\bpi^{l}_{\Goly{},\TTs}\bq=\sum_{\tau\in\TTs}\chi_{\tau}\bpi^{l}_{\Goly{},\tau}\bq_{|\tau}$, where $\chi_{\tau}$ is the characteristic function of  $\tau\in \TTs$;
  similarly, we have
  $\bpi^{{\rm c}, l}_{\Goly{},\TTs}\bq=\sum_{\tau\in\TTs}\chi_{\tau}\bpi^{{\rm c}, l}_{\Goly{},\tau}\bq_{|\tau}$.
  Thus,
  \begin{multline*}
    \bpi^{l}_{\Goly{},\TTs}\bpi^{{\rm c}, l}_{\Goly{},\TTs}\bq
    =
    \bpi^{l}_{\Goly{},\TTs}\left(
    \sum_{\tau\in\TTs}\chi_{\tau}\bpi^{{\rm c}, l}_{\Goly{},\tau}\bq_{|\tau}
    \right)
    \\
    =
    \sum_{\tau^\prime\in\TTs}\chi_{\tau^\prime}\bpi^{l}_{\Goly{},\tau^\prime}\left(
    \sum_{\tau\in\TTs}\chi_{\tau}\bpi^{{\rm c}, l}_{\Goly{},\tau}\bq_{|\tau}
    \right)_{|\tau^\prime}
    =
    \sum_{\tau\in\TTs}\chi_{\tau}\bpi^{l}_{\Goly{},\tau}\bpi^{{\rm c}, l}_{\Goly{},\tau}\bq_{|\tau},
  \end{multline*}
  and, using the same procedure recursively, we obtain
  $$
  \left(\bpi^{l}_{\Goly{},\TTs}\bpi^{{\rm c}, l}_{\Goly{},\TTs}\right)^n\bq
  =
  \sum_{\tau\in\TTs}\chi_{\tau}\left(\bpi^{l}_{\Goly{},\tau}\bpi^{{\rm c}, l}_{\Goly{},\tau}\right)^n\bq_{|\tau},
  $$
  thus it is inferred that
  \[
  \GammaGT{l}\bq
  =
  \sum_{n=0}^\infty \left({\bpi^{l}_{\Goly{},\TTs}\bpi^{{\rm c},l}_{\Goly{},\TTs}}\right)^n\bq
  =
  \sum_{n=0}^\infty 
  \sum_{\tau\in\TTs}\chi_{\tau}\left(\bpi^{l}_{\Goly{},\tau}\bpi^{{\rm c}, l}_{\Goly{},\tau}\right)^n\bq_{|\tau}
  =
  \sum_{\tau\in\TTs}
  \chi_{\tau}
  \sum_{n=0}^\infty 
  \left(\bpi^{l}_{\Goly{},\tau}\bpi^{{\rm c}, l}_{\Goly{},\tau}\right)^n\bq_{|\tau},\\
  \]
  where, in the last step, we have used Fubini's theorem to exchange the order of the sums
  and the fact that the infinite expansion of  $\GammaGT{l}$ is well-defined.
  Now, using a triangle inequality and the fact that $|\chi_{\tau}(\bx)|= 1$
  for all $\bx \in \tau$, we get
  \begin{align*}
    \norm{\Ldeuxd[T]}{\GammaGT{l}\bq}
    &\leq
    \sum_{\tau\in\TTs}
    \bigg\|
    \sum_{n=0}^\infty 
    \left(\bpi^{l}_{\Goly{},\tau}\bpi^{{\rm c}, l}_{\Goly{},\tau}\right)^n\bq_{|\tau}
    \bigg \|_{\Ldeuxd[\tau]}\\
    &\lesssim
    \sum_{\tau\in\TTs}
    \norm{\Ldeuxd[\tau]}{\bq_{|\tau}}
    \leq
    \sum_{\tau\in\TTs}
    \norm{\Ldeuxd[T]}{\bq}
    \overset{\eqref{ineq:card.IT.F}}\lesssim
    \norm{\Ldeuxd[T]}{\bq},
  \end{align*}
  where we have used the inequality $\norm{\calL(\Polyd{l}(\tau),\Polyd{l}(\tau))}{(\Id -\bpi^{l}_{\Goly{},\tau}\bpi^{{\rm c}, l}_{\Goly{},\tau})^{-1}}\lesssim 1$ (see \cite[Proof of Lemma 2]{Di-Pietro.Droniou:21}) in the second step and
  the fact that $\tau\subset T$ for $\tau\in\TTs$ in third step.
\end{proof}

The second intermediate result needed in the proof of Lemma \ref{lem:rhoOp} is the following local Poincar\'e inequality.

\begin{lemma}[Poincaré inequality over a simplex]\label{lem:poincare.simplex}
  Let a mesh element $T\in \Th$, a simplex $\tau\in\TTs$, and  integer $l\geq0$ be given.
  Then, for any $q\in\Poly{l}(\tau)$ vanishing at one of the vertices of $\tau$, the following holds:
  \begin{equation}\label{ineq:poincare.tau}
    \norm{\Ldeux[\tau]}{q}
    \lesssim
    h_\tau
    \norm{\Ldeuxd[\tau]}{\GRAD q}.
  \end{equation}
\end{lemma}

\begin{proof}
  The case $l=0$ is trivial. Then we separate the remaining cases
  when $l=1$, and $l>1$. We begin with $l=1$.
  By assumption, there exists a vertex  of $\tau$ denoted by $\bx_\tau\in\Real^d$
  such that $q(\bx_\tau)=0$. 
  Thus, $q$ can be expressed as  $q(\bx)=\sum_{i=1}^d a_i( x_i - x_{\tau,i})$
  where $a_i\in\Real$, and  $x_i$ and $x_{\tau,i}$ are the Cartesian coordinates of $\bx$, and  $\bx_\tau$, respectively.
  We write
  \[
    \norm{\Ldeux[\tau]}{q}^2
    =
    \int_\tau\left(\sum_{i=1}^d a_i( x_i - x_{\tau,i}) \right)^2 
    \overset{\eqref{eq.l1l2.bound}}\lesssim
    \sum_{i=1}^d\int_\tau a_i^2( x_i - x_{\tau,i})^2 
    \leq
    h_\tau^2\sum_{i=1}^d\int_\tau a_i^2 
    =
    h_\tau
    \norm{\Ldeuxd[\tau]}{\GRAD q}^2.
  \]
  Passing to the square root, we get \eqref{ineq:poincare.tau}.
  Now, to prove the case $l>1$, let $q\in\Poly{l}(\tau)$ and denote by $\hat{q}$ the standard nodal interpolate of $q$ on $\Poly{1}(\tau)$ (see, e.g., \cite[Eq. (1.36)]{Ern.Guermond:04}).
  Then, it is inferred that
  \[ %% \begin{multline*}     
    \norm{\Ldeux[\tau]}{q}
    {\le}
    \norm{\Ldeux[\tau]}{q - \hat{q} }
    + \norm{\Ldeux[\tau]}{\hat{q} }
%%     \\
%%     \lesssim
%%     h_\tau^2\seminorm{H^2(\tau)}{q}
%%     +  \norm{\Ldeux[\tau]}{\hat{q} }
    \lesssim
    h_\tau\seminorm{H^1(\tau)}{q}
    +
    \norm{\Ldeux[\tau]}{\hat{q} }
    \lesssim
    h_\tau\norm{\Ldeuxd[\tau]}{\GRAD q},
  \] %% \end{multline*}
  where in the first step we have used a triangle inequality,
  in the second step standard approximation properties of $\hat{q}$ (\cite[Theorem (1.103)]{Ern.Guermond:04}) followed by an inverse inequality to write $\norm{\Ldeux[\tau]}{q - \hat{q} } \lesssim h_\tau^2 \seminorm{H^2(\tau)}{q} \lesssim h_\tau \seminorm{H^1(\tau)}{q}$,
  and in the last step the fact that $\hat{q}\in\Poly{1}(\tau)$
  and $\hat{q}(\bx_\tau)=0$, so we can use \eqref{ineq:poincare.tau}, since this case has already been  proved.
\end{proof}

\begin{proof}[Proof of Lemma \ref{lem:rhoOp}]
  We first prove that $\potOp{l+1}$ is well defined.
  By the definitions \eqref{def:rhoOp} of $\potOp{l+1}$ and \eqref{def:golygammas} of ${\bGamma}^l_{\Goly{},\TTs}$, it holds, for all $\bq\in \Polyd{l}(\TTs)$, 
  \begin{equation}\label{eq:rhoOp.Gamma}
    \GRAD \potOp{l+1} \bq={\bGamma}^l_{\Goly{},\TTs}
    (\bpi^{l}_{\Goly{},\TTs}\bq -\bpi^{l}_{\Goly{},\TTs}\bpi^{{\rm c},l}_{\Goly{},\TTs}\bq).
  \end{equation}
  By the first equation of \eqref{lem:rec.golygammas.eq.b}, 
  ${\bGamma}^l_{\Goly{},\TTs} (\bpi^{l}_{\Goly{},\TTs}\bq -\bpi^{l}_{\Goly{},\TTs}\bpi^{{\rm c},l}_{\Goly{},\TTs}\bq)\eqcolon\bg \in \Goly{l}(\TTs)$.
  Thus, there exists  $g\in\Polyd{l+1}(\Ths)$ such that $\GRAD g= \bg$.
  Let $\potOp{l+1} \bq \coloneq g - g(\bx_T)$.
  Clearly, $\potOp{l+1}\bq$ satisfies \eqref{eq:rhoOp.Gamma} as well as the condition at the common vertex $\bx_T$ required by \eqref{def:rhoOp}.
  Using this process and the fact that the operators ${\bGamma}^l_{\Goly{},\TTs}, \bpi^{l}_{\Goly{},\TTs}$ and $\bpi^{{\rm c},l}_{\Goly{},\TTs}$ are all linear, it is easy to prove the linearity of $\potOp{l+1}$.
  \medskip\\
  \noindent\underline{(i) \emph{Proof of \eqref{lem:rhoOp.a}.}} The first identity in  \eqref{lem:rhoOp.a} 
  comes from \eqref{eq:rhoOp.Gamma} taking $\bq=\bb\in\Goly{l}(\TTs)\subset \Polyd{l}(\TTs)$ and using the first identity in \eqref{lem:rec.golygammas.eq.bb} and the fact that $\bpi^{l}_{\Goly{},\TTs}\bb=\bb$.
  To prove the second identity in \eqref{lem:rhoOp.a}, we write \eqref{eq:rhoOp.Gamma} with $\bq=\bc\in\Goly{\rm c, l}(\TTs)\subset \Polyd{l}(\TTs)$ to get
  \[
  \GRAD \potOp{l+1}_{\TTs} \bc
  = \GammaGT{l}(\bpi^{l}_{\Goly{},\TTs}\bc -\bpi^{l}_{\Goly{},\TTs}\bpi^{{\rm c},l}_{\Goly{},\TTs}\bc)
  = \GammaGT{l}(\bpi^{l}_{\Goly{},\TTs}\bc -\bpi^{l}_{\Goly{},\TTs}\bc)
  = \bzero,
  \]
  where, in the third step, we have used the fact that $\bpi^{{\rm c},k-1}_{\Goly{},\TTs}$ is  a projection onto $\Goly{{\rm c},k-1}(\TTs)$ (so that $\bpi^{{\rm c},l}_{\Goly{},\TTs}\bc = \bc$) and, in the last step, the linearity of $\GammaGT{l}$.
  The conclusion follows from the condition at the common vertex $\bx_T$ required by \eqref{def:rhoOp}.
  \medskip\\
  \noindent\underline{(ii) \emph{Proof of \eqref{lem:rhoOp.b}.}}
  We first prove the bound for the first term in the left-hand side of \eqref{lem:rhoOp.b}.
  We observe that
  \begin{equation}\label{eq:continuity.varrho:L2}
    \norm{\Ldeux[T]}{\potOp{l+1}\bq}^2
    =
    \sum_{\tau\in\TTs}
    \norm{\Ldeux[\tau]}{\potOp{l+1}\bq}^2
    \lesssim
    \sum_{\tau\in\TTs}
    h_\tau^2\norm{\Ldeux[\tau]}{\GRAD\potOp{l+1}\bq}^2
    \lesssim
    \sum_{\tau\in\TTs}
    h_\tau^2\norm{\Ldeux[\tau]}{\bq}^2
    \leq
    h_T^2\norm{\Ldeux[T]}{\bq}^2,
  \end{equation}
  where, in the second step, we have used the local Poincaré inequality \eqref{ineq:poincare.tau} for all $\tau \in \TTs$
  (this is made possible by the condition at the  common vertex $\bx_T$  in \eqref{def:rhoOp}), in the third step we have used first \eqref{eq:rhoOp.Gamma}, then the first bound  in \eqref{lem:rec.golygammas.eq.c} along with a triangle inequality and then the $\bL^2$-boundedness of $\bpi^{{\rm c},l}_{\Goly{},\TTs}$ and $\bpi^{l}_{\Goly{},\TTs}$,
  and in the last step the fact that $\tau\subset T$ (so that, by mesh regularity, $h_\tau^{-1}\lesssim h_T^{-1}$).
  Taking the square root yields the bound for the first term in the left-hand side of \eqref{lem:rhoOp.b}.
  The bound for the second term in the left-hand side of \eqref{lem:rhoOp.b} is a straightforward consequence of \eqref{eq:continuity.varrho:L2} combined with the discrete inverse inequality \eqref{eq:inverse} below with $p = 2$.
  \medskip\\
  \noindent\underline{(iii) \emph{Proof of \eqref{lem:rhoOp.d}.}}  To prove \eqref{lem:rhoOp.d}, we use 
  the identities
  $
  \bq
  =
  {\GRAD \potOp{l+1}\bq} + (\Id -\GRAD \potOp{l+1})\bq
  $
  and \eqref{lem:rec.golygammas.eq.a} along with the definitions \eqref{def:recoveryop} of $\RecGT{l}$ and \eqref{def:golygammas} of $\GammaGT{l}$ and $\GammacGT{l}$,
  and use \eqref{eq:rhoOp.Gamma} to write
  \[
  (\Id -\GRAD \potOp{l+1})\bq
  = \bq - \GRAD \potOp{l+1}\bq
  \overset{\eqref{lem:rec.golygammas.eq.a},\,\eqref{def:recoveryop},\,\eqref{eq:rhoOp.Gamma}}
  =\GammacGT{l}(\bpi^{{\rm c},l}_{\Goly{},\TTs}\bq - \bpi^{{\rm c},l}_{\Goly{},\TTs}\bpi^{l}_{\Goly{},\TTs}\bq).
  \]
  Using \eqref{lem:rec.golygammas.eq.b} with $(\bc,\bq) = (\bpi^{{\rm c},l}_{\Goly{},\TTs}\bq,\bpi^{l}_{\Goly{},\TTs}\bq)$, we obtain $(\Id -\GRAD \potOp{l+1})\bq \in\Goly{{\rm c},l}(\TTs)$.  
  \medskip\\
  \noindent\underline{(iv) \emph{Proof of \eqref{lem:rhoOp.e}.}}
  Let  $\bq \in \Polyd{l}(T)$
    and write
  $\bq {\overset{\eqref{eq:kz.decomp}}=} \bg_T + \bg_T^{\rm c}$, where $\bg_T\in\Goly{l}(T)$ and $\bg_T^{\rm c} \in  \Goly{{\rm c},l}(T)$.
  Using  the fact that $\Goly{l}(T)\subset \Goly{l}(\TTs)$ and $\Goly{{\rm c},l}(T)\subset \Goly{{\rm c},l}(\TTs)$ 
  along with the definition \eqref{def:rhoOp} of  $\potOp{l+1}$, \eqref{lem:rhoOp.a}, and its linearity, it is readily seen that
  $$ \Goly{k-1}(\TTs)\ni \GRAD\potOp{l+1} {\bq} 
  = \GRAD\potOp{l+1}(\bg_T + \bg_T^{\rm c})
  =
  \GRAD\potOp{l+1}\bg_T
  + \GRAD\potOp{l+1}\bg_T^{\rm c}
  = \bg_T + \bzero
  = \bg_T  \in \Goly{k-1}(T),
  $$
  thus $\potOp{{l+1}}(\bq) = g_T + g_{0,\TTs}$, where $g_T\in \Poly{l+1}(T)$ and $g_{0,\TTs}\in \Poly{0}(\TTs)$.
  But, using the second condition in \eqref{def:rhoOp}, we have that $g_T(\bx_T)= - g_{0,\TTs}(\bx_T)$ for all $\tau$ in $\TTs$,
  thus $g_{0,\TTs}\in \Poly{0}(T)$, since $g_T(\bx_T)$ is independent of $\tau$, and we conclude.
\end{proof}

\subsection{Convective term}\label{sec:discrete.problem:convective.term}

Let an element $T\in\calT_h$ be fixed.
For every simplicial face $\sigma\in \Fhs$, we introduce an arbitrary but fixed ordering of the simplicial elements $\tau_1$ and $\tau_2$ such that $\sigma \subset \partial \tau_1 \cap \partial \tau_2$, and let $\bn_\sigma\coloneq\bn_{\tau_1\sigma}=-\bn_{\tau_2\sigma}$, where  $\bn_{\tau_i\sigma}, i\in\{1,2\}$,
denotes the unit vector normal to $\sigma$ pointing out of $\tau_i$ (see Figure \ref{fig:simplices.faces.T.b}).
With this convention, for every scalar-valued function $\zeta$ admitting a possibly two-valued trace on $\sigma$, we define the jump and average of $\zeta$ across $\sigma$ respectively as
\begin{equation}\label{eq:trace.op}
  \llbracket \zeta\rrbracket_\sigma \coloneq \zeta_{|\tau_1} - \zeta_{|\tau_2}
  \qquad\text{and}\qquad
  \lbrace \zeta\rbrace_\sigma \coloneq \frac{1}{2}\left(\zeta_{|\tau_1} + \zeta_{|\tau_2}\right).
\end{equation}
For any boundary simplicial face $\sigma \subset F \in \Fhb$, we set $ \llbracket \zeta\rrbracket_\sigma \coloneq  \lbrace \zeta\rbrace_\sigma \coloneq \zeta $. When applied to vector- or tensor-valued functions, the jump and average  operators act component-wise.

We introduce the global function $t_h:\left[\uline{\bU}_h^k\right]^3\to\mathbb{R}$ such that, for all $(\uline{\bw}_h,\uline{\bv}_h,\uline{\bz}_h) \in \left[\uline{\bU}_h^k\right]^3$,
\begin{equation}\label{eq:th}
  \begin{aligned}
    t_h(\uline{\bw}_h,\uline{\bv}_h,\uline{\bz}_h)
    \coloneq 
	  & 
      \int_\Omega (\bR_h^k\uline{\bw}_h \cdot \GRAD) \bR_h^k\uline{\bv}_h \cdot \bR_h^k\uline{\bz}_h
	  - \sum_{\sigma\in\Fhs^{\rm i}}\int_\sigma (\bR_h^k\uline{\bw}_h\cdot\bn_\sigma) \llbracket \bR_h^k\uline{\bv}_h \rrbracket_\sigma \cdot \lbrace \bR_h^k\uline{\bz}_h\rbrace_\sigma\\
	  &+ \sum_{\sigma\in\Fhs^{\rm i}}\int_\sigma  \frac{1}{2} |\bR_h^k\uline{\bw}_h\cdot\bn_\sigma| \llbracket \bR_h^k\uline{\bv}_h \rrbracket_\sigma \cdot \llbracket \bR_h^k\uline{\bz}_h\rrbracket_\sigma
	  + \sum_{T \in\Th}\sum_{\sigma\in\Fhs[T]^{\rm i}}t_{T,\sigma}^k(\uline{\bw}_T,\uline{\bv}_T,\uline{\bz}_T).
  \end{aligned}
\end{equation}
The first term in the second line is the usual upwind stabilization.
The form $t_{T,\sigma}^k:\uline{\bU}_T^k\times\uline{\bU}_T^k\times\uline{\bU}_T^k\to\mathbb{R}$ also works as a penalty term and, setting ${\bw}_T^0\coloneq\bpi_T^0\bw_T$, it is defined as follows:
\begin{equation}\label{eq:tT}
  t_{T,\sigma}^k(\uline{\bw}_T,\uline{\bv}_T,\uline{\bz}_T)
  \coloneq
  \begin{cases}
    0 &\,\text{if }k=0,\\
    \int_\sigma
    \llbracket
    \potOp{k}({{\bpi_{\TTs}^{k-1}}}(({\bw}_T^0 \cdot \GRAD) \bR_T^k\uline{\bv}_T)) \rrbracket_\sigma 
    \llbracket
    \potOp{k}({{\bpi_{\TTs}^{k-1}}}(({\bw}_T^0 \cdot \GRAD) \bR_T^k\uline{\bz}_T)) \rrbracket_\sigma
 &\,\text{if }k\ge1\\
  \end{cases}.
\end{equation}

Notice that $t_{T,\sigma}^k$ in \eqref{eq:tT} is linear only in its second and third arguments.
\begin{remark}[Comparison with Virtual Elements]\label{rem:tsigma.VEM}
  In \cite{Beirao-da-Veiga.Brezzi.ea:Oseen.2021}, the authors propose a Virtual Element discretization of the Oseen equation which includes a penalization term somewhat similar to $t_{T,\sigma}^k$ (see, in particular, \cite[Eq. (4.18)]{Beirao-da-Veiga.Brezzi.ea:Oseen.2021}).
  The penalization term $t_{T,\sigma}^k$ used here appears, however, more subtle, since the factor $({\bw}^0_T\cdot \GRAD) \bR_T^k\uline{\bv}_T$ is in $\Polyd{k}(\TTs)$, and does not necessarily belong to $\Polyd{k}(T)$.  How to deal with this difficulty is  one of the main contributions of this paper, as detailed in Section \ref{sec:convergence.analysis.vel} below.
\end{remark}

\subsection{Discrete problem}\label{sec:discrete.problem}

The HHO discrete formulation of problem \eqref{eq:nstokes:weak} then reads:
Find $\uline{\bu}_h: [0,\tF] \to  \uline{\bU}_{h,0}^k$ with $\uline{\bu}_h(0)= \uline{\bI}_h^k\bu_0 \in \uline{\bU}_{h,0}^{k}$ and $p_h: (0,\tF] \to P_h^k$ such that it holds, for all $(\uline{\bv}_h,q_h)\in \uline{\bU}_{h,0}^{k}\times {\Poly{k}(\Th)}$ and almost every $t\in(0,\tF)$,
\begin{multline}\label{eq:nstokes:discrete}
     a_{R,h}(d_t \uline{\bu}_h(t),\uline{\bv}_h) +
    \nu a_h(\uline{\bu}_h(t),\uline{\bv}_h)
    +
    t_h(\uline{\bu}_h(t),\uline{\bu}_h{(t)},\uline{\bv}_h)
    +b_h(\uline{\bv}_h,p_h{(t)})
    \\
    -b_h(\uline{\bu}_h{(t)},q_h)
     = \ell_h(\bef{(t)},\uline{\bv}_h).
\end{multline}

%------------------------------------------------------------------------------%
%------------------------------------------------------------------------------%

\section{Velocity error analysis}\label{sec:convergence.analysis.vel}

Most of the relations that we will write in this section hold for almost every $t\in(0,\tF)$. To simplify the notation, we omit the dependence on $t$ and write, e.g., $\bu$ and $\uline{\bu}_h$ instead of $\bu(t)$ and $\uline{\bu}_h(t)$.
Let $\bu$ solve the continuous problem \eqref{eq:nstokes:weak} and $\uline{\bu}_h$ the discrete problem \eqref{eq:nstokes:discrete}.
We define  the velocity error as
\begin{equation}\label{def:error.split}
  \uline{\be}_h
  \coloneq
  \uline{\bu}_h - \hat{\uline{\bu}}_h,
\end{equation}
where $\hat{\uline{\bu}}_h=\uline{\bI}_h^k\bu$.
It will be also useful to define {$\bbeta_h \in \bL^2(\Ths)$ such that}
\begin{equation}\label{def:bbeta.T}
  \text{%
    ${(\bbeta_h)_{|T} \coloneq} \bbeta_T \coloneq \bR_T^k\hhointerpT{\bu} - \bu$ for all $T\in \Th$.
  }
\end{equation}
We now present the two main results of this manuscript, which bound the velocity error \eqref{def:error.split}.

\begin{theorem}[Velocity error estimate]\label{error.vel.theorem}
  Let the pair $(\bu,p)$ be the solution of the continuous problem \eqref{eq:nstokes:weak},
  and let  the pair $(\uline{\bu}_h, p_h)$
  be the solution of the  discrete problem \eqref{eq:nstokes:discrete}.
  Suppose that 
    $\bu \in L^{2}(\bW^{1,\infty}(\Omega))\cap L^\infty(\bH^{1}(\Omega)) \cap L^2(\bH^{2}(\Omega))$,
  $d_t \bu \in
  L^2(\bH^{1}(\Omega))$,
  and
  $p:(0,\tF)\to  P\cap\Hun$.
  Then, we have the following error estimate:
  \begin{align}\label{ineq:main.theo}
    \norm{L^\infty(\norm{R,h}{\cdot})}{\uline{\be}_h}^2
    +
    \nu\int _0^{\tF} \norm{1,h}{\uline{\be}_h}^2
    \lesssim
    e^{G(\bu,\tF)}
    \int _0^{\tF} 
    H(\bu),
  \end{align}
  where
  $
  G(\bu,\tF)\coloneq
  \tF
  +
  \int _0^{\tF} 
  \left( 
  \seminorm{\bW^{1,\infty}(\Omega)}{\bu}
  + h\gamma_k\norm{\Linftyd[\Omega]}{\bu}
  + h \widetilde{\delta}_{0k} \seminorm{\bW^{1,\infty}(\Omega)}{\bu}^2
  \right)
  $
  and 
  \begin{equation}\label{def:H_u}
  H(\bu)
  \coloneq
  \sum_{T\in\Th}
  \nu
  h_T^2\norm{{\bL^2(T)}}{\LAP{\bu} -\bpi_T^{k-1}\LAP{\bu} }^2
  +
  \nu\|{\cal{E}}_{a,h}(\bu;\cdot)\|_{1,h,*}^2
  +
  \frakN_1
  +
  \frakN_2,
  \end{equation}
  where, denoting by $\delta_{0k}$ the Kronecker delta, we have let
  \begin{equation}\label{eq:tilde.delta}
    \widetilde{\delta}_{0k}\coloneq 1 - \delta_{0k},
  \end{equation}
  while $\gamma_k$ is defined by
  \begin{equation}\label{def:gammak}
    \gamma_{k}\coloneq
    \begin{cases}
      0,& \text{if } k\in\{0,1\}\\
      1,& \text{otherwise}
    \end{cases},
  \end{equation}
  and the terms $\frakN_1$ and  $\frakN_2$ are 
    respectively such that, for almost every $t \in (0,\tF)$,
  \begin{equation}\label{def:ve.fI_1.tilde}
    \begin{aligned}
      \frakN_1
      &\coloneq
      \sum_{T\in \Th}\left( \seminorm{\bW^{1,\infty}(T)}{\bu}
      \norm{\bL^{2}(T)}{\hhointerpTT{\bu} - \bu}^2
      +
      {\gamma_{k}}{h_T^{-3}}
      \norm{\Linftyd[T]}{\bu}
      \norm{\Ldeuxd[T]}{\hhointerpTT{\bu}-\bu}^2\right )\\
      &\quad+
      \sum_{T\in \Th}
      \norm{\bW^{1,\infty}(T)}{\bu}
      \left(
      h_{T}^{-1}\norm{\Ldeuxd[T]}{ \bbeta_{T}}^2
      +
      h_{T}\seminorm{\bH^{1}(\TTs)}{ \bbeta_{T}}^2
      +
      \norm{\Ldeuxd[T]}{ \bbeta_{T}}^2
      \right)
      \\
      &\quad+
      \widetilde{\delta}_{0k}
      \sum_{T\in \Th} 
      h_{T}
      \norm{\bL^{\infty}(T)}{\bu}^2
      \left(
      \norm{\bL^2(T)}{ \GRAD \bbeta_T}^2
      +
      \norm{\bL^2(T)}{\GRAD( \bu- \bpi_T^k{\bu} )}^2
      \right),
    \end{aligned}
  \end{equation}
  and
  \begin{equation}\label{def:ve.fI_2.tilde}
    \begin{aligned}
      \frakN_2
      &\coloneq
      \sum_{T\in\Th}
      \left(
      \norm{\Ldeuxd[T]}{ \bR_T^k(\widehat{\underline{{d_t}\bu}}_T) 
        - {{d_t\bu}}}^2
      +
      \norm{\Ldeuxd[T]}{ {d_t\bu} 
        - (\widehat{{d_t\bu}})_T}^2
      \right)\\
      &\quad
      +
      \sum_{T\in\Th}
      h_T^2\norm{\Ldeuxd[T]}{\GRAD((\widehat{{d_t\bu}})_T
        - d_t\bu)}^2
      +
      \sum_{T\in\Th}
      \sum_{F\in \calF_T}
      h_F\norm{\Ldeuxd[F]}{d_t\bu
        - (\widehat{{d_t\bu}})_F}^2,
    \end{aligned}
  \end{equation}
where $\widehat{\underline{d_t\bu}}_T \coloneq\underline{\bI}_T^k(d_t\bu) $ for all $T\in\Th$.
\end{theorem}

Using the theorem above, the approximation properties of the local operator $\bR_T^k$ presented in Lemma \ref{lemm:rtn}, and the approximation properties  \eqref{eq:l2proj:error:cell} of  $\bpi_T^{k^\star}$ for all $T\in\Th$,
along with the following standard property of the $L^2$-projector $\bpi_F^k$  on $F\in\Fh$:
\begin{equation*}
  \| \bpi_F^k\bv- \bv \|_{\bL^2(F)}
  =
  \inf_{\bw \in \Polyd{k}(F)}\| \bw - \bv \|_{\bL^2(F)}
  \leq\|\bpi_T^k\bv -\bv \|_{\bL^2(F)},
\end{equation*}
we have  the following corollary.

\begin{corollary}[Velocity convergence rates]\label{cor.vel.convergence}
  Under the notations and the assumptions of the previous theorem, and additionally assuming $\bu \in L^\infty(\bH^{k+1}(\Th))\cap L^2(\bH^{k+2}(\Th))$ for $k\in\{0,1\}$, $\bu \in L^\infty(\bH^{k+2}(\Th))$ for $k>1$, and $d_t \bu \in L^2(\bH^{k+1}(\Th))$, it holds:
  \begin{align}\label{ineq:main.cor}
    \norm{L^\infty(\norm{R,h}{\cdot})}{\uline{\be}_h}^2
    +
    \nu\int _0^{\tF} \norm{1,h}{\uline{\be}_h}^2
    \lesssim
    e^{G_1(\bu,\tF)}
    H_1(\bu,\tF),
  \end{align}
  where
  $
  G_1(\bu,\tF)\coloneq
  \tF
  +
  \norm{L^1(\bL^\infty)}{\GRAD \bu}
  +
  h \gamma_k \norm{L^1(\bL^\infty)}{\bu}
  +
  h\widetilde{\delta}_{0k}
  \norm{L^2(\bL^\infty)}{\GRAD \bu}^2
  $ and
  \begin{align*}
    H_1(\bu,\tF)
    \coloneq&
    \nu
    h^{(2k+2)}
    \norm{L^2(\bH^{k+2}(\Th))}{\bu}^2
    +
    h^{(2k+1)}
    \norm{L^1(\bW^{1,\infty})}{\bu}
    \left(
    \norm{L^\infty(\bH^{k+1}(\Th))}{\bu}^2
    +
    \gamma_k
    \norm{L^\infty(\bH^{k+2}(\Th))}{\bu}^2
    \right)
    \\
    &
    +
    \widetilde{\delta}_{0k}
    h^{(2k+1)}
    \norm{L^2(\bL^{\infty})}{\bu}^2
    \norm{L^\infty(\bH^{k+1}(\Th))}{\bu}^2
    +
    h^{(2k+2)}
    \norm{L^2(\bH^{k+1}(\Th))}{d_t\bu}^2.
  \end{align*}
\end{corollary}

\begin{remark}[No dependency on $\nu^{-1}$ or $p$]\label{rem:nuinverse}
  The right-hand sides of \eqref{ineq:main.theo} and \eqref{ineq:main.cor} do not depend on $\nu^{-1}$ or on the pressure $p$, thus we have indeed a Reynolds-semi-robust and pressure-robust method.
\end{remark}

\begin{remark}[Formulation using only $\boldsymbol{H}_{\text{div}}(\Ths)$]
  One may wonder why we do not use a formulation completely based on a discrete conforming subspace of $\Hdiv[\Ths]$. 
The main reason is that the functions in this space are not fully continuous at interfaces, hence jump penalization for the discretization of the viscous term would be required for stability.
    This would, in turn, result in a stronger coupling among the degrees of freedom of neighbouring elements and, thus, in a larger stencil.
\end{remark}
\begin{remark}[Estimate including the upwind-norm]\label{rem:upwindnorm}
Following the same steps in the proof of Theorem \ref{error.vel.theorem},
and similar steps in the proof of Lemma \ref{lem:fI.1.vel},
we can  include the upwind norm in the left hand side of the inequalities \eqref{ineq:main.theo} and \eqref{ineq:main.cor}, and  have the following estimate  instead
  \begin{align*}
    \norm{L^\infty(\norm{R,h}{\cdot})}{\uline{\be}_h}^2
    +
   \int _0^{\tF} \Big( \nu  \norm{1,h}{\uline{\be}_h}^2
     + 
    \frac{1-\epsilon}{2} \sum_{\sigma\in\Fhs^{\rm i}}\int_\sigma |\bR_h^k\uline{\bu}_h\cdot\bn_\sigma| |\llbracket \bR_h^k\uline{\be}_h \rrbracket_\sigma |^2\Big)
\lesssim
    e^{G_1(\bu,\tF)}
    H_1(\bu,\tF),
  \end{align*}
where $\epsilon$ is a real number such that $0<\epsilon\leq 1$. 
 \end{remark}
\begin{remark}[Error in the $L^\infty(\bL^2)$-norm]\label{rem:LinftyL2norm}
  Using the inequality \eqref{eq:L2.hho<=RT} for $T\in \Th$, we obtain an order of convergence of $k+\frac{1}{2}$ in the $L^\infty(0,\tF;\bL^2(\Omega))$-norm for the velocity  in \eqref{ineq:main.cor}.
  This order of convergence equals, e.g., the one obtained in \cite[Corollary 5.7]{HanHou:2021} on simplicial meshes.
  However, in Corollary \ref{cor.vel.convergence} above, we do not need to require dominating convection, whereas in \cite{HanHou:2021} the condition $\nu\lesssim h$ is necessary.
  Specifically, observe that in the definition \eqref{def:H_u} of $H(\bu)$ in Theorem \ref{error.vel.theorem} includes the consistency error \eqref{eq:ns:ah:consistency} which, under the regularity conditions stated in Corollary \ref{cor.vel.convergence},  gives an order of convergence of $k+1$,
  thus
  no further assumptions on $\nu$ are needed, in order to make
  the  error $\norm{L^\infty(\norm{R,h}{\cdot})}{\uline{\be}_h}$ in \eqref{ineq:main.cor}  of order $k+\frac{1}{2}$.
  In contrast,  the consistency error
  of the diffusion term using $\boldsymbol{H}_{\text{div}}$-conforming discontinous Galerkin elements, such as the ones used in \cite{HanHou:2021}, is of order $k$,
  which increases to $k+\frac{1}{2}$ assuming dominating convection,
  thus making this assumption essential to obtain a convergence error of $k+\frac{1}{2}$ for the overall numerical scheme.
\end{remark}

\begin{proof}[Proof of Theorem \ref{error.vel.theorem}]
  Recalling that $\uline{\be}_h = \uline{\bu}_h - \hat{\uline{\bu}}_h$ (cf. \eqref{def:error.split}), we take $(\uline{\bv}_h,q_h)=(\uline{\be}_h,0)$ in  \eqref{eq:nstokes:discrete} and, after subtracting the quantity $(a_{R,h}(d_t \hhointerph{\bu},\uline{\be}_h) + \nu a_h(\hhointerph{\bu},\uline{\be}_h))$
  from both sides, we get, for almost every $t \in (0,\tF$),
  \begin{multline*}
    a_{R,h}(d_t \uline{\be}_h,\uline{\be}_h) 
    + \nu a_h(\uline{\be}_h,\uline{\be}_h)
    =
    \ell_h(\bef,\uline{\be}_h)
    - t_h(\uline{\bu}_h,\uline{\bu}_h,\uline{\be}_h)
    - b_h(\uline{\be}_h,p_h)\\
    - a_{R,h}(d_t \hhointerph{\bu},\uline{\be}_h) 
    -\nu a_h(\hhointerph{\bu},\uline{\be}_h).
  \end{multline*}
  Recalling that $\bef= d_t\bu -\nu\LAP \bu + (\bu \cdot \GRAD )\bu  + \nabla p$ almost everywhere in $(0,\tF)\times\Omega$, we go on writing
    \[ %% \begin{equation}\label{eq:ve.fI_1.fI_2.fI_3.fI_4}
    \begin{aligned}
      a_{R,h}(d_t \uline{\be}_h,\uline{\be}_h)
      &=
      \underbrace{
        \left(
        \ell_h((\bu \cdot \GRAD ) \bu,\uline{\be}_h)
        - t_h(\uline{\bu}_h,\uline{\bu}_h,\uline{\be}_h)
        \right)
      }_{\frakI_1}
      +
      \underbrace{
        \left(
        \ell_h(d_t\bu,\uline{\be}_h)
        -a_{R,h}(d_t \hhointerph{\bu},\uline{\be}_h)
        \right)
      }_{\frakI_2}
      \\
      &\quad
      - \underbrace{
        \left(
        \nu\ell_h(\LAP\bu,\uline{\be}_h)
        +
        \nu a_h(\hhointerph{\bu},\uline{\be}_h)
        +
        \nu a_h(\uline{\be}_h,\uline{\be}_h)
        \right)
      }_{\frakI_3}
      +
      \underbrace{
        \left(
        \ell_h(\GRAD p,\uline{\be}_h)
        -
        b_h(\uline{\be}_h,p_h)	
        \right).
      }_{\frakI_4}
    \end{aligned}
    \] %% \end{equation}
    We next bound the terms in the right-hand side.
    \medskip\\
    \underline{(i) \emph{Estimate of $\frakI_1$.}}
    As we will see, the first term is the most difficult to estimate, so the details of the following bound are provided in the separate Lemma \ref{lem:fI.1.vel} below:
    \begin{equation}\label{eq:ve.fI_1.final}
      \frakI_1
      \lesssim
      \sum_{T\in \Th}  \norm{{\tR,T}}{\uline{\be}_T}^2
      \left(\seminorm{\bW^{1,\infty}(T)}{\bu}+ h_T\gamma_k\norm{\Linftyd[T]}{\bu}
      +
      h_T\widetilde{\delta}_{0k}
      \seminorm{\bW^{1,\infty}(T)}{\bu}^2
      \right) +
      \frakN_1.
    \end{equation}
    \\
    \underline{(ii) \emph{Estimate of $\frakI_2$.}}
    Let us now prove that
    \begin{equation}\label{eq:ve.fI_2.final}
      \frakI_2
      \lesssim
      \sum_{T\in\Th}
      \|\uline{\be}_T \|_{\tR,T}^2
      + \frakN_2.
    \end{equation}%
    Using the definition of the $L^2$-orthogonal projections $\bpi_T^{k^\star}$ and $\bpi _F^k$, it can bee seen that $d_t \hhointerpT{\bu} = \underline{\bI}_T^k(d_t\bu) \eqcolon \widehat{\underline{d_t\bu}}_T$ for all $T\in\Th$ (assuming that $d_t\bu: (0,\tF)\to \Hund[T]$) and, by the definitions \eqref{def:aRT}--\eqref{def:deltasR} of $a_{\tR,T}$ and \eqref{eq:lh:form} of the linear form $\ell_h$, we get
    \begin{equation}\label{eq:ve.fI_2}
      \begin{aligned}
        \frakI_2
        &=
        \sum_{T\in\Th}
        \int_T 
        \big(
        {d_t}\bu -\bR_T^k(\widehat{\underline{d_t\bu}}_T)
        \big) \cdot \bR_T^k\uline{{\be}}_T
        \\
        &\quad
        -
        \sum_{T\in\Th}
        \int_T 
        \bpi_T^{k^\star}\big(
        \bR_T^k(\widehat{\underline{d_t\bu}}_T)- (\widehat{{d_t\bu}})_T
        \big) \cdot \bpi_T^{k^\star}(\bR_T^k\uline{\be}_T- \be_T  )
        \\
        &\quad
        -
        \sum_{T\in\Th}\sum_{F\in \calF_T}
        h_F\int_F
        \bpi_F^{k}\big(
        \bR_T^k(\widehat{\underline{d_t\bu}}_T)- (\widehat{{d_t\bu}})_F
        \big)\cdot \bpi_F^{k}(\bR_T^k\uline{{\be}}_T- \be_F  )
        \eqcolon
        \frakI_{2,1} + \frakI_{2,2}+ \frakI_{2,3}.
      \end{aligned}
    \end{equation}
    To bound $\frakI_{2,1}$ we use  Cauchy--Schwarz and Young inequalities followed by the definition \eqref{def:normRT} of $\norm{\tR,T}{{\cdot}}$ to write
    \[
    \frakI_{2,1}
    \lesssim 
    \sum_{T\in\Th}
    \norm{\Ldeuxd[T]}{ d_t \bu -\bR_T^k(\widehat{\underline{d_t\bu}}_T)}^2 
    + \sum_{T\in\Th} \|\uline{\be}_T \|_{\tR,T}^2
      \overset{\eqref{def:ve.fI_2.tilde}}\le
      \frakN_2 + \sum_{T\in\Th} \|\uline{\be}_T \|_{\tR,T}^2.
    \]
    We bound $\frakI_{2,2}$ in \eqref{eq:ve.fI_2}  using similar steps along with the $\bL^2$-boundedness of $\bpi_T^{k^\star}$, \eqref{def:normRT}, and \eqref{eq:L2.hho<=RT} to write $\norm{\bL^2(T)}{\bR_T^k\uline{\be}_T - \be_T}\le \norm{\bL^2(T)}{\bR_T^k\uline{\be}_T} + \norm{\bL^2(T)}{\be_T} \lesssim \norm{\tR,T}{\uline{\be}_T}$, and then adding and subtracting $d_t \bu$ and using \eqref{eq.l1l2.bound} with $n=2$ to infer
    \[
    \begin{aligned}
      \frakI_{2,2}
      &\lesssim 
      \sum_{T\in\Th}
      \left(
      \norm{\Ldeuxd[T]}{ \bR_T^k(\widehat{\underline{d_t\bu}}_T) - d_t \bu}^2 
      + \norm{\Ldeuxd[T]}{d_t \bu- (\widehat{{d_t\bu}})_T)}^2
      \right)
      + \sum_{T\in\Th} \|\uline{\be}_T \|_{\tR,T}^2
      \\
      \overset{\eqref{def:ve.fI_2.tilde}}&\le
        \frakN_2 + \sum_{T\in\Th} \|\uline{\be}_T \|_{\tR,T}^2.
    \end{aligned}
    \]
    The term $\frakI_{2,3}$ in \eqref{eq:ve.fI_2} is treated similarly to obtain
    \[
    \frakI_{2,3} 
    \lesssim
    \sum_{T\in\Th}
    \sum_{F\in \calF_T}
    h_F\norm{\Ldeuxd[F]}{ \bR_T^k\uline{\be}_T - \be_F}^2
    + \sum_{T\in\Th} \sum_{F\in \calF_T} h_F\norm{\Ldeuxd[F]}{ \bR_T^k(\widehat{\underline{d_t\bu}}_T) 
      - (\widehat{{d_t\bu}})_F}^2
    \eqcolon \frakI_{2,3,1} + \frakI_{2,3,2}.
    \]
    By definition \eqref{def:normRT} of the $\norm{\tR,T}{{\cdot}}$-norm, we have
    \[
    \frakI_{2,3,1}\leq\sum_{T\in\Th} \|\uline{\be}_T \|_{\tR,T}^2.
    \]
    To bound  $\frakI_{2,3,2}$, we use repeatedly \eqref{eq.l1l2.bound} with $n=2$, the discrete trace inequality \eqref{eq:discrete.trace} below valid for piecewise polynomial functions, and the bound \eqref{ineq:card.IT.F} for $\card(\Fh[T])$ as follows:
    \begin{align*}
      &\frakI_{2,3,2}
      \\
      &\quad\lesssim
      \sum_{T\in\Th}
      \left(
      \sum_{F\in \calF_T}
      h_F\norm{\Ldeuxd[F]}{ \bR_T^k(\widehat{\underline{d_t\bu}}_T) 
        - (\widehat{{d_t\bu}})_T}^2
      +
      \sum_{F\in \calF_T}
      h_F\norm{\Ldeuxd[F]}{(\widehat{{d_t\bu}})_T
        - (\widehat{{d_t\bu}})_F}^2
      \right)
      \\
      &\quad\lesssim
      \sum_{T\in\Th}
      \left[
        \norm{\Ldeuxd[T]}{ \bR_T^k(\widehat{\underline{d_t\bu}}_T) 
          - (\widehat{{d_t\bu}})_T}^2
        +
        \sum_{F\in \calF_T}
        \left(
        h_F\norm{\Ldeuxd[F]}{(\widehat{{d_t\bu}})_T
          - d_t\bu}^2
        +
        h_F\norm{\Ldeuxd[F]}{d_t\bu
          - (\widehat{{d_t\bu}})_F}^2
        \right)
        \right]\\
      &\quad\lesssim
      \sum_{T\in\Th}
      \left(
      \norm{\Ldeuxd[T]}{ \bR_T^k(\widehat{\underline{d_t\bu}}_T) 
        - {{d_t\bu}}}^2
      +
      \norm{\Ldeuxd[T]}{ {d_t\bu} 
        - (\widehat{{d_t\bu}})_T}^2
      +
      h_T^2\norm{\Ldeuxd[T]}{\GRAD((\widehat{{d_t\bu}})_T
        - d_t\bu)}^2
      \right)\\
      &\qquad
      +
      \sum_{T\in\Th}
      \sum_{F\in \calF_T}
      h_F\norm{\Ldeuxd[F]}{d_t\bu
        - (\widehat{{d_t\bu}})_F}^2
      \overset{\eqref{def:ve.fI_2.tilde}}=\frakN_2,
    \end{align*}
    where, in the last bound, we have used a continuous trace inequality.
    Plugging this last estimate along with those obtained before for $\frakI_{2,1}, \frakI_{2,2}$, and $\frakI_{2,3,1}$ into \eqref{eq:ve.fI_2} concludes the proof of \eqref{eq:ve.fI_2.final}.
    \medskip\\
    \underline{(iii) \emph{Estimate of $\frakI_3$.}}
    We next show that
    \begin{equation}\label{eq:ve.fI_3.final}
      \frakI_3
      \leq C {H(\bu)}
%%       \nu 
%%       \left(
%%       \sum_{T\in\Th} h_T^2\norm{\Ldeux[T]}{\LAP{\bu} -\bpi_T^{k-1}\LAP{\bu} }^2
%%       + \|{\cal{E}}_{a,h}(\bu;\cdot)\|_{1,h,\corr{\star}{*}{[DDP]}}^2
%%       \right)
      - \frac{1}{2}\nu C_a \norm{{1,h}}{\uline{\be}_h}^2,
    \end{equation}
    where $C$ is a constant independent of $\nu$ and $h$ and $C_a$ is the coercivity constant of $a_h$ (cf. \eqref{eq:ns:stab.h}).
    We use  the definition \eqref{eq:lh:form} of $\ell_h$ with $\bphi = \LAP \bu$ and add and subtract $\nu\int_\Omega\LAP {\bu}\cdot {\be}_h$ to get
    \begin{align*}
      \frakI_3& = 
      - \nu\int_\Omega \LAP {\bu}\cdot (\bR_h^k\underline{\be}_h -{\be}_h )
      - \nu\int_\Omega \LAP {\bu}\cdot {\be}_h
      - \nu a_h(\hhointerph{\bu},\uline{\be}_h)
      -\nu a_h(\uline{\be}_h,\uline{\be}_h)
      \notag\\
      \overset{\eqref{eq:Eh.a}}&=
      - \nu\sum_{T\in\Th}
      \int_T
      \cancel{ \bpi_T^{k-1}\LAP{\bu}\cdot (\bR_T^k\uline{\be}_T- {\be}_T ) }
      - \nu \sum_{T\in\Th}\int_T (\LAP{\bu} -\bpi_T^{k-1}\LAP{\bu} )\cdot (\bR_T^k\underline{\be}_T -{\be}_T)
      \notag
      \\
      &\quad
      + \mathcal{E}_{a,h}(\bu;\uline{\be}_h)
        %% \left(
%%       - \nu\int_\Omega \LAP {\bu}\cdot {\be_h}
%%       - \nu a_h(\hhointerph{\bu},\uline{\be}_h)
%%       \right)
      -\nu a_h(\uline{\be}_h,\uline{\be}_h)
      \notag
      \\
      &\le
      \nu\sum_{T\in\Th}
      h_T \norm{\bL^2(T)}{\LAP{\bu} - \bpi_T^{k-1}\LAP{\bu}} \norm{1,T}{\uline{\be}_T}
      + \norm{1,h,{*}}{{\cal{E}}_{a,h}(\bu;\cdot)} \norm{1,h}{\uline{\be}_h}
      - \nu C_a \norm{1,h}{\uline{\be}_h}^2
      \\
      &\leq
      C\nu
      \left(
      \sum_{T\in\Th}
      h_T^2\norm{\bL^2(T)}{\LAP{\bu} -\bpi_T^{k-1}\LAP{\bu} }^2
      + \norm{1,h,{*}}{{\cal{E}}_{a,h}(\bu;\cdot)}^2
      \right)
      -\frac{1}{2}\nu C_a
      \norm{1,h}{\uline{\be}_h}^2,
    \end{align*}
      where we have used \eqref{ineq:rtn:consis} in the cancellation,
      Cauchy--Schwarz inequalities along with \eqref{ineq:rtn:bound} for the first term, the definition \eqref{eq:dual.norm} of $\norm{1,h,*}{{\cdot}}$ for the second term, and the coercivity \eqref{eq:ns:stab.h} for the last term, while the conclusion follows using the generalized Young inequality 
        \begin{equation}\label{ineq.young.eps}
          ab\leq \epsilon a^2 + \frac{1}{4\epsilon} b^2
        \end{equation}
      for the first and second term with $\epsilon$ selected so as to make the contribution $\frac12C_a\norm{1,h}{\uline{\be}_h}^2$ appear.
    Observe that the constant $C$ appearing in the last inequality is indeed independent of $\nu$ and $h$ and recall the definition \eqref{def:H_u} of $H(\bu)$ to obtain \eqref{eq:ve.fI_3.final}.
    \medskip\\
    \underline{(iv) \emph{Estimate of $\frakI_4$.}}
    We finally show that
    \begin{equation}\label{eq:ve.fI_4.final}
      \frakI_4
      =0.
    \end{equation}
    Using the definitions \eqref{def:b_h} of $b_h$ and \eqref{eq:lh:form} of $\ell_h$, we get
    \begin{equation*}
      \frakI_4
      =
      \sum _{T\in \calT_h}\left\{ \int_T  \GRAD p \cdot \bR_T^k\uline{\be}_T
      +\int_T   D_T^k \uline{\be}_T p_T \right\}
      =
      \sum _{T\in \calT_h}\left \{- \int_T    p \cancel{(\DIV \bR_T^k\uline{\be}_T)}
      + \int_T \cancel{ D_T^k \uline{\be}_T} p_T \right\}
      =0,
    \end{equation*}
    where, in the second step, we have integrated by parts element by element the first term and used the fact that, for all $\sigma\in\Fhs^{\rm i}$, $\llbracket p\rrbracket_\sigma = 0$ (because $p: (0,\tF)\to \Hun\cap P$) and $\llbracket \bR_h^k\uline{\be}_h\rrbracket_\sigma \cdot \bn_\sigma = 0$, as well as $\bR_h^k\uline{\be}_h \cdot \bn_\sigma = 0$ on $\partial \Omega$ (consequence of \eqref{eq:darcyT:weak:bd}),
    while, to conclude, we have used the fact that $\DIV \bR_T^k\uline{\be}_T=D_T^k \uline{\be}_T$ by \eqref{eq:darcyT:weak:a} 
      and $D_T^k \uline{\be}_T \overset{\eqref{def:error.split}}= D_T^k \uline{\bu}_T - D_T^k \hat{\uline{\bu}}_T = 0$, the conclusion being a consequence of \eqref{eq:nstokes:discrete} (test with $(\uline{\boldsymbol{0}},q_h)$ and let $q_h$ span $\Poly{k}(\Th)$ to infer $D_T^k \uline{\bu}_T = 0$ for all $T \in \Th$) and of \eqref{eq:DT.commuting} along with $\DIV\bu = 0$ (which give $D_T^k \hat{\uline{\bu}}_T = \pi_T^k ( \DIV\bu ) = 0$ for all $T \in \Th$).
    \medskip\\
    \underline{(v) \emph{Conclusion.}}
    Using the fact that $\frac{1}{2} {d_t}\norm{R,h}{\uline{\be}_h}^2=a_{R,h}(d_t \uline{\be}_h,\uline{\be}_h)$ and gathering the estimates \eqref{eq:ve.fI_1.final}, \eqref{eq:ve.fI_2.final}, \eqref{eq:ve.fI_3.final}, and \eqref{eq:ve.fI_4.final}, we obtain
    \[
    \begin{aligned}
      &\frac{1}{2}\frac{d}{dt}\norm{R,h}{\uline{\be}_h}^2
      +
      \frac{1}{2}\nu C_a
      \norm{{1,h}}{\uline{\be}_h}^2
      \\
      &\quad\lesssim
      \sum_{T\in \Th}  \norm{{\tR,T}}{\uline{\be}_T}^2
      \left( 1 + \seminorm{\bW^{1,\infty}(T)}{\bu}+ h_T\gamma_k\norm{\Linftyd[T]}{\bu}
      +
      h_T\widetilde{\delta}_{0k}
      \seminorm{\bW^{1,\infty}(T)}{\bu}^2
      \right)
      \\
      &\qquad +
      \nu\left(
      \sum_{T\in\Th}
      h_T^2\norm{\Ldeux[T]}{\LAP{\bu} -\bpi_T^{k-1}\LAP{\bu} }^2
      +
      \|{\cal{E}}_{a,h}(\bu;\cdot)\|_{1,h,{*}}^2
      \right)
      + \frakN_1
      + \frakN_2.
    \end{aligned}
    \]
    Applying a Gronwall inequality (see, e.g., \cite[Lemma 6.9]{Ern.Guermond:04}) and observing that $\uline{\bu}_h(0)=\uline{\bI}_h^k\bu_0$ implies $\uline{\be}_h(0)=\bzero$ concludes the proof of \eqref{ineq:main.theo}.
\end{proof}

The rest of this section is devoted to the proof of \eqref{eq:ve.fI_1.final}.

\begin{lemma}[Estimate of $\frakI_1$]\label{lem:fI.1.vel}
  Estimate \eqref{eq:ve.fI_1.final} holds.
\end{lemma}

\begin{proof} 
  We begin by expanding $\ell_h$ and $t_h$ according to the respective definitions \eqref{eq:lh:form} and \eqref{eq:th}, then add and subtract $\sum_{T \in\Th} \int_T (\bR_T^k\hhointerpT{\bu} \cdot \GRAD ) \bu \cdot \bR_T^k\uline{\be}_T +\sum_{T \in\Th} \int_T  (\bR_T^k\uline{\bu}_T \cdot \GRAD)  \bu \cdot \bR_T^k\uline{\be}_T$ and recall the definitions \eqref{def:error.split} of $\uline{\be}_h$ and \eqref{def:bbeta.T} of $\bbeta_T$ to write
  \begin{equation}\label{eq:ve.fI_1}
    \begin{aligned}
      \frakI_1
      &= \ell_h((\bu \cdot \GRAD ) \bu,\uline{\be}_h)
      - t_h(\uline{\bu}_h,\uline{\bu}_h,\uline{\be}_h)
      \\      
      &=
      -\sum_{T \in\Th}\int_T (\bbeta_T \cdot \GRAD ) \bu \cdot \bR_T^k\uline{\be}_T 
      - \sum_{T \in\Th}\int_T ( \bR_T^k \uline{\be}_T\cdot \GRAD ) \bu   \cdot \bR_T^k\uline{\be}_T
      \\
      &\quad
      -\sum_{T \in\Th} \int_T  (\bR_T^k\uline{\bu}_T \cdot \GRAD) (\bR_T^k\uline{\bu}_T - \bu) \cdot \bR_T^k\uline{\be}_T
      \\
      &\quad
      + \sum_{\sigma\in\Fhs^{\rm i}}\int_\sigma (\bR_h^k\uline{\bu}_h\cdot\bn_\sigma) \llbracket \bR_h^k\uline{\bu}_h - \bu\rrbracket_\sigma \cdot \lbrace \bR_h^k\uline{\be}_h\rbrace_\sigma
      \\
      &\quad
      - \sum_{\sigma\in\Fhs^{\rm i}}\int_\sigma  \frac{1}{2} |\bR_h^k\uline{\bu}_h\cdot\bn_\sigma| \llbracket \bR_h^k\uline{\bu}_h- \bu\rrbracket_\sigma \cdot \llbracket \bR_h^k\uline{\be}_h\rrbracket_\sigma
      - \sum_{T \in\Th}\sum_{\sigma\in\Fhs[T]^{\rm i}} t_{T,\sigma}^k(\uline{\bu}_T,\uline{\bu}_T,\uline{\be}_T)
      \\
      &\eqcolon 
      \frakI_{1,1} + \frakI_{1,2} + \frakI_{1,3}+ \frakI_{1,4} + \frakI_{1,5} + \frakI_{1,6},
    \end{aligned}
  \end{equation}
  where we have additionally used the fact that $\llbracket \bu\rrbracket_\sigma=\bzero$ for all $\sigma\in\Fhs^{\rm i}$ to insert this quantity into the fourth and fifth terms.
  \medskip\\
  \underline{(i) \emph{Estimate of $\frakI_{1,1} + \frakI_{1,2}$.}}
  Using a H\"{o}lder inequality with exponents $(2,\infty,2)$ followed by Young's inequality, we obtain
  \[
    \frakI_{1,1}
    \lesssim
    \sum _{T\in \Th} \seminorm{\bW^{1,\infty}(T)}{\bu} \norm{\bL^{2}(T)}{\bbeta_T}^2
    + \sum _{T\in \Th} \seminorm{\bW^{1,\infty}(T)}{\bu} \norm{{\tR,T}}{\uline{\be}_T}^2
    \]
  where we have additionally used the definition \eqref{def:normRT} of $\norm{\tR,T}{{\cdot}}$ in the second addend.
  We bound $|\frakI_{1,2}|$ using a H\"{o}lder inequality with exponents $(2,\infty,2)$ to get
  \[
    \frakI_{1,2}\leq
    \sum _{T\in \Th}
    \seminorm{\bW^{1,\infty}(T)}{\bu}
    \norm{{\tR,T}}{\uline{\be}_T}^2.
  \]
    Gathering the above bounds and recalling the definition \eqref{def:ve.fI_1.tilde} of $\frakN_1$, we obtain
    \begin{equation}\label{eq:ve.fI_11_12}
      \frakI_{1,1} + \frakI_{1,2}
      \lesssim\frakN_1 
      + \sum _{T\in \Th} \seminorm{\bW^{1,\infty}(T)}{\bu} \norm{{\tR,T}}{\uline{\be}_T}^2.
    \end{equation}
  \medskip\\
  \underline{(ii) \emph{Estimate of $\frakI_{1,3} + \frakI_{1,4} + \frakI_{1,5} + \frakI_{1,6}$.}}
  Using an element-by-element integration by parts along with the fact that $\DIV \bR_T^k\uline{\bu}_T=0$ for all $T\in \Th$ 
(use $(\uline{\boldsymbol{0}},q_h)$ with $q_h \in \Poly{k}(\Th)$ in \eqref{eq:nstokes:discrete} and recall the definition \eqref{def:b_h} of $b_h$ to infer $D_T^k \uline{\bu}_T = 0$ for all $T \in \Th$, and plug this result into \eqref{eq:darcyT:weak:a} to conclude),
  we get
  \[
  \begin{aligned}
    \frakI_{1,3}+ \frakI_{1,4} 
    &=
    \sum_{T \in\Th} \int_T  (\bR_T^k\uline{\bu}_T \cdot \GRAD) \bR_T^k\uline{\be}_T  \cdot (\bR_T^k\uline{\bu}_T - \bu)\\
    &\quad
    - \sum_{\sigma\in\Fhs^{\rm i}}\int_\sigma (\bR_h^k\uline{\bu}_h\cdot\bn_\sigma) \llbracket\bR_h^k\uline{\be}_h \rrbracket_\sigma \cdot \lbrace \bR_h^k\uline{\bu}_h - \bu  \rbrace_\sigma.
  \end{aligned}
  \]
  Adding and subtracting
  \[
  \sum_{T \in\Th}   \int_T  (\bR_T^k\uline{\bu}_T \cdot \GRAD) \bR_T^k\uline{\be}_T  \cdot \bR_T^k\hhointerpT{\bu} + \sum_{\sigma\in\Fhs^{\rm i}}\int_\sigma (\bR_h^k\uline{\bu}_h\cdot\bn_\sigma) \llbracket\bR_h^k\uline{\be}_h \rrbracket_\sigma \cdot \lbrace   \bR_h^k\hat{\uline{\bu}}_h \rbrace_\sigma,
  \]
  recalling the definitions \eqref{def:error.split} of $\uline{\be}_h$ and \eqref{def:bbeta.T} of $\bbeta_h$, the following is obtained:
  \begin{align*}
    \frakI_{1,3}+ \frakI_{1,4} 
    =
    &
    \cancel{\sum_{T \in\Th}  \int_T  (\bR_T^k\uline{\bu}_T \cdot \GRAD) \bR_T^k\uline{\be}_T  \cdot \bR_T^k\uline{\be}_T}
    - \sum_{\sigma\in\Fhs^{\rm i}}\int_\sigma (\bR_h^k\uline{\bu}_h\cdot\bn_\sigma) \llbracket\bR_h^k\uline{\be}_h \rrbracket_\sigma \cdot \lbrace   \bbeta_h \rbrace_\sigma\\
    &
    \hspace{-8pt} 
    -  \cancel{\sum_{\sigma\in\Fhs^{\rm i}}\int_\sigma (\bR_h^k\uline{\bu}_h\cdot\bn_\sigma) \llbracket\bR_h^k\uline{\be}_h \rrbracket_\sigma \cdot \lbrace   \bR_h^k\uline{\be}_h \rbrace_\sigma}
    + \sum_{T \in\Th}   \int_T  (\bR_T^k\uline{\bu}_T \cdot \GRAD) \bR_T^k\uline{\be}_T  \cdot \bbeta_T,
  \end{align*}
  where the first and third terms are cancelled using the integration by parts formula \cite[Eq.~(2)]{Botti.Di-Pietro.ea:19} and the fact that $\DIV \bR_T^k\uline{\be}_T=0$.
  Summing $\frakI_{1,5}$ to the previous expression, adding and subtracting $\frac{1}{2}\sum_{\sigma\in\Fhs^{\rm i}}\int_\sigma |\bR_h^k\uline{\bu}_h\cdot\bn_\sigma| \llbracket \bR_h^k\hat{\uline{\bu}}_h\rrbracket_\sigma \cdot \llbracket \bR_h^k\uline{\be}_h\rrbracket_\sigma$, and rearranging, it is inferred that
  \[
    \begin{aligned}
      &\frakI_{1,3}+ \frakI_{1,4} + \frakI_{1,5}
      \\
      &\quad=
      \sum_{T \in\Th}   \int_T  (\bR_T^k\uline{\bu}_T \cdot \GRAD) \bR_T^k\uline{\be}_T  \cdot \bbeta_T
      - \sum_{\sigma\in\Fhs^{\rm i}}\int_\sigma (\bR_h^k\uline{\bu}_h\cdot\bn_\sigma) \llbracket\bR_h^k\uline{\be}_h \rrbracket_\sigma \cdot \lbrace   \bbeta_h \rbrace_\sigma
      \\
      &\qquad
      - \frac{1}{2} \sum_{\sigma\in\Fhs^{\rm i}}\int_\sigma  |\bR_h^k\uline{\bu}_h\cdot\bn_\sigma| |\llbracket \bR_h^k\uline{\be}_h \rrbracket_\sigma |^2
      - \frac{1}{2} \sum_{\sigma\in\Fhs^{\rm i}}\int_\sigma |\bR_h^k\uline{\bu}_h\cdot\bn_\sigma| \llbracket \bbeta_h \rrbracket_\sigma \cdot \llbracket \bR_h^k\uline{\be}_h\rrbracket_\sigma
      \\
      &\quad\leq
      \sum_{T \in\Th}  \int_T (\bR_T^k\uline{\bu}_T \cdot \GRAD) \bR_T^k\uline{\be}_T  \cdot \bbeta_T
      + \sum_{\sigma\in\Fhs^{\rm i}}\int_\sigma |\bR_h^k\uline{\bu}_h\cdot\bn_\sigma |\left| \llbracket\bR_h^k\uline{\be}_h \rrbracket_\sigma \cdot \lbrace   \bbeta_h \rbrace_\sigma\right|
      \\
      &\qquad
      - \frac{1}{2} \sum_{\sigma\in\Fhs^{\rm i}}\int_\sigma |\bR_h^k\uline{\bu}_h\cdot\bn_\sigma| |\llbracket \bR_h^k\uline{\be}_h \rrbracket_\sigma |^2
      + \frac{1}{2} \sum_{\sigma\in\Fhs^{\rm i}}\int_\sigma |\bR_h^k\uline{\bu}_h\cdot\bn_\sigma| \left | \llbracket \bbeta_h \rrbracket_\sigma \cdot \llbracket \bR_h^k\uline{\be}_h\rrbracket_\sigma\right |.
    \end{aligned}
  \]
  Using the generalized Young inequality \eqref{ineq.young.eps}  for the second and fourth terms with $\epsilon$ respectively equal to $1$ and $\frac12$, so that the third term cancels out, we finally get
  \begin{equation}\label{eq:ve.fI_13.fI_14.fI_15.2}
    \begin{aligned}
      \frakI_{1,3}+ \frakI_{1,4} + \frakI_{1,5}
      &\le
      \sum_{T \in\Th}  \int_T  (\bR_T^k\uline{\bu}_T \cdot \GRAD) \bR_T^k\uline{\be}_T  \cdot \bbeta_T
      + \sum_{\sigma\in\Fhs^{\rm i}}\int_\sigma |\bR_h^k\uline{\bu}_h\cdot\bn_\sigma ||\lbrace   \bbeta_h \rbrace_\sigma|^2
      \\
      &\quad
      + {\frac14} \sum_{\sigma\in\Fhs^{\rm i}}\int_\sigma  |\bR_h^k\uline{\bu}_h\cdot\bn_\sigma| | \llbracket \bbeta_h \rrbracket_\sigma |^2,
    \end{aligned}
  \end{equation}
  Denote by $\Th[\sigma]\subset\Th$ the set collecting the (one or two) mesh elements that share $\sigma\in\Fhs^{\rm i}$.
  To bound the second term in the right hand side of \eqref{eq:ve.fI_13.fI_14.fI_15.2}, we recall the definition $\eqref{def:error.split}$ of $\uline{\be}_h$ and use a triangle inequality to write
  \[
    \begin{aligned}
      &\int_\sigma |\bR_h^k\uline{\bu}_h\cdot\bn_\sigma | |\lbrace \bbeta_h \rbrace_\sigma|^2
      \\
      &\quad\leq
      \int_\sigma |\bR_h^k\uline{\be}_h\cdot\bn_\sigma | |\lbrace \bbeta_h \rbrace_\sigma|^2
      +
      \int_\sigma |\bR_h^k\uline{\hat{\bu}}_h\cdot\bn_\sigma | |\lbrace \bbeta_h \rbrace_\sigma|^2
      \\
      &\quad\lesssim
      \norm{\bL^{\infty}(\Th[\sigma])}{\bbeta_h}
        \norm{\bL^2(\sigma)}{\bR_h^k\uline{\be}_h{\cdot\bn_\sigma}}
        \norm{\bL^2(\sigma)}{\lbrace \bbeta_h \rbrace_\sigma}
      +
      \norm{\bL^{\infty}(\Th[\sigma])}{\bR_h^k\uline{\hat{\bu}}_h}
      \norm{\bL^2(\sigma)}{\lbrace \bbeta_h \rbrace_\sigma}^2
      \\
      &\quad\lesssim
      h_\sigma\seminorm{\bW^{1,\infty}(\Th[\sigma])}{\bu}
      \norm{{L^2}(\sigma) }{\bR_h^k\uline{\be}_h{\cdot\bn_\sigma}}^2
      +
      \norm{\bW^{1,\infty}(\Th[\sigma])}{\bu}
      \norm{\bL^2(\sigma) }{\lbrace \bbeta_h \rbrace_\sigma}^2,
    \end{aligned}
  \]
  where we have used H\"{o}lder inequalities with exponents {$(\infty,2,2)$} along with the fact that $\| \bn_\sigma\|_{\bL^\infty(\sigma)} \leq 1$
  in the second step and Young's inequality for the first term together with the approximation properties \eqref{ineq:rtn:approx.Wsp} of $\bR_{T}^k\circ\uline{\bI}_{T}^k$ with $(m,p,s)=(0,\infty,1)$ and the fact that $h_{T}\lesssim 1$ for $T\in \Th[\sigma]$ in the third step.
 We continue using a discrete trace inequality on $\sigma$ to write 
 $\norm{L^2(\sigma)}{\bR_h^k\uline{\be}_h\cdot\bn_\sigma}^2
 \lesssim h_\sigma^{-1} \norm{\bL^2(\Th[\sigma])}{\bR_h^k\uline{\be}_h}^2
 \overset{\eqref{def:normRT}} \lesssim h_\sigma^{-1} \norm{\tR,\Th[\sigma]}{\uline{\be}_h}^2$
 and a continuous trace inequality followed by $h_\sigma^{-1} \lesssim h_T^{-1}$ (consequence of mesh regularity) to infer
  \begin{equation}\label{ineq.etaT.avg.sigma}
    \begin{aligned}
      \norm{\bL^2(\sigma) }{\lbrace   \bbeta_{T_\sigma} \rbrace_\sigma}^2
      \overset{\eqref{eq:trace.op},\,\eqref{eq.l1l2.bound}}
      &\lesssim
      \sum_{\tau\in\Ths[\sigma]}
      \norm{\bL^2(\sigma) }{(\bbeta_h)_{|_{\tau}}}^2
      \\
      &\lesssim
      \sum_{\tau\in\Ths[\sigma]}
      \left(
      h_\sigma^{-1}\norm{\Ldeuxd[\tau]}{ \bbeta_{\tau}}^2
      + h_\sigma \seminorm{\bH^{1}(\tau)}{ \bbeta_{\tau}}^2%
      \right)
      \\
      &\lesssim
      \sum_{T\in\Th[\sigma]}\left(
      h_T^{-1}\norm{\Ldeuxd[T]}{\bbeta_T}^2
      + h_T \seminorm{\bH^{1}(\Ths[T])}{ \bbeta_T }^2
      \right).
    \end{aligned}
  \end{equation}
  This gives
  \[
  \int_\sigma |\bR_h^k\uline{\bu}_h\cdot\bn_\sigma | |\lbrace \bbeta_h \rbrace_\sigma|^2
  \lesssim
  \seminorm{\bW^{1,\infty}(\Th[\sigma])}{\bu}\norm{\tR,\Th[\sigma]}{\uline{\be}_h}^2
  + \sum_{T\in\Th[\sigma]}{\norm{\bW^{1,\infty}(T)}{\bu}}\left(
  h_T^{-1}\norm{\Ldeuxd[T]}{\bbeta_T}^2
  + h_T \seminorm{\bH^{1}(\Ths[T])}{ \bbeta_T }^2
  \right).
  \]
  Plugging this bound into
  \eqref{eq:ve.fI_13.fI_14.fI_15.2}, observing that $\card(\Fhsi[T]) +\card(\Fh[T])\lesssim \card(\TTs)\lesssim 1$ by \eqref{ineq:card.IT.F}, and  bounding the last term in the right hand side of \eqref{eq:ve.fI_13.fI_14.fI_15.2} similarly as we just did with the second term, it is seen that
  \begin{equation}\label{eq:ve.fI_13.fI_14.fI_15.3}
    \begin{aligned}
      \frakI_{1,3}+ \frakI_{1,4} + \frakI_{1,5} 
      &\lesssim
      \sum_{T \in\Th}   \int_T  (\bR_T^k\uline{\bu}_T \cdot \GRAD) \bR_T^k\uline{\be}_T  \cdot \bbeta_T
      + \sum_{T \in\Th}
      \seminorm{\bW^{1,\infty}(T)}{\bu}
      \norm{\tR,T}{\uline{\be}_{T}}^2
      \\
      &\quad+
      \sum_{T \in\Th}
      \norm{\bW^{1,\infty}(T)}{\bu}
      \left(
      h_{T}^{-1}\norm{\Ldeuxd[T]}{ \bbeta_{T}}^2
      + h_{T} \seminorm{\bH^{1}(\TTs)}{ \bbeta_{T}}^2%
      \right).
    \end{aligned}
  \end{equation}
  We 
    use the following bound, proved at the end of this section, for the first term in the right hand side of \eqref{eq:ve.fI_13.fI_14.fI_15.3}:
  \begin{equation}\label{ineq:RtRtEta.bound}
    \begin{aligned}
      \left |  \int_T   (\bR_T^k\uline{\bu}_T \cdot \GRAD) \bR_T^k\uline{\be}_T  \cdot \bbeta_T \right|
      &\lesssim \norm{{\tR,T}}{\uline{\be}_T}^2 \left (
      \seminorm{\bW^{1,\infty}(T)}{\bu} + h_T\gamma_k\norm{\Linftyd[T]}{\bu}
      \right)
      + \seminorm{\bW^{1,\infty}(T)}{\bu} \norm{\bL^{2}(T)}{{\bbeta}_T}^2
      \\
      &\quad+ \left(
      \seminorm{\bW^{1,\infty}(T)}{\bu}
      + \gamma_{k} h_T^{-3}\norm{\Linftyd[T]}{\bu}
      \right)
      \norm{\bL^{2}(T)}{\hhointerpTT{\bu} - \bu}^2
      \\
      &\quad+
      \widetilde{\delta}_{0k}
      \sum_{\sigma\in\Fhs[T]^{\rm i}}
      \norm{L^2(\sigma)}
           { 
             \llbracket
             \potOp{k}(({\bu}^0_T\cdot \GRAD) \bR_T^k\uline{\be}_T)
             \rrbracket _\sigma 
           }
           \norm{L^2(\sigma)}{ \bbeta_T\cdot\bn_\sigma },
    \end{aligned}
  \end{equation}
  to finally get
  \begin{equation}\label{eq:ve.fI_13.fI_14.fI_15.4}
    \begin{aligned}
      \frakI_{1,3}+ \frakI_{1,4} + \frakI_{1,5} 
      &\lesssim
      \sum_{T\in \Th} \left[ \norm{{\tR,T}}{\uline{\be}_T}^2
        \left(
        \seminorm{\bW^{1,\infty}(T)}{\bu}+ h_T\gamma_k\norm{\Linftyd[T]}{\bu}
        \right)
      +
      \seminorm{\bW^{1,\infty}(T)}{\bu}
      \norm{\bL^{2}(T)}{{\bbeta}_T}^2
      \right]
      \\
      &\quad +
      \sum_{T\in \Th}\left(
        \seminorm{\bW^{1,\infty}(T)}{\bu}
        + \gamma_{k} h_T^{-3} \norm{\Linftyd[T]}{\bu}
        \right)
      \norm{\Ldeuxd[T]}{\hhointerpTT{\bu}-\bu}^2
      \\
      &\quad +
      \sum_{T\in \Th}
      \norm{\bW^{1,\infty}(T)}{\bu}
      \left(
      h_{T}^{-1}\norm{\Ldeuxd[T]}{ \bbeta_{T}}^2
      + h_{T}\seminorm{\bH^{1}(\TTs)}{ \bbeta_{T}}^2
      \right)
      \\
      &\quad +
      \widetilde{\delta}_{0k}
      \sum_{T\in \Th} \sum_{\sigma\in\Fhs[T]^{\rm i}}
      \norm{L^2(\sigma)}{ \llbracket        
        \potOp{k}
        (({\bu}^0_T\cdot \GRAD) \bR_T^k\uline{\be}_T)
        \rrbracket _\sigma }
      \norm{L^2(\sigma)}{\bbeta_T\cdot\bn_\sigma}.
    \end{aligned}
  \end{equation}

  To bound $\frakI_{1,6}$ in \eqref{eq:ve.fI_1}, we use the definition \eqref{def:error.split} of $\uline{\be}_h$ and the {linearity} of $t_{T,\sigma}^k$ in its second argument to write
  \begin{equation}\label{eq:ve.fI_6}
    \begin{aligned}
      \frakI_{1,6}
      &=
      -\sum_{T\in \Th} \sum_{\sigma\in\Fhs[T]^{\rm i}}
      t_{T,\sigma}^k( \uline{\bu}_T,\uline{\bu}_T,\uline{\be}_T)
      =
      -\sum_{T\in \Th} \sum_{\sigma\in\Fhs[T]^{\rm i}}
      \left(
      t_{T,\sigma}^k(\uline{\bu}_T,\uline{\hat{\bu}}_T,\uline{\be}_T)
      + t_{T,\sigma}^k(\uline{\bu}_T,\uline{\be}_T,\uline{\be}_T)
      \right)
      \\
      &\leq
      \sum_{T\in \Th} \sum_{\sigma\in\Fhs[T]^{\rm i}}
      \left(
      |t_{T,\sigma}^k(\uline{\bu}_T,\uline{\hat{\bu}}_T,\uline{\be}_T)|
      -t_{T,\sigma}^k(\uline{\bu}_T,\uline{\be}_T,\uline{\be}_T)
      \right)
      \\
      &\leq
      \frac{\widetilde{\delta}_{0k}}2\sum_{T\in \Th} \sum_{\sigma\in\Fhs[T]^{\rm i}}
      \norm{L^2(\sigma)}{%
          \llbracket
          \potOp{k} 
          (({\bu}_T^0 \cdot \GRAD) \bR_T^k\uline{\hat{\bu}}_T)
          \rrbracket_\sigma
        }^2
      \\
      &\qquad
      \boxed{%
        - \frac{\widetilde{\delta}_{0k}}2\sum_{T\in \Th} \sum_{\sigma\in\Fhs[T]^{\rm i}} \norm{L^2(\sigma)}{%
          \llbracket \potOp{k}(({\bu}_T^0 \cdot \GRAD) \bR_T^k\uline{\be}_T) \rrbracket_\sigma
        }^2,
      }
    \end{aligned}
  \end{equation}
  where, in the last step, we have used the definitions \eqref{eq:tT} of $t_{T,\sigma}^k$ and \eqref{eq:tilde.delta} of $\widetilde{\delta}_{0k}$, Young's inequality,
  and the fact that 
  $\bR_T^k\uline{\hat{\bu}}_T$ and $\bR_T^k\uline{\be}_T$ both belong to $\Polyd{k}(\TTs)$ (as proved in the item {(iv.b)} below),
  thus $\bpi_{\TTs}^{k-1}(({\bu}_T^0 \cdot \GRAD) \bR_T^k\uline{\hat{\bu}}_T) =({\bu}_T^0 \cdot \GRAD) \bR_T^k\uline{\hat{\bu}}_T$,
  and $\bpi_{\TTs}^{k-1}(({\bu}_T^0 \cdot \GRAD) \bR_T^k\uline{\be}_T)=({\bu}_T^0 \cdot \GRAD) \bR_T^k\uline{\be}_T$.
  Using the generalized Young inequality \eqref{ineq.young.eps} in the last term of \eqref{eq:ve.fI_13.fI_14.fI_15.4} with $\epsilon$ selected so as to compensate the boxed term in \eqref{eq:ve.fI_6} and summing the resulting relation to \eqref{eq:ve.fI_6}, we obtain
  \begin{equation}\label{eq:ve.fI_13.fI_14.fI_15.fI_16}
    \begin{aligned}
      &\frakI_{1,3}+ \frakI_{1,4} + \frakI_{1,5} + \frakI_{1,6}
      \\
      &\quad\lesssim
      \sum_{T\in \Th}\left[
      \norm{{\tR,T}}{\uline{\be}_T}^2
      \left(
      \seminorm{\bW^{1,\infty}(T)}{\bu}+ h_T\gamma_k\norm{\Linftyd[T]}{\bu}
      \right)
      + \seminorm{\bW^{1,\infty}(T)}{\bu} \norm{\bL^{2}(T)}{{\bbeta}_T}^2
      \right]
      \\
      &\qquad+
      \sum_{T\in \Th}\left( 
      \seminorm{\bW^{1,\infty}(T)}{\bu}
      + \gamma_{k}h_T^{-3} \norm{\Linftyd[T]}{\bu}
      \right)
      \norm{\Ldeuxd[T]}{\hhointerpTT{\bu}-\bu}^2
      \\
      &\qquad+
      \sum_{T\in \Th}
      \norm{\bW^{1,\infty}(T)}{\bu}
      \left(
      h_{T}^{-1}\norm{\Ldeuxd[T]}{ \bbeta_{T}}^2
      +
      h_{T}\seminorm{\bH^{1}(\TTs)}{ \bbeta_{T}}^2
      \right)
      \\
      &\qquad+
      \widetilde{\delta}_{0k}
      \sum_{T\in \Th} \sum_{\sigma\in\Fhs[T]^{\rm i}}
      \Big[
        \underbrace{\norm{L^2(\sigma)}{ \bbeta_T\cdot\bn_\sigma }^2}_{\eqcolon\mathfrak{A}(\sigma)}
        +
        \underbrace{%
          \norm{L^2(\sigma)}{%
            \llbracket
            \potOp{k}
            (({\bu}^0_T\cdot \GRAD) \bR_T^k\uline{\hat{\bu}}_T)
            \rrbracket _\sigma }^2
        }_{\eqcolon\mathfrak{B}(\sigma)}
        \Big].
    \end{aligned}
  \end{equation}

  We now proceed to bound the last term in the the right-hand side of \eqref{eq:ve.fI_13.fI_14.fI_15.fI_16}.
  We bound the first contribution as we did  with $\norm{\bL^2(\sigma) }{\lbrace   \bbeta_{T_\sigma} \rbrace_\sigma}^2$ in \eqref{ineq.etaT.avg.sigma}, that is,
  \begin{align}\label{ineq.etaT.sigma}
    \mathfrak{A}(\sigma)
    \lesssim
    {\sum_{T\in\Th[\sigma]}}\left(
    h_{T}^{-1}\norm{\Ldeuxd[T]}{ \bbeta_{T}}^2
    +
    h_{T}\seminorm{\bH^{1}(\Ths[T])}{ \bbeta_{T}}^2
    \right).
  \end{align}

  To bound $\mathfrak{B}(\sigma)$ in  \eqref{eq:ve.fI_13.fI_14.fI_15.fI_16}, observe that, since $(\bu_T^0 \cdot   \GRAD)\bpi_T^k{\bu} \in \Polyd{k-1}(T)$, using \eqref{lem:rhoOp.e}, we have $\potOp{{k}}((\bu_T^0 \cdot   \GRAD)\bpi_T^k{\bu} )\in  \Polyd{k}(T)$, and then
  ${\llbracket\potOp{k} (\bu_T^0 \cdot   \GRAD)\bpi_T^k{\bu}) \rrbracket_\sigma}\equiv 0$ for 
    all simplicial faces $\sigma\in\Fhs[T]^{\rm i}$ internal to any mesh element $T \in \Th$.
 Therefore, denoting by $\tau_1$ and $\tau_2$ the simplices in $\TTs$ that share $\sigma\in\Fhs[T]^{\rm i}$,
   it is inferred, letting for the sake of brevity $\varrho^{k}_{\tau_i} \coloneq (\varrho^{k}_{\TTs})_{|\tau_i}$,
   \[
   \begin{aligned}
    %% &\norm{L^2(\sigma)}{%
%%       \llbracket  \potOp{k}(({\bu}_T^0 \cdot   \GRAD) \bR_T^k\uline{\hat{\bu}}_T) \rrbracket_\sigma
%%      }^2
     %%      \\
     {\mathfrak{B}(\sigma)}
     &=
    \norm{L^2(\sigma)}{%
      \llbracket   \potOp{k}(({\bu}_T^0 \cdot   \GRAD) (\bR_T^k\uline{\hat{\bu}}_T - {\bpi_T^k{\bu}})) \rrbracket_\sigma
    }^2
    \\
    \overset{\eqref{eq:trace.op},\,\eqref{eq.l1l2.bound}}&\lesssim
    \sum_{i=1}^2
    \norm{L^2(\sigma)}{{\varrho^{k}_{\tau_i}} (({\bu}_T^0 \cdot \GRAD)( \bR_T^k\uline{\hat{\bu}}_T -\bpi_T^k{\bu}))}^2
    \\    
    &\lesssim
    \sum_{i=1}^2
    h_{\tau_i}^{-1}
    \norm{L^2(\tau_i)}{ \varrho^k_{\tau_i}(({\bu}_T^0 \cdot \GRAD)( \bR_T^k\uline{\hat{\bu}}_T- \bpi_T^k{\bu} ))}^2
    \\
    &\lesssim
    h_T^{-1}
    \norm{L^2(T)}{ \potOp{k}(({\bu}_T^0 \cdot \GRAD)( \bR_T^k\uline{\hat{\bu}}_T- \bpi_T^k{\bu} ))}^2
    \\
    \overset{\eqref{lem:rhoOp.b}}&\lesssim
    h_{T}
    \norm{\bL^2(T)}{ (({\bu}_T^0 \cdot \GRAD)( \bR_T^k\uline{\hat{\bu}}_T- \bpi_T^k{\bu} ))}^2
    \\
    \overset{\eqref{def:error.split},\,\eqref{eq.l1l2.bound}}&\lesssim
    h_{T}
    \norm{\bL^2(T)}{{\be}_T^0 \cdot \GRAD( \bR_T^k\uline{\hat{\bu}}_T- \bpi_T^k{\bu} )}^2
    +
    h_{T}
    \norm{\bL^2(T)}{\hat{\bu}_T^0 \cdot \GRAD( \bR_T^k\uline{\hat{\bu}}_T- \bpi_T^k{\bu} )}^2,
  \end{aligned}
   \]
   where we have used a discrete trace inequality in the third step and
   the fact that $h_{\tau_i}^{-1} \lesssim h_T^{-1}$ by mesh regularity along with $\tau_1 \cup \tau_2 \subset T$ in the fourth step.
   We continue adding and subtracting $\bu$ and using \eqref{eq.l1l2.bound} with $n=2$ to write
   \begin{equation}\label{eq:est.norm.jump.potOp}
     \begin{aligned}
       %% &\norm{L^2(\sigma)}{%
%%          \llbracket  \potOp{k}(({\bu}_T^0 \cdot   \GRAD) \bR_T^k\uline{\hat{\bu}}_T) \rrbracket_\sigma
%%        }^2
%%        \\
%%        &\quad
       {\mathfrak{B}(\sigma)}
       &\lesssim
       h_{T} \norm{\bL^2(T)}{{\be}_T^0 \cdot \GRAD( \bR_T^k\uline{\hat{\bu}}_T- \bu )}^2
       +
       h_{T} \norm{\bL^2(T)}{{\be}_T^0 \cdot \GRAD( \bu- \bpi_T^k{\bu} )}^2
       \\
       &\quad+  h_{T}
       \norm{\bL^2(T)}{\hat{\bu}_T^0 \cdot \GRAD( \bR_T^k\uline{\hat{\bu}}_T- \bu )}^2
       +
       h_{T}
       \norm{\bL^2(T)}{\hat{\bu}_T^0  \cdot \GRAD( \bu- \bpi_T^k{\bu} )}^2
       \\
       &\lesssim
       h_{T}
       \norm{\bL^\infty(T)}{ \GRAD( \bR_T^k\uline{\hat{\bu}}_T- \bu )}
       \norm{\bL^2(T)}{{\be}_T}^2
       +
       h_{T}
       \norm{\bL^\infty(T)}{ \GRAD( \bu- \bpi_T^k{\bu} )}^2
       \norm{\bL^2(T)}{{\be}_T}^2
       \\
       &\quad
       +  h_{T}
       \norm{\bL^\infty(T)}{\hat{\bu}_T }^2
       \norm{\bL^2(T)}{\GRAD( \bR_T^k\uline{\hat{\bu}}_T- \bu )}^2
       +
       h_{T}
       \norm{\bL^\infty(T)}{\hat{\bu}_T }^2
       \norm{\bL^2(T)}{ \GRAD( \bu- \bpi_T^k{\bu} )}^2
       \\
       &\lesssim
       h_{T}
       \seminorm{\bW^{1,\infty}(T)}{\bu}^2
       \norm{{\rm R},T}{\uline{\be}_T }^2
       + h_{T}
       \norm{\bL^{\infty}(T)}{\bu}^2
       \seminorm{\bH^1(T)}{\bbeta_T}
       + h_{T}
       \norm{\bL^{\infty}(T)}{\bu}^2
       \norm{\bL^2(T)}{\GRAD( \bu- \bpi_T^k{\bu} )}^2,
     \end{aligned}
   \end{equation}
   where in the second step we have used H\"{o}lder inequalities with exponents $(2,\infty)$
   for the first two terms and with exponents $(\infty,2)$ for last two terms,
   then the $\bL^2$-boundedness of $\bpi_T^0$  in all terms.
   In the last step we have used: for the first term the approximation properties \eqref{ineq:rtn:approx.Wsp} of $\bR_T^k\circ\uline{\bI}_T^k$ with $(m,p,s)=(1,\infty,1)$ along with \eqref{eq:L2.hho<=RT};
   for the second term, the {$\bH^1$-boundedness of $\bpi_T^k$} and again \eqref{eq:L2.hho<=RT};
   for the third term, the definition \eqref{def:bbeta.T} of $\bbeta_T$,
   and invoked the $\bL^\infty$-boundedness of $\bpi_T$  for the third and fourth terms.
   \medskip\\
   \underline{(iii) \emph{Estimate of $\frakI_1$.}}
   Plugging \eqref{eq:est.norm.jump.potOp} along with  \eqref{ineq.etaT.sigma} into \eqref{eq:ve.fI_13.fI_14.fI_15.fI_16}, then using the fact that $\card(\Fhsi[T]) \lesssim \card(\TTs)\lesssim 1$ for $T\in \Th$ by \eqref{ineq:card.IT.F}, and using additionally the bound \eqref{eq:ve.fI_11_12}, we obtain in \eqref{eq:ve.fI_1}
  \[
    \frakI_1
    \lesssim
    \sum_{T\in \Th}  \norm{{\tR,T}}{\uline{\be}_T}^2
    \left(\seminorm{\bW^{1,\infty}(T)}{\bu}+ h_T\gamma_k\norm{\Linftyd[T]}{\bu}
    +
    h_T\widetilde{\delta}_{0k}
    \seminorm{\bW^{1,\infty}(T)}{\bu}^2
    \right) +
    \frakN_1,
  \]
  which is precisely \eqref{eq:ve.fI_1.final}.
  \medskip\\
  \underline{(iv) \emph{Proof of \eqref{ineq:RtRtEta.bound}.}}
  To conclude the proof {of Lemma \ref{lem:fI.1.vel}}, it only remains to prove \eqref{ineq:RtRtEta.bound}, which we do next.
  We begin using a triangle inequality as follows:
  \begin{multline}\label{ineq:rtk.bound.proof.1}
    \left |\int_T  (\bR_T^k\uline{\bu}_T \cdot \GRAD) \bR_T^k\uline{\be}_T  \cdot \bbeta_T \right|
    \\
    \leq
    \left |
    \int_T  \left((\bR_T^k\uline{\bu}_T -{\bu}_T)  \cdot \GRAD\right) \bR_T^k\uline{\be}_T  \cdot \bbeta_T 
    \right|
    +
    \left |
    \int_T  ({\bu}_T  \cdot \GRAD) \bR_T^k\uline{\be}_T  \cdot \bbeta_T 
    \right|
    \eqcolon
    \frakT_1+ \frakT_2.
  \end{multline}
  \underline{(iv.a) \emph{Estimate of $\frakT_1$.}}
  Adding the terms $\pm(\bR_T^k\hat{\uline{\bu}}_T - \hat{{\bu}}_T)$ in the inner parentheses of $\frakT_1$,  then
  recalling that $\uline{\be}_h = \uline{\bu}_h - \hat{\uline{\bu}}_h$ (cf. \eqref{def:error.split})
  and using a triangle inequality, we get
  \[ %% \begin{equation}\label{ineq:rtk.bound.proof.frakT_1}
    \frakT_1
    \leq
    \left |
    \int_T  \left((\bR_T^k\uline{\be}_T -{\be}_T)  \cdot \GRAD\right) \bR_T^k\uline{\be}_T  \cdot \bbeta_T 
    \right|
    +
    \left |
    \int_T  \left((\bR_T^k\hhointerpT{\bu} -\hhointerpTT{\bu})  \cdot \GRAD\right) \bR_T^k\uline{\be}_T  \cdot \bbeta_T 
    \right|
    \eqcolon
    \frakT_{1,1} +\frakT_{1,2}.
    \] %% \end{equation}
  To bound $\frakT_{1,1}$, we use a H\"{o}lder inequality with exponents $(2, 2, \infty)$ to obtain
  \begin{align*}
    \frakT_{1,1}
    &\leq
    \norm{\bL^{2}(T)}{\bR_T^k\uline{\be}_T -{\be}_T}
    \norm{\bL^{2}(T)}{\GRAD\bR_T^k\uline{\be}_T}
    \norm{\bL^{\infty}(T)}{{\bbeta}_T}\\
    &\lesssim
    h_T^2
    \norm{1,T}{\uline{\be}_T}
    \norm{\bL^{2}(T)}{\GRAD\bR_T^k\uline{\be}_T}
    \seminorm{\bW^{1,\infty}(T)}{\bu},
    %\lesssim
    %\norm{\bL^{2}(T)}{\bR_T^k\uline{\be}_T}^2
    %.\seminorm{\bW^{1,\infty}(T)}{\bu},
  \end{align*}
  where, in the second step, we have used \eqref{ineq:rtn:bound} and \eqref{ineq:rtn:approx.Wsp} with $(m,p,s)=(0,\infty,1)$ for the first and third factors, respectively, after recalling, for the latter, that $\bbeta_T = \bR_T^k \hat{\uline{\bu}}_T - \bu$ by
  \eqref{def:bbeta.T}.
    To proceed, we bound $\norm{1,T}{\uline{\be}_T}$ with \eqref{eq:L2.hho<=RT}, then use
  the inverse inequality \eqref{ineq:RT.inverse} along with the definition \eqref{def:normRT} of $\norm{\tR,T}{{\cdot}}$ to write $\norm{\bL^{2}(T)}{\GRAD\bR_T^k\uline{\be}_T} \lesssim h_T^{-1}\norm{\Ldeuxd[T]}{\bR_T^k\uline{\bv}_T} \le h_T^{-1} \norm{\tR,T}{\uline{\be}_T}$. Thus, we get
  \[
  \frakT_{1,1}
  \lesssim
  \norm{{\tR,T}}{\uline{\be}_T}^2
  \seminorm{\bW^{1,\infty}(T)}{\bu}.
  \]
  We bound $\frakT_{1,2}$
  using a H\"{o}lder inequality with exponents $(2, 2, \infty)$,
    then using the inverse inequality \eqref{ineq:RT.inverse} along with the definition \eqref{def:normRT} of $\norm{\tR,T}{{\cdot}}$ for the second factor,
    and recalling that $\bbeta_T = \bR_T^k \hat{\uline{\bu}}_T - \bu$ by \eqref{def:bbeta.T} to use the approximation properties \eqref{ineq:rtn:approx.Wsp} of $\bR_T^k\circ\underline{\bI}_T^k$ with $(m,p,s)=(0,\infty,1)$ for the third factor, thus obtaining
  \[
    \begin{aligned}
      \frakT_{1,2}
      &\lesssim
      \norm{\bL^{2}(T)}{\bR_T^k\hhointerpT{\bu} -\hhointerpTT{\bu}}
      \norm{{\tR,T}}{\uline{\be}_T}
      \seminorm{\bW^{1,\infty}(T)}{\bu}\\
      &\lesssim
      \seminorm{\bW^{1,\infty}(T)}{\bu}
      \norm{\bL^{2}(T)}{\bR_T^k\hhointerpT{\bu} -\hhointerpTT{\bu}}^2
      +
      \seminorm{\bW^{1,\infty}(T)}{\bu}
      \norm{{\tR,T}}{\uline{\be}_T}^2
      \\&\lesssim
      \seminorm{\bW^{1,\infty}(T)}{\bu}
      \left(
      \norm{\bL^{2}(T)}{\bbeta_T}^2
      + \norm{\bL^{2}(T)}{\bu -\hhointerpTT{\bu}}^2
      \right)
      + \seminorm{\bW^{1,\infty}(T)}{\bu}
      \norm{{\tR,T}}{\uline{\be}_T}^2,
    \end{aligned}
  \]
  where in the second step we have used Young's inequality while, in the third step, we have added $\pm \bu$ into the second term, used \eqref{eq.l1l2.bound} with $n=2$, and recalled \eqref{def:bbeta.T}.
    Gathering the above bounds for $\frakT_{1,1}$ and $\frakT_{1,2}$, we obtain
    \begin{equation}\label{eq:RtRtEta.T1}
      \frakT_1 \lesssim \seminorm{\bW^{1,\infty}(T)}{\bu}
      \left(
      \norm{\bL^{2}(T)}{\bbeta_T}^2
      + \norm{\bL^{2}(T)}{\bu -\hhointerpTT{\bu}}^2
      \right)
      + \seminorm{\bW^{1,\infty}(T)}{\bu}
      \norm{{\tR,T}}{\uline{\be}_T}^2.
    \end{equation}
  \medskip\\
  \noindent\underline{(iv.b) \emph{Estimate of $\frakT_2$ and conclusion.}}
  To bound $\frakT_2$ in \eqref{ineq:rtk.bound.proof.1}, recalling that ${\bu}_T^0\coloneq  \bpi_T^0{\bu}_T$ and using a triangle inequality, we have that
  \begin{equation}
    \frakT_2
    \leq
    \left | \int_T   (({\bu}_T -{\bu}_T^0) \cdot \GRAD) \bR_T^k\uline{\be}_T ) \cdot \bbeta_T \right |
    +
    \left | \int_T   ({\bu}_T^0\cdot \GRAD) \bR_T^k\uline{\be}_T  \cdot \bbeta_T \right |
    \eqcolon
    \frakT_{2,1} + \frakT_{2,2}.
    \label{ineq:rtk.bound.proof.I2}
  \end{equation}
  To treat $\frakT_{2,1}$, we add and subtract $( \hhointerpTT{\bu} - \hhointerpTT{\bu} ^0 )$,
  use  \eqref{def:error.split},
  and apply a triangle inequality to obtain
  \begin{equation}\label{ineq:rtk.bound.proof.I21}
    \frakT_{2,1}
    \leq
    \left | \int_T   (({\be}_T -{\be}_T^0) \cdot \GRAD) \bR_T^k\uline{\be}_T ) \cdot \bbeta_T \right |
    +
    \left | \int_T   ((\hhointerpTT{\bu} -\hhointerpTT{\bu}^0) \cdot \GRAD) \bR_T^k\uline{\be}_T ) \cdot \bbeta_T \right |
    \eqcolon
    \frakT_{2,1,1} + \frakT_{2,1,2}.
  \end{equation}
  To bound $\frakT_{2,1,1}$, we use a H\"{o}lder inequality with exponents $(2, 2,\infty)$,
  the approximation properties \eqref{eq:l2proj:error:cell} of  $\bpi^0_{T}$ with $(l,m,r,s)=(0,0,2,1)$ followed by the definition \eqref{eq:norm.1T} of $\norm{1,T}{{\cdot}}$ to write $\norm{\bL^2(T)}{\be_T - \be_T^0} \lesssim h_T \norm{\bL^2(T)}{\GRAD \be_T} \le h_T\norm{1,T}{\uline{\be}_T}$,
  the inverse inequality \eqref{ineq:RT.inverse},
  and the approximation properties \eqref{ineq:rtn:approx.Wsp} of $\bR_T^k\circ\underline{\bI}_T^k$ with $(m,p,s)=(0,\infty,1)$ {to write $\norm{\bL^\infty(T)}{\bbeta_T} \lesssim h_T \seminorm{\bW^{1,\infty}(T)}{\bu}$ and thus} get
  \begin{align*}
    \frakT_{2,1,1}
    &\lesssim
    h_T
    \norm{1,T}{\uline{\be}_T}
    \norm{\bL^{2}(T)}{\bR_T^k\uline{\be}_T}
    \seminorm{\bW^{1,\infty}(T)}{\bu}
    \overset{\eqref{eq:L2.hho<=RT}}\lesssim
    \norm{{\tR,T}}{\uline{\be}_T}^2
    \seminorm{\bW^{1,\infty}(T)}{\bu}.
  \end{align*}
  Similarly, using a H\"{o}lder inequality with exponents $(\infty, 2, 2)$ followed by a triangle inequality, we obtain
  \[ %% \begin{align}\label{ineq:rtk.bound.proof.I212}
  \begin{aligned}
    \frakT_{2,1,2}
    &\leq
    \left(\norm{\bL^{\infty}(T)}{\hhointerpTT{\bu} -{\bu}}  
    + \norm{\bL^{\infty}(T)}{{\bu} -\hhointerpTT{\bu}^0}
    \right)
    \norm{\bL^{2}(T)}{\GRAD\bR_T^k\uline{\be}_T}
    \norm{\bL^{2}(T)}{{\bbeta}_T}\nonumber\\
    &\lesssim
    \seminorm{\bW^{1,\infty}(T)}{\bu}
    \norm{\bL^{2}(T)}{\bR_T^k\uline{\be}_T}
    \norm{\bL^{2}(T)}{{\bbeta}_T}
    \lesssim
    \seminorm{\bW^{1,\infty}(T)}{\bu}
    \norm{{\tR,T}}{\uline{\be}_T}^2
    +
    \seminorm{\bW^{1,\infty}(T)}{\bu}
    \norm{\bL^{2}(T)}{{\bbeta}_T}^2,
  \end{aligned}
  \] %% \end{align}
  where, in the second step, we have used the approximation properties \eqref{eq:l2proj:error:cell} of $\bpi^{k^\star}_T$ and $\bpi^0_{T}$ with $(l,m,r,s)=(k^\star,0,\infty,1)$ and
  $(l,m,r,s)=(0,0,\infty,1)$, respectively, and then
  the inverse inequality \eqref{ineq:RT.inverse} while, in the last step,
  we have used the definition \eqref{def:normRT} of $\norm{\tR,T}{{\cdot}}$ and Young's inequality.

  To bound $\frakT_{2,2}$ in \eqref{ineq:rtk.bound.proof.I2}, we use the fact that $\DIV \bu = 0$, the commutation property \eqref{eq:DT.commuting}, and \eqref{eq:darcyT:weak:a} to infer that $\DIV \bR_T^k\hhointerpT{\bu} = 0$.
  In addition, we have $\DIV \bR_T^k\uline{\bu}_T = 0$ 
  (take $\uline{\bv}_h = \uline{\bzero}$ and $q_h \in \Poly{k}(\Th)$ supported in $T$ in  \eqref{eq:nstokes:discrete} and use \eqref{def:b_h} and \eqref{eq:darcyT:weak:a}).
  By \eqref{def:error.split}, the above relations yield $\DIV \bR_T^k\uline{\be}_T=0$; this implies 
  that
  $\bR_T^k\uline{\be}_T \in \Polyd{k}(\TTs)$ (see \cite[Corollary 2.3.1, p. 90]{Boffi.Brezzi.ea:13}),
  thus
  $({\bu}^0_T\cdot \GRAD) \bR_T^k\uline{\be}_T \in \Polyd{k-1}(\TTs)$, and the proof is finished  for $k=0$, since $\bR_T^0\uline{\be}_T \in \Polyd{0}(\TTs)$ and thus $({\bu}^0_T\cdot \GRAD) \bR_T^0\uline{\be}_T = \bzero$, i.e., $\frakT_{2,2} = 0$.
  If {$k\ge 1$}, on the other hand, we recall the definition \eqref{def:rhoOp} of $\potOp{k}$ and observe that
  \[
  \Polyd{k-1}(\TTs)\ni
  ({\bu}^0_T\cdot \GRAD) \bR_T^k\uline{\be}_T
  = \GRAD g + \bc,
  \]
  with $g \coloneq \potOp{k}(({\bu}^0_T\cdot \GRAD) \bR_T^k\uline{\be}_T)$ and
  $\bc \coloneq (\Id - \GRAD\potOp{k})(({\bu}^0_T\cdot \GRAD) \bR_T^k\uline{\be}_T)$.
  Using a triangle inequality, we split $\frakT_{2,2}$ as follows:
  \begin{align}\label{ineq:rtk.bound.proof.I22}
    \frakT_{2,2}
    &\leq
    \left | \int_T \GRAD g 
    \cdot \bbeta_T
    \right |
    +
    \left | \int_T \bc 
    \cdot \bbeta_T
    \right |
    \eqcolon
    \frakT_{2,2,1} + \frakT_{2,2,2}.
  \end{align}
  Applying an integration by parts to the integral contained in $\frakT_{2,2,1}$ and using the fact that $\bbeta_T$ has continuous normal trace across simplicial faces $\sigma \in \Fhs[T]^{\rm i}$, we obtain
  \begin{align*}
    \frakT_{2,2,1}
    &\leq
    \sum_{\sigma\in\Fhs[T]^{\rm i}}
    \left |
    \int_\sigma \llbracket g \rrbracket_\sigma\,(\bbeta_T\cdot \bn_\sigma)
    \right |
    +
    \sum_{F\in \calF_T}\sum_{\sigma  \in \Fhs[F]}
    \left |
    \cancel{\int_\sigma g\,(\bbeta_T\cdot \bn_\sigma)}
    \right |
    +
    \left |
    \cancel{\int_T g\, (\DIV \bbeta_T)}
    \right |\\
    &\leq
    \sum_{\sigma\in\Fhs[T]^{\rm i}}
    \norm{L^2(\sigma)}{ \llbracket g\rrbracket _\sigma }
    \norm{L^2(\sigma)}{\bbeta_T\cdot\bn_\sigma }\\
    &=
    \widetilde{\delta}_{0k}
   \sum_{\sigma\in\Fhs[T]^{\rm i}}
    \norm{L^2(\sigma)}{ \llbracket \potOp{k}(({\bu}^0_T\cdot \GRAD) \bR_T^k\uline{\be}_T)\rrbracket _\sigma }
    \norm{L^2(\sigma)}{ \bbeta_T\cdot\bn_\sigma },
  \end{align*}
  where  the cancellations in the first step follow using
  the fact that $\DIV \bbeta_T=\DIV(\bR_T^k\hhointerpT{\bu} - \bu)=0$ (since $\DIV \bu =\DIV \bR_T^k\hhointerpT{\bu} = 0$) and the boundary condition \eqref{eq:darcyT:weak:bd}, 
    while the second step follows from a Cauchy--Schwarz inequality.
  We now proceed to bound $\frakT_{2,2,2}$ in \eqref{ineq:rtk.bound.proof.I22}.
  By \eqref{lem:rhoOp.d} with $l=k-1$, we have that $\bc\in\Goly{{\rm c},k-1}(\TTs)$ and, recalling  that, by definition, $\Goly{{\rm c},k-1}(\TTs)$  
  is the  trivial space when $k=1$, we get $\frakT_{2,2,2}=0$.
  For $k\ge 2$, on the other hand, recalling the definition \eqref{def:gammak} of $\gamma_k$, the fact that 
$\bbeta_T \coloneq \bR_T^k\hhointerpT{\bu} - \bu$, 
and using the condition \eqref{eq:darcyT:weak:ab}, we obtain
  \begin{equation}\label{ineq:rtk.bound.proof.I222}
    \begin{aligned}
      \frakT_{2,2,2}
      &=
      \gamma_k\left | \int_T\bc\cdot (\hhointerpTT{\bu}-\bu)
      \right |
      \\
      &\leq
      \gamma_k
      \norm{\Ldeuxd[T]}{\bc}
      \norm{\Ldeuxd[T]}{\hhointerpTT{\bu}-\bu}
      \\
      &\lesssim
      \gamma_k
      \norm{\Ldeuxd[T]}{({\bu}^0_T\cdot \GRAD) \bR_T^k\uline{\be}_T)}
      \norm{\Ldeuxd[T]}{\hhointerpTT{\bu}-\bu}
      \\
      &\leq
      \gamma_k
      \left(\norm{\Linftyd[T]}{{\be}^0_T} + \norm{\Linftyd[T]}{\hhointerpTT{\bu}^0}\right)
      \norm{\Ldeuxd[T]}{\GRAD \bR_T^k\uline{\be}_T}
      \norm{\Ldeuxd[T]}{\hhointerpTT{\bu}-\bu}
      \\
      &\lesssim
      \gamma_k
      h_T^{-\frac{d}{2}}\norm{\Ldeuxd[T]}{{\be}_T} 
      \norm{\Ldeuxd[T]}{\GRAD \bR_T^k\uline{\be}_T}
      \norm{\Ldeuxd[T]}{\hhointerpTT{\bu}-\bu}
      \\
      &\qquad+
      \gamma_k\norm{\Linftyd[T]}{\bu}
      h_T^{\frac32}\norm{\Ldeuxd[T]}{\GRAD\bR_T^k\uline{\be}_T}
      h_T^{-\frac32}\norm{\Ldeuxd[T]}{\hhointerpTT{\bu}-\bu},
    \end{aligned}
  \end{equation}
  where we have used the  Cauchy--Schwarz inequality in the second step,
  the fact that 
  $\norm{\Ldeuxd[T]}{(\Id - \GRAD\potOp{k})\bq}
  \leq
  \norm{\Ldeuxd[T]}{\bq}
  +
  \norm{\Ldeuxd[T]}{\GRAD\potOp{k}\bq}
  $
  for $\bq\in\Polyd{k-1}(\TTs)$
  along with the bound \eqref{lem:rhoOp.b} with $l=k-1$
  in the third step,
  a H\"{o}lder inequality with exponents $(\infty,2)$  for the first factor
  along with \eqref{def:error.split} 
  and then a triangle inequality
  in the fourth step,
  and, for the fifth step, we have used the discrete Lebesgue inequality \eqref{eq:rev:disc.emb.pw} with $(\alpha,\beta, X)=(\infty,2, T)$ and the $L^2$-boundeness of $\bpi_T^0$ for the first addend and the $L^\infty$-boundeness of $\bpi_T^0$ for the second addend.
  To further bound the first term in the right-hand side of \eqref{ineq:rtk.bound.proof.I222},  we use the inverse inequality \eqref{ineq:RT.inverse} along with the definition \eqref{def:normRT} of $\norm{{\tR,T}}{{\cdot}}$ to write
  \begin{multline*}
    \gamma_k
    h_T^{-\frac{d}{2}}\norm{\Ldeuxd[T]}{{\be}_T} 
    \norm{\Ldeuxd[T]}{\GRAD \bR_T^k\uline{\be}_T}
    \norm{\Ldeuxd[T]}{\hhointerpTT{\bu}-\bu}
    \\
    \lesssim
    \gamma_k
    h_T^{-\frac{d}{2}-1}
    \norm{\Ldeuxd[T]}{{\be}_T}
    \norm{{\tR,T}}{\uline{\be}_T}
    \norm{\Ldeuxd[T]}{\hhointerpTT{\bu}-\bu}
    \lesssim
    \norm{{\tR,T}}{\uline{\be}_T}^2
    \seminorm{\bW^{1,\infty}(T)}{\bu},
  \end{multline*}
  where, in the last step, we have used the fact that $\gamma_k\leq1$, the inequality \eqref{eq:L2.hho<=RT} for the first factor, and the approximation properties \eqref{eq:l2proj:error:cell} of $\bpi_{T}^{k^\star}$ with $(l,m,r,s)=(k^\star,0,2,1)$ along with
    $\norm{L^2(T)}{\nabla \bu}\lesssim h_T^{\frac{d}{2}} \norm{L^\infty(T)}{\nabla \bu}$
  for the last factor.
  For the second term in the right-hand side of \eqref{ineq:rtk.bound.proof.I222}, we use the inverse inequality \eqref{ineq:RT.inverse} along with the definition \eqref{def:normRT} of $\norm{{\tR,T}}{{\cdot}}$, and then Young's inequality, thus we obtain
  \begin{multline*}
    \gamma_k\norm{\Linftyd[T]}{\bu}
    h_T^{\frac32}\norm{\Ldeuxd[T]}{\GRAD\bR_T^k\uline{\be}_T}
    h_T^{-\frac32}\norm{\Ldeuxd[T]}{\hhointerpTT{\bu}-\bu}
    \\
    \lesssim
    \gamma_k
    h_T\norm{\Linftyd[T]}{\bu}
    \norm{{\tR,T}}{\uline{\be}_T}^2
    +
    \gamma_k
    h_T^{-3}
    \norm{\Linftyd[T]}{\bu}
    \norm{\Ldeuxd[T]}{\hhointerpTT{\bu}-\bu}^2.
  \end{multline*}
  Using the above estimates in \eqref{ineq:rtk.bound.proof.I222}, we get
  \begin{align*}
    \frakT_{2,2,2}
    \lesssim
    \seminorm{\bW^{1,\infty}(T)}{\bu}
    \norm{{\tR,T}}{\uline{\be}_T}^2
    +
    \gamma_k
    h_T\norm{\Linftyd[T]}{\bu}
    \norm{{\tR,T}}{\uline{\be}_T}^2
    +
    \gamma_k
    h_T^{-3}
    \norm{\Linftyd[T]}{\bu}
    \norm{\Ldeuxd[T]}{\hhointerpTT{\bu}-\bu}^2.
  \end{align*}    
  Therefore plugging the bounds for $\frakT_{2,2,1}$ and $\frakT_{2,2,2}$ into \eqref{ineq:rtk.bound.proof.I22}, 
  the bounds $\frakT_{2,1,1}$ and $\frakT_{2,1,2}$ into \eqref{ineq:rtk.bound.proof.I21}, 
  and then using the result in \eqref{ineq:rtk.bound.proof.I2},
  we get
  \[
  \begin{multlined}
    \frakT_2
    \lesssim 
    \seminorm{\bW^{1,\infty}(T)}{\bu}
    \norm{{\tR,T}}{\uline{\be}_T}^2
    +
    \norm{\Linftyd[T]}{\bu}
    \norm{\Ldeuxd[T]}{\bbeta_T}^2
    +
    \gamma_k
    h_T\norm{\Linftyd[T]}{\bu}
    \norm{{\tR,T}}{\uline{\be}_T}^2
    \\
     +
    \gamma_k
    h_T^{-3}
    \norm{\Linftyd[T]}{\bu}
    \norm{\Ldeuxd[T]}{\hhointerpTT{\bu}-\bu}^2  
    +\widetilde{\delta}_{0k}
    \sum_{\sigma\in\Fhs[T]^{\rm i}}
    \norm{L^2(\sigma)}{ \llbracket \potOp{k}(({\bu}^0_T\cdot \GRAD) \bR_T^k\uline{\be}_T)\rrbracket _\sigma }
    \norm{L^2(\sigma)}{ \bbeta_T\cdot\bn_\sigma }.
  \end{multlined}
  \]
  Thus, plugging the above estimate along with \eqref{eq:RtRtEta.T1} into \eqref{ineq:rtk.bound.proof.1}, we conclude the proof of \eqref{ineq:RtRtEta.bound}.
\end{proof}

%------------------------------------------------------------------------------%
%------------------------------------------------------------------------------%
\section{Numerical tests}\label{sec:ntest1}

%------------------------------------------------------------------------------%
In this section we  verify  numerically the proposed method for general meshes of the unit square domain $\Omega = (0,1)^2$. 
The analytical solution $(\bu,p)$ in \eqref{eq:nstokes:strong} is taken from \cite{DeFrutos.ea:2019,HanHou:2021},
 with velocity components and pressure as follows
\begin{align*}
\bu(\bx)& \coloneq \frac{6+4\cos(4t)}{10}\begin{pmatrix}8\sin ^2(\pi x_1)(2x_2(1-x_2)(1-2x_2)) \\ - 8\pi\sin (2 \pi x_1)(x_2(1-x_2))^2\end{pmatrix},\\
p(\bx) &\coloneq \frac{6+4\cos(4t)}{10}\sin (\pi x_1)\cos (\pi x_2).
\end{align*}
For each element $T \in \Th$, we construct its simplicial submesh $\Ths[T]$
by adding an internal node which corresponds to the geometrical center $\bx_T$ of the element $T$ and construct the simplicial mesh $\Ths[T]$ in such a way that all simplices in $\Ths[T]$ have the vertex $\bx_T$ in common (this construction fulfills the assumptions made in Section \ref{sec:setting:mesh}).
We set $k=1$ and consider computations  over three $h$-refined mesh families (Cartesian, hexagonal and Voronoi type). 
Figure \ref{fig:meshes:coarsest} shows the coarsest mesh for each family.
As for the temporal discretization, an implicit/explicit (IMEX) BDF2 scheme
is used, in which
$t_h(2\uline{\bu}_h^{n-1}-\uline{\bu}_h^{n-2},\uline{\bu}_h^n,\uline{\bv}_h)$
is used in the convective term \eqref{eq:th}.
The interpolate of the analytical solution is used in the first two steps as the initialization of the algorithm.
We set as a time step $\Delta t=10^{-3}$, and the final time as $\tF=2$.
  This choice has been numerically verified to make the error in time negligible with respect to the error in space for all the considered meshes.
Our implementation is based on the \texttt{HArDCore} library\footnote{\url{https://github.com/jdroniou/HArDCore}} which makes extensive use of the linear algebra \texttt{Eigen} open-source library (see \url{http://eigen.tuxfamily.org}), and we use the linear direct solver Pardiso \cite{Schenk.Gartner.ea:01}.
We monitor the following quantities in Table \ref{tbl:ntest1}:
$N_{\rm dof}$  denoting the number of discrete unknowns and nonzero entries of the global system;
$\norm{L^\infty(\bL^2)}{\uline{\be}_h}$,
the discrete infinity $L^2(\Omega)$-norm of the velocity  error
$\uline{\be}_h\coloneq\uline{\bu}_h - \uline{\bI}_h^k\bu$;
$\norm{\sharp,\Delta t}{\uline{\be}_h}$, the discrete $L^2$-energy-upwind-norm of the velocity  error defined as follows

$$
\norm{\sharp,\Delta t}{\uline{\be}_h}^2
\coloneq
  \Delta t \sum _{n=2}^{N_{\tF}} \left(\nu  \norm{1,h}{\uline{\be}_h^n}^2
   +
    \frac{1}{2}  \sum_{\sigma\in\Fhs^{\rm i}}\int_\sigma |\bR_h^k(2\uline{\bu}_h^{n-1}-\uline{\bu}_h^{n-2})\cdot\bn_\sigma| |\llbracket \bR_h^k\uline{\be}_h^n \rrbracket_\sigma |^2\right),
$$
where $N_{\tF}$ is the total number of time steps.
The error norms are accompanied by their
corresponding spatial Estimated Order of Convergence (EOC) computed using successive spatial refinement steps.
To confirm the independence of the velocity error on $\nu^{-1}$,
we solve numerically the problem for  $\nu\in\{10^{-2},10^{-4},10^{-6},10^{-10}\}$.
Recalling the Corollary \ref{cor.vel.convergence},
and the Remarks 
\ref{rem:upwindnorm} and \ref{rem:LinftyL2norm},
we expect to get a convergence rate of $1.5$  in both norms (since we have set $k=1$).
The results collected in Table \ref{tbl:ntest1} 
show that the convergence rate of $\norm{L^\infty(\bL^2)}{\uline{\be}_h}$ is always greater than 1.5 for all values of $\nu$, and 
the convergence rate of $\norm{\sharp,\Delta t}{\uline{\be}_h}$
is asymptotically close  to 1.5.
Moreover for small values of $\nu$, the velocity errors hold practically unchanged, that is to say, they are independent of $\nu^{-1}$. This is consistent with our theoretical results, see Theorem \ref{error.vel.theorem} and its Corollary \ref{cor.vel.convergence}.

\begin{figure}[!htb]
  \centering
  \begin{minipage}[b]{0.32\textwidth}
    \centering
    \includegraphics[width=0.80\textwidth]{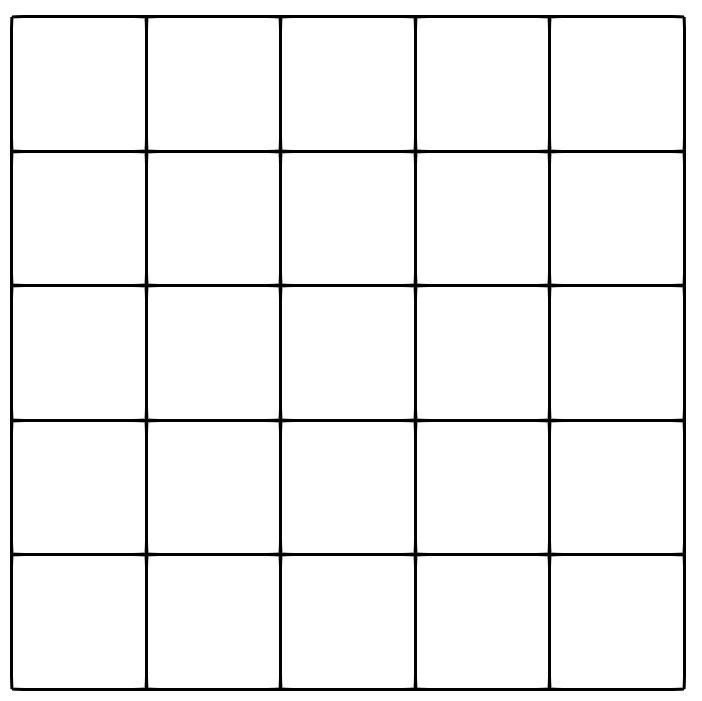} 
    \subcaption{Cartesian.}
  \end{minipage}%
  \begin{minipage}[b]{0.32\textwidth}
    \centering
    \includegraphics[width=0.80\textwidth]{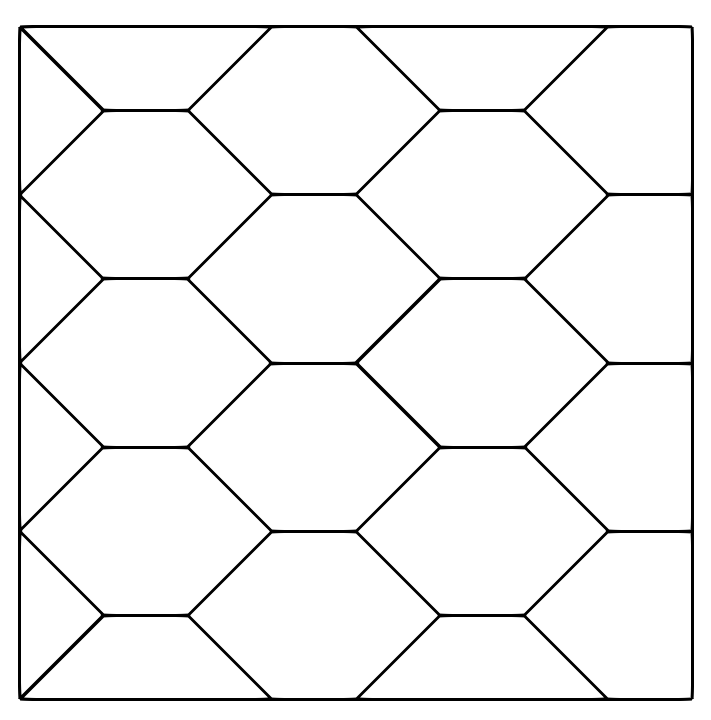} 
    \subcaption{Hexagonal.}
  \end{minipage}%
  \begin{minipage}[b]{0.32\textwidth}
  \centering
  \includegraphics[width=0.80\textwidth]{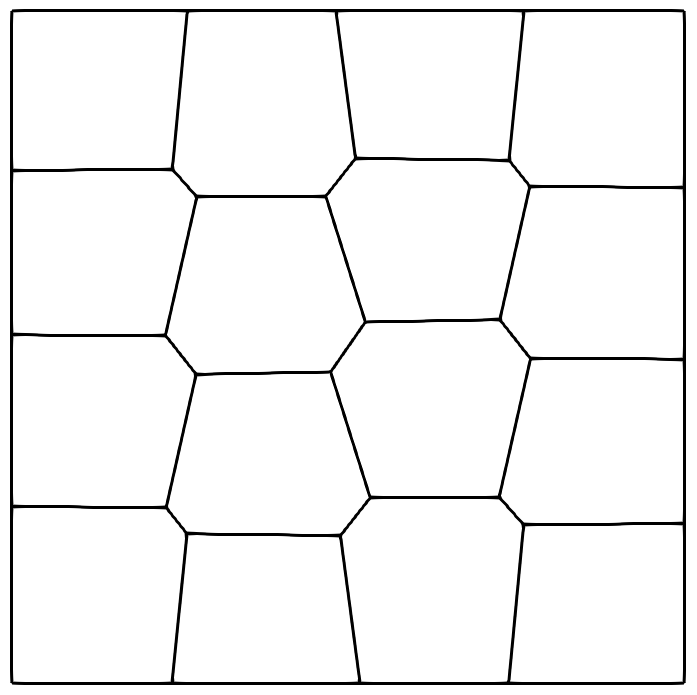} 
  \subcaption{Voronoi.}
  \end{minipage}%
  \caption{Coarsest meshes used in Section \ref{sec:ntest1}. \label{fig:meshes:coarsest}}
\end{figure}

\begin{figure}[!htb]
  \centering
  \begin{minipage}[b]{0.9\textwidth}
    \centering
    \includegraphics[width=0.80\textwidth]{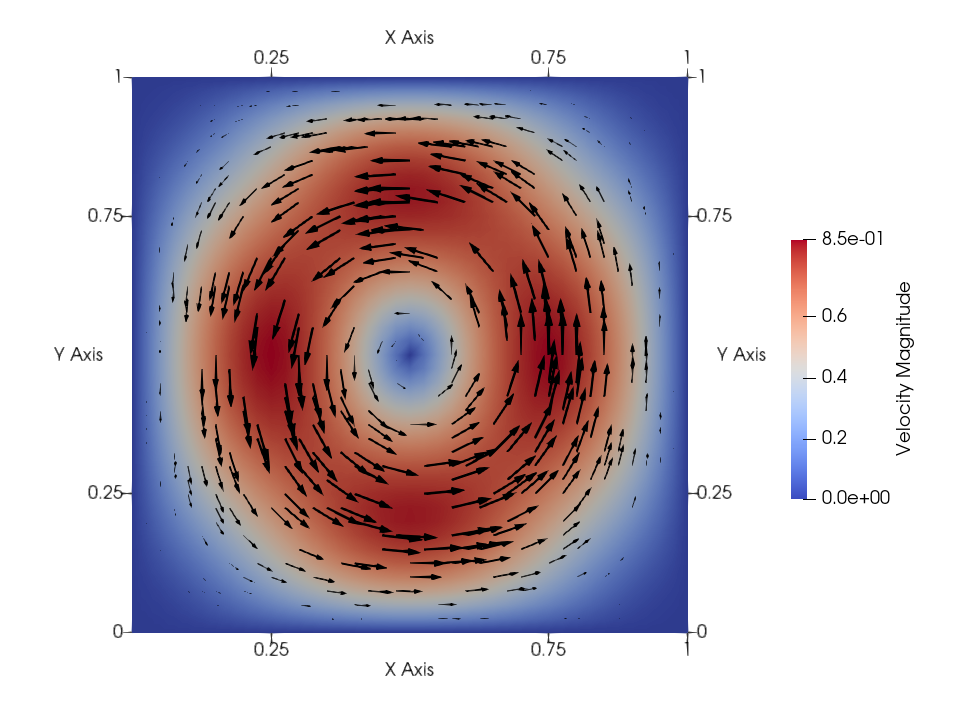} 
  \end{minipage}%
 \caption{Velocity solution at $t=\tF$ of  Section \ref{sec:ntest1}. \label{fig:sol.ntest1}}
\end{figure}

%------------------------------------------------------------------------------%
% Cartesian, Hexagonal and Kershaw numerical test 1
%------------------------------------------------------------------------------%
\renewcommand\floatplace[1]{}
\begin{table}
%\begin{sidewaystable}
  \centering
  \begin{tabular}{ccccc|cccc|}
    \toprule
    $N_{\rm dof}$      &
    $\norm{L^\infty(\bL^2)}{\uline{\be}_h}$& EOC &
    $\norm{\sharp,\Delta t}{\uline{\be}_h}$& EOC &
    $\norm{L^\infty(\bL^2)}{\uline{\be}_h}$   & EOC  &
    $\norm{\sharp,\Delta t}{\uline{\be}_h}$   & EOC  \\
    \midrule
    \multicolumn{1}{c}{}&
    \multicolumn{4}{c|}{Cartesian, $\nu=10^{-2}$}&
    \multicolumn{4}{c|}{Cartesian,  $\nu=10^{-4}$}\\
    \midrule
     385    &   1.88E-01  &      --      & 2.58E-01     &    --     & 3.69E-01       & --     & 2.09E-01        &-- \\
     1620   &   2.64E-02  &      2.82    & 1.04E-01     &   1.30    & 7.21E-02       & 2.34   & 9.86E-02        & 1.08\\
     6640   &   3.24E-03  &      3.04    & 3.74E-02     &   1.48   & 1.30E-02       & 2.48   & 3.87E-02        & 1.36 \\
     26880  &   3.95E-04  &      3.04    & 1.30E-02     &   1.53   & 2.25E-03       & 2.54   & 1.43E-02        & 1.43\\
     108160 &   4.88E-05  &      3.02    & 4.36E-03     &   1.57   & 3.85E-04       & 2.54   & 5.17E-03        & 1.47\\
   \midrule
    \multicolumn{1}{c}{}&
    \multicolumn{4}{c|}{Cartesian, $\nu=10^{-6}$}&
    \multicolumn{4}{c|}{Cartesian,  $\nu=10^{-10}$}\\
    \midrule
   385 & 3.72E-01 & --   &  2.07E-01 & --   & 3.72E-01 & --   & 2.07E-01 & --   \\       
  1620 & 7.45E-02 & 2.31 &  9.82E-02 & 1.07 & 7.45E-02 & 2.31 & 9.82E-02 & 1.07 \\
  6640 & 1.51E-02 & 2.31 &  3.85E-02 & 1.35 & 1.51E-02 & 2.31 & 3.85E-02 & 1.35 \\
 26880 & 3.22E-03 & 2.23 &  1.43E-02 & 1.43 & 3.25E-03 & 2.22 & 1.43E-02 & 1.43 \\
 108160  & 6.85E-04 & 2.23 &  5.18E-03 &  1.46 & 7.16E-04& 2.19 & 5.18E-03 & 1.46 \\
    \bottomrule
   \toprule
     \multicolumn{1}{c}{}&
    \multicolumn{4}{c|}{Hexagonal, $\nu=10^{-2}$}&
    \multicolumn{4}{c|}{Hexagonal,  $\nu=10^{-4}$}\\
    \midrule
   386   & 3.01E-01 & --    & 4.14E-01   & --    & 5.77E-01  & --   &  3.54E-01 & --  \\        
   1436  & 5.68E-02 & 2.42  & 1.88E-01   &1.14   & 1.48E-01  & 1.97 &  1.79E-01 & 0.99 \\
   5560  & 8.67E-03 & 2.70  & 7.17E-02   &1.39   & 2.48E-02  & 2.57 &  7.05E-02 & 1.34 \\
   21872 & 1.14E-03 & 2.93  & 2.54E-02   &1.50   &3.48E-03   & 2.83 &  2.57E-02 & 1.46 \\
  \midrule
    \multicolumn{1}{c}{}&
    \multicolumn{4}{c|}{Hexagonal, $\nu=10^{-6}$}&
    \multicolumn{4}{c|}{Hexagonal,  $\nu=10^{-10}$}\\
    \midrule

   386   & 5.83E-01 & --   & 3.53E-01 & --   & 5.83E-01 & --   & 3.53E-01 & -- \\
   1436  & 1.51E-01 & 1.95 & 1.79E-01 & 0.98 & 1.51E-01 & 1.95 & 1.79E-01 & 0.98 \\
   5560  & 2.64E-02 & 2.51 & 7.03E-02 & 1.34 & 2.64E-02 & 2.51 & 7.03E-02 & 1.34 \\
   21872 & 4.25E-03 & 2.63 & 2.56E-02 & 1.46 & 4.27E-03 & 2.63 & 2.56E-02 & 1.46 \\

  \bottomrule
   \toprule
    \multicolumn{1}{c}{}&
    \multicolumn{4}{c|}{Voronoi, $\nu=10^{-2}$}&
    \multicolumn{4}{c|}{Voronoi,  $\nu=10^{-4}$}\\
    \midrule
     276   & 3.61E-01 & --    &  3.63E-01 & --   & 5.66E-01 & --   &  3.02E-01 & -- \\
     1228  & 5.11E-02 & 3.29  &  1.61E-01 & 1.37 & 1.15E-01 & 2.68 &  1.49E-01 & 1.19 \\
     5136  & 6.19E-03 & 2.95  &  5.99E-02 & 1.38 & 1.94E-02 & 2.49 &  5.67E-02 & 1.35 \\
     21032 & 7.60E-04 & 2.94  &  2.11E-02 & 1.46 & 3.01E-03 & 2.62 &  2.02E-02 & 1.45 \\
    \midrule
    \multicolumn{1}{c}{}&
    \multicolumn{4}{c|}{Voronoi, $\nu=10^{-6}$}&
    \multicolumn{4}{c|}{Voronoi,  $\nu=10^{-10}$}\\
    \midrule
    276   & 5.70E-01 & --    & 3.01E-01 & --   & 5.70E-01  & --   & 3.01E-01 & -- \\
    1228  & 1.17E-01 &  2.66 & 1.48E-01 & 1.19 & 1.17E-01  & 2.66 & 1.48E-01 & 1.19 \\
    5136  & 2.05E-02 &  2.44 & 5.64E-02 & 1.35 & 2.05E-02  & 2.44 & 5.64E-02 & 1.35 \\
    21032 & 3.55E-03 &  2.46 & 2.00E-02 & 1.45 & 3.56E-03  & 2.46 & 2.00E-02 & 1.45 \\
\bottomrule
 \end{tabular}
 \caption{Convergence rates for the numerical test of Section \ref{sec:ntest1} for $k=1$ using the Cartesian, hexagonal and the Voronoi mesh families
for values of $\nu \in \{10^{-2}, 10^{-4},10^{-6},10^{-10}\}$.
\label{tbl:ntest1}}
%\end{sidewaystable}
\end{table}

%---------------------------------------------------------------------------%
\section*{Acknowledgements}

The work of Daniel Castanon Quiroz was partially supported by UNAM-PAPIIT grant IA-101723.

Funded by the European Union  (ERC Synergy, NEMESIS, project number 101115663).
Views and opinions expressed are however those of the authors only and do not necessarily reflect those of the European Union or the European Research Council Executive Agency. Neither the European Union nor the granting authority can be held responsible for them.

%------------------------------------------------------------------------------%
\newpage
\appendix

\section{Local inequalities for piecewise polynomials on a submesh}\label{sec:appx.LE}

\begin{lemma}[Local discrete Lebesgue embedding for piecewise polynomials]\label{lemma:appx.LE.pw}
  Let $T\in\Th$ and, for an integer $l\ge 0$, let $q\in \Poly{l}(\TTs)$.
  Then, for all $(\alpha,\beta)\in [1,+\infty]$, it holds, with the convention that $\frac1\infty \coloneq 0$,
  \begin{equation}
    \|q\|_{L^{\alpha}(T)}
    \simeq h_T^{d\left(\frac1{\alpha}-\frac1{\beta}\right)}  \|q\|_{L^{\beta}(T)},
    \label{eq:rev:disc.emb.pw}
  \end{equation}
  where the hidden constant is independent of  $h_T$, $T$, and $q$, but possibly depends on $l$, $\alpha$, $\beta$, and the mesh regularity parameter. 
\end{lemma}
\begin{proof}
  We first recall the discrete Lebesgue embedding proved in \cite[Lemma 1.25]{Di-Pietro.Droniou:20}, which establishes that  for all $(\alpha,\beta)\in [1,+\infty]$, all $X\in\Th\cup\Ths\cup\Fh\cup\Fhs$, and all $\zeta\in\Poly{l}(X)$
  \begin{equation}
    \|\zeta\|_{L^{\alpha}(X)}
    \simeq |X|^{\frac{1}{\alpha}-\frac{1}{\beta}}  \|\zeta\|_{L^{\beta}(X)},
    \label{eq:rev:disc.emb}
  \end{equation}
  where $|X|$ is the Lebesgue measure {of $X$}, and  the hidden constant is independent of  $X$, and $\zeta$, but possibly depends on $l$, $\alpha$, $\beta$, and the mesh regularity parameter. 
  Then, for $\alpha\in [1,+\infty)$, we obtain that
    \begin{align*}
      \norm{L^\alpha(T)}{q}^\alpha
      =
      \sum_{\tau\in \TTs}\norm{L^\alpha(\tau)}{q}^\alpha
      &\simeq
      \sum_{\tau\in \TTs}|\tau|^{\alpha\left(\frac{1}{\alpha}-\frac{1}{\beta}\right)}\norm{L^\beta(\tau)}{q}^\alpha\\
      &\simeq
      \sum_{\tau\in \TTs}|\tau|^{\alpha\left(\frac{1}{\alpha}-\frac{1}{\beta}\right)}\norm{L^\beta(T)}{q}^\alpha\\
      &\simeq
      h_T^{\alpha d\left(\frac{1}{\alpha}-\frac{1}{\beta}\right)}\norm{L^\beta(T)}{q}^\alpha,
    \end{align*}
    where in the second step we have used \eqref{eq:rev:disc.emb} with $X=\tau$, 
    in the third step the fact that $\tau\subset T$,
    and in the  fourth step the fact that $|\tau| \simeq h_\tau^d {\simeq h_T^d}$  by mesh regularity along with
     \eqref{ineq:card.IT.F} for $\TTs$.
    Raising to the power of $\frac{1}{\alpha}$  the above  inequality, we obtain \eqref{eq:rev:disc.emb.pw}.
    If $\alpha =+\infty$, we use  similar steps as before to  obtain that
    \begin{align*}
      \norm{L^\infty(T)}{q}
      &=
      \sup_{\tau\in \TTs}\norm{L^\infty(\tau)}{q}
      \lesssim
      \sup_{\tau\in \TTs}|\tau|^{-\frac{1}{\beta}}\norm{L^\beta(T)}{q}
      \lesssim
      h_T^{-\frac{d}{\beta}}\norm{L^\beta(T)}{q}.
      \qedhere
    \end{align*}
\end{proof}

\begin{lemma}[Local inverse inequalities for piecewise polynomials]
  Let $T \in \Th$ and, for an integer $l \ge 0$, let $q \in \Poly{l}(\TTs)$.
  Then, for all $p \in [1,+\infty]$, it holds
  \begin{equation}\label{eq:inverse}
    \norm{L^p(T)}{q}
    \lesssim h_T^{-1}\norm{\bL^p(T)}{\GRAD q},
  \end{equation}
  where we remind the reader that here $\GRAD$ stands for the piecewise gradient on $\TTs$.
\end{lemma}

\begin{proof}
  For $p \in \lbrack 1, +\infty)$, we write
  \[
  \norm{L^p(T)}{q}^p
  = \sum_{\tau \in \TTs} \norm{L^p(\tau)}{q}^p
  \lesssim \sum_{\tau \in \TTs} h_\tau^{-p} \norm{\bL^p(\tau)}{\GRAD q}^p
  \lesssim h_T^{-p}\norm{\bL^p(T)}{\GRAD q}^p,
  \]
  where we have used a standard discrete inverse inequality on simplices (cf., e.g., \cite[]{Brenner.Scott:08}) in the second passage and the fact that $h_\tau \simeq h_T$ by mesh regularity to conclude.
  If $p = +\infty$, the proof is similar with sums over $\tau \in \TTs$ replaced by maximums over $\tau \in \TTs$.
\end{proof}

  \begin{lemma}[Local trace inequality for piecewise polynomials]
    Let $T \in \Th$, $F \in \Fh[T]$, and, for an integer $l \ge 0$, let $q \in \Poly{l}(\TTs)$.
    Then, for all $p \in [1,+\infty]$, it holds
    \begin{equation}\label{eq:discrete.trace}
      \norm{L^p(F)}{q} \lesssim h_F^{-\frac1p} \norm{L^p(T)}{q}.
    \end{equation}
  \end{lemma}

  \begin{proof}
    Recalling that $\Fhs[F]$ collects the simplicial faces contained in $F$ and denoting, for all $\sigma \in \Fhs[F]$, by $\tau_\sigma \in \Ths[T]$ the simplex to which $F$ belongs, we write
    \[
    \norm{L^p(F)}{q}^p
    = \sum_{\sigma \in \Fhs[F]} \norm{L^p(\sigma)}{q}^p
    \lesssim \sum_{\sigma \in \Fhs[F]} h_\sigma^{-1} \norm{L^p(\tau_\sigma)}{q}^p
    \lesssim h_T^{-1}\norm{L^p(T)}{q}^p,
    \]
    where we have used a standard discrete trace inequality on simplices in the second step and the fact that $h_\sigma^{-1} \lesssim h_T^{-1}$, $\tau_\sigma \subset T$ for all $\sigma \in \Fhs[F]$, and $\card(\Fhs[F]) \lesssim 1$ to conclude.
    Taking the $q$-th root yields \eqref{eq:discrete.trace}.
  \end{proof}

%------------------------------------------------------------------------------%

\printbibliography

\end{document}